\documentclass[a4paper,10pt,reqno]{amsart}

\usepackage[pdftex]{graphicx}
\usepackage{amsfonts}
\usepackage{amssymb}
\usepackage{amsrefs} 
\usepackage{mathtools} 
\usepackage{mathrsfs}
\usepackage{stmaryrd}
\usepackage{amsmath}
\usepackage{amsthm}
\usepackage[shortlabels]{enumitem}
\usepackage{nomencl}
\usepackage[bookmarks=true]{hyperref}   
\usepackage{esint}
\usepackage{dsfont}
\usepackage{upgreek}
\usepackage{xcolor}
\usepackage{slashed}
\usepackage{scalerel}
\usepackage{comment}
\usepackage[utf8]{inputenc}
\usepackage{todonotes}
\usepackage{faktor}
\usepackage{tikz-cd}

\hypersetup{
    colorlinks,
    linkcolor={red!50!black},
    citecolor={blue!50!black},
    urlcolor={blue!80!black}
}

\newcounter{abb}
\newcommand{\Fig}[1]{\refstepcounter{abb} \centering \footnotesize \emph{Fig.~\arabic{abb}:} #1}

\def\colour{\colour}

\newcommand{\sym}{\upsigma}
\newcommand{\spec}{\mathrm{spec}}		
\newcommand{\res}{\mathrm{res}}		
\newcommand{\rest}[1]{{{\lvert_{}}_{}}_{#1}}
\newcommand{\R}{\mathbb{R}}
\newcommand{\C}{\mathbb{C}}
\newcommand{\N}{\mathbb{N}}
\newcommand{\In}{\mathbb{Z}}
\newcommand{\cbrac}[1]{\left(#1\right)}
\newcommand{\bbrac}[1]{\left[#1\right]}
\newcommand{\dbrac}[1]{\left\{#1\right\}}

\newcommand{\modulus}[1]{|#1|}

\newcommand{\norm}[1]{\| #1 \|}			

\newcommand{\set}[1]{\dbrac{#1}}
 
\newcommand{\Lp}[2][{}]{{\rm L}^{#2}_{\rm #1}}		
\newcommand{\Ck}[2][{}]{{\rm C}^{#2}_{\rm #1}}		
\newcommand{\Hard}[2][{}]{{\rm H}^{#2}_{\rm #1}}		
\newcommand{\SobH}[2][{}]{\Hard[#1]{#2}}

\newcommand{\checkH}{\check{\mathrm{H}}}	
\newcommand{\hatH}{\hat{\mathrm{H}}}		

\newcommand{\embed}{\hookrightarrow}		
\DeclareMathOperator{\id}{id}

\newcommand{\nul}{\mathrm{ker}}
\newcommand{\ran}{\mathrm{ran}}
\newcommand{\dom}{\mathrm{dom}}
\newcommand{\coker}{\mathrm{coker}}

\DeclareMathOperator{\indx}{index}

\newcommand{\intersect}{\cap} 
\newcommand{\close}[1]{\overline{#1}} 
\newcommand{\union}{\cup} 

\DeclareMathOperator{\spt}{spt}		
\DeclareMathOperator{\sgn}{sgn}		

\newcommand{\inprod}[1]{\left\langle #1 \right\rangle}	
\newcommand{\ext}{\mathcal{E}}

\newcommand{\ad}{\ast}

\newcommand{\cH}{\mathcal{H}}
\newcommand{\cB}{\mathcal{B}}

\DeclareMathOperator{\graph}{graph}

\newcommand{\comp}{\, \circ\, }

\newcommand{\dagstar}{\ast}

\newtagform{simple}{}{}{}
\newcommand{\rvline}{\hspace*{-\arraycolsep}\vline\hspace*{-\arraycolsep}}
\newcommand{\PI}{\pi^{\frac32}}
\newcommand{\DD}{\mathscr{D}}

\newcommand{\Ind}{\indx}

\renewcommand{\Re}{\mathrm{Re}\ }

\renewcommand{\epsilon}{\varepsilon}
\renewcommand{\implies}{\Rightarrow}
\renewcommand{\imath}{\mathrm{i}}
\renewcommand{\SS}{S^{\frac32}}

 \newtheorem{theorem}{Theorem}[section]
 \newtheorem{lem}[theorem]{Lemma}
 \newtheorem{cor}[theorem]{Corollary}
 \newtheorem{proposition}[theorem]{Proposition}

 \theoremstyle{definition}
 \newtheorem{definition}[theorem]{Definition}
 
 \newtheorem{remark}[theorem]{Remark}

\author{Christian Bär}
\author{Lashi Bandara}
\title[BVPs for general first-order elliptic differential operators]
{Boundary value problems for general first-order elliptic differential operators}
\date{\today}

\address{Christian Bär, 
Institut für Mathematik,
Universität Potsdam, 
D-14476, Potsdam, Germany
}
\urladdr{\href{https://www.math.uni-potsdam.de/baer}{https://www.math.uni-potsdam.de/baer}}
\email{\href{mailto:cbaer@uni-potsdam.de}{cbaer@uni-potsdam.de}}

\address{Lashi Bandara, 
Institut für Mathematik,
Universität Potsdam, 
D-14476, Potsdam, Germany
}
\urladdr{\href{http://www.math.uni-potsdam.de/~bandara}{http://www.math.uni-potsdam.de/~bandara}}
\email{\href{mailto:lashi.bandara@uni-potsdam.de}{lashi.bandara@uni-potsdam.de}}

\dedicatory{Dedicated to Werner Ballmann on the occasion of his 68$^\mathrm{th}$ birthday}


\setlist{noitemsep, topsep=-\parskip,leftmargin=*, listparindent=0pt}

\begin{document}

\maketitle

\begin{abstract}
We study boundary value problems for first-order elliptic differential operators on manifolds with compact boundary.
The adapted boundary operator need not be selfadjoint and the boundary condition need not be pseudo-local.

We show the equivalence of various characterisations of elliptic boundary conditions and demonstrate how the boundary conditions traditionally considered in the literature fit in our framework.
The regularity of the solutions up to the boundary is proven.
We show that imposing elliptic boundary conditions yields a Fredholm operator if the manifold is compact.
We provide examples which are conveniently treated by our methods.

Link to video abstract: \href{https://vimeo.com/523211595}{https://vimeo.com/523211595}
\end{abstract} 

\tableofcontents
\setlength{\parskip}{.3\baselineskip}

\section{Introduction}

First-order elliptic operators are abundant in the context of physics and geometry.
In particular, index theory is a topic of significance  where first-order elliptic differential operators play a key role.
While these questions were firstly focused on closed manifolds, it was apparent that the case of manifolds with boundary is useful, particularly for applications.
Traditionally, when dealing with second-order elliptic operators, one imposes local boundary conditions which are conditions on the function and its derivatives at all points of the boundary.
Dirichlet- and Neumann boundary conditions are prominent examples.
Local elliptic boundary conditions can also be defined for operators of arbitrary order and are known as Lopatinsky-Schapiro conditions.
For first-order operators however they cannot always be imposed; depending on the operator and the underlying manifold with boundary there can be topological obstructions against their existence.

Atiyah, Patodi, and Singer overcame this problem in their pioneering work \cite{APS-Ann,APS1,APS2,APS3} by introducing \emph{non-local} boundary conditions.
These boundary conditions were phrased in terms of the spectral decomposition of an operator on the boundary.
Since the work of Atiyah-Patodi-Singer, a plethora of results have been obtained in different settings under various assumptions. 
Much of this work has focused on the so-called pseudo-local case, where the boundary condition is obtained as the range of a pseudodifferential projector on the boundary. 
To name a few relevant references, see \cite{BBC, BLZ, B, BL2001, G96, Melrose, RS, S01, S04} by Ballmann, Booß-Bavnbek, Boutet de Monvel, Brüning, Carron, Grubb, Lesch, Melrose, Rempel, Schulze, and Zhu. 
There has also been effort to understand such questions in non-smooth settings as well as for general elliptic operators, for instance \cite{KM1, KM2} by  Krainer and Mendoza.

In \cite{BB12} by Bär and Ballmann, a framework is constructed to study boundary value problems for first-order elliptic operators which induce a selfadjoint operator on the boundary. 
Although this is a restriction, they are able to account for a large class of operators, including all Dirac type operators.
A tremendous boon of their analysis is that they are able to capture all possible boundary conditions. 
In particular, they are able to go beyond those that are pseudo-local.
Such conditions are significant and arise naturally. 
For instance, the so-called \emph{matching} boundary conditions are not pseudo-local but are extremely useful to study relative index theory.

In the present paper, we show how to study boundary value problems for a general first-order elliptic differential operator, i.e.\ we drop the assumption that the adapted boundary operator be selfadjoint.
This is necessary for the study of more exotic, but geometrically natural operators such at the Rarita-Schwinger operator on $3/2$-spinors. 

More generally,  we address aspects of the research program sketched out in \cite{BL2009} by  Booß-Bavnbek and Lesch to close the gap in the understanding of boundary value problems between the geometric Dirac type operator situation and general first-order elliptic case.
We define the boundary trace map on the maximal domain of the operator and understand its topology given by the spectral subspaces of the adapted operator on the boundary. 
The minimal domain is the kernel of the boundary trace map and hence, we are able to understand the totality of possible boundary conditions. 
In particular, we introduce a very general notion of an elliptic boundary condition through a host of different characterisations.
We also include an important characterisation of ellipticity of a boundary condition in the language of Fredholm pairs, as initially demonstrated by \cite{BS} by Braverman and Shi in the context of selfadjoint boundary operators.
Under a coercivity assumption on the operator at infinity, we  demonstrate that elliptic boundary conditions are Fredholm. 
Moreover, we also understand the  boundary regularity of solutions.

We emphasise that although we are guided by \cite{BB12}, the methods and techniques used here are vastly different. 
The analysis in  \cite{BB12} is carried out in the spirit of the Fourier series through concentrating the analysis to each eigenspace of the boundary adapted operator and reconstructing the global picture through orthogonality.
In our more general case, the adapted operators  may not be realised as a selfadjoint operator.
There, the eigenspaces alone may not sum to the total space, and to make matters worse, there may be non-orthogonality of the generalised eigenspaces.

Therefore, the basis of our approach is to understand the adapted operator on the boundary. 
We show that up to the addition of a real constant, such an operator is an invertible bisectorial operator.
This observation is of paramount importance and is stronger than the so-called parametric ellipticity along the imaginary axis that was observed in \cite{BL2009}.
Using the work of Grubb in \cite{Grubb}, we obtain spectral projectors to the generalised eigenspaces of the boundary operator, as pseudo-differential operators of order zero. 
This allows us to build the modulus operator as an invertible sectorial operator. 

The engine of our work is the $H^\infty$-functional calculus and its quadratic estimates characterisation, pioneered by McIntosh in \cite{Mc86}.
We show that our model operator enjoys this functional calculus.
Then, this allows us to access the so-called \emph{maximal regularity} for sectorial operators, which is obtained through the $H^\infty$-functional calculus and underpins much of the analysis.
These methods are used in a different context of non-smooth coefficient first-order factorisations of divergence form equations, see e.g.\ \cite{AAH08,AAMc10, AA11,AA12} by Auscher, Axelsson (Rosén), Hofmann, and McIntosh. 

The structure of the paper is as follows. 
In Section~\ref{Sec:Setup}, we describe the key assumptions as well as the results, so that the results may be accessed with the minimal amount of technicalities.
Section~\ref{Sec:Examples} illustrates examples of some adapted operators to the boundary which exhibits phenomena that cannot be captured by the framework in \cite{BB12}. 
In particular, this section shows through explicit calculation that the adapted boundary operator for the Rarita-Schwinger operator cannot be chosen in such a way that it becomes selfadjoint.
In Section~\ref{Sec:OpTheory}, we study the operator theory, particularly the spectral theory,  of the boundary adapted operator.
There, we demonstrate that the boundary adapted operator is, up to the addition of a real perturbation, an invertible bisectorial operator. 
This simple but important observation allows us in Section~\ref{Sec:Hinfty} to show the existence of the $H^\infty$-functional calculus for the modulus  operator.
Moreover, we also demonstrate the connections between fractional powers of the modulus to fractional Sobolev scales.
Then, in Section~\ref{Sec:ModelOp}, we consider the reduction to the cylinder, and study the properties of the so-called model operator. 
In this section, we also demonstrate how maximal regularity is used in order to obtain higher regularity results on the cylinder.
Following this, in Section~\ref{Sec:MaximalDom}, we return to the general manifold situation and study the maximal domain of the operator.
We conclude the paper with Section~\ref{Sec:BVPs} where we consider boundary value problems.
An appendix~\ref{Sec:Appendix} is included that captures some of the abstract tools needed in various parts of the paper.

\textbf{Acknowledgements.}
We thank Andreas Ros\'en for useful discussions. 
Penelope Gehring also deserves acknowledgement for carefully reading through the manuscript and providing feedback.
The second author was supported by SPP2026 from the German Research Foundation (DFG).

\section{Setup and results}
\label{Sec:Setup}

In this section, we list background assumptions, necessary definitions, as well as the results which we obtain in this paper. 
The proofs of many of these results are technical in nature and we defer these to later points of the paper. 
Here, we present the necessary details to enable a working knowledge of the framework and results so that it can be readily applied.

\subsection{Notation}
Let $\cB$ be a Banach space and  $T: \dom(T)\subset\cB \to \cB$ an unbounded operator.
The objects $\dom(T)$, $\ran(T)$, $\nul(T)$, $\spec(T)$, and $\res(T)$ denote the domain, range, null space, spectrum and resolvent set of $T$, respectively.
The operator $T$ is said to be \emph{invertible} if it is injective, $\ran(T)$ is dense, and $T^{-1}$ is bounded.
In this case, we often write $T^{-1}$ to denote the unique bounded extension of $T^{-1}$ on $\ran(T)$ to all of $\cB$. 
The graph norm of $T$ is $\norm{\cdot}^2_T := \norm{\cdot}^2 + \norm{T\cdot}^2$ on $\dom(T)$.

For $M$ a smooth manifold with smooth boundary $\partial M$, by $K \Subset M$ we mean that $K \subset M$, the interior $\mathring{K} \neq \emptyset$ and that $K$ is compact.
On $(E,h^E) \to M$ a Hermitian vector bundle, the support of a section $u$ will denoted by $\spt u$.
We equip $M$ with a smooth measure $\mu$.
For $p \in [1,\infty)$,  we denote measurable sections $u$ such that $\int_{M} \modulus{u}_{h^E}^p\ d\mu < \infty$ by  $\Lp{p}(M;E)$.
This is a Banach space with norm $\norm{u}_{\Lp{p}} = (\int_{M} \modulus{u}^p_{h^E}\ d\mu)^{1/p}$. 
Since $h^E$ is Hermitian, the norm in the case of $p = 2$ polarises and hence $\Lp{2}(M;E)$ is a Hilbert space with inner product 
$$ 
\inprod{u,v} = \int_{M} h^E(u,v)\ d\mu.
$$
The space $\Lp{\infty}(M;E)$ is defined similarly, namely as the space consisting of measurable sections $u$, each of which for there exists a $C > 0$ with $|u(x)|_{h^E(x)} \leq C$ for almost every $x \in M$.
The $\Lp{\infty}$-norm is then the infimum over such constants $C$.

The following table lists the notations for significant function spaces we will require in this paper:
\renewcommand{\arraystretch}{1.2}
\begin{table}[h!]
\begin{centre}
\begin{tabular}{|l|l|}
\hline
$\Ck{k}(M;E)$		& $k$-times continuously differentiable sections \\ 
\hline
$\Ck[c]{k}(M;E)$ 		& sections $u \in \Ck{k}(M;E)$ such that $\spt u$ compact in $M$ \\
\hline
$\Ck[cc]{k}(M;E)$		& sections $u \in \Ck[c]{k}(M;E)$ such that $\spt u \subset M \setminus \partial M$ \\
\hline
$\SobH{k}(K;E)$ & $\Lp{2}$-Sobolev space of order $k$ on the compact subset $K\subset M$ \\
\hline
$\SobH[loc]{k}(M;E)$	& $u \in \Lp{2}(M;E)$ s.~ t.\ $u\rest{K} \in \SobH{k}(K;E)$  for every compact $K \subset M$ \\
\hline
\end{tabular}
\end{centre}
\end{table} 
\renewcommand{\arraystretch}{1}

We emphasise that for $\Ck[c]{k}(M;E)$ as well as $\SobH[loc]{k}(M;E)$, the support of sections may touch the boundary $\partial M$.

For a first-order differential operator $D: \Ck{\infty}(M;E) \to \Ck{\infty}(M;F)$, where $(F,h^F) \to M$ is another Hermitian bundle, we denote the \emph{principal symbol} by  $\sym_D(x,\xi)$.
In our convention, this symbol is given by the expression $[D, f](x)$ where $df(x) = \xi$, $[\cdot,\cdot]$ is the commutator, and $f$ acts by multiplication.

There exists a formal adjoint $D^\dagger$ of $D$ on the domains $\Ck[cc]{\infty}(M;F)$ and $\Ck[cc]{\infty}(M;E)$, respectively. 
Denoting these operators $D_{cc}$ and $D^\dagger_{cc}$, we can obtain the maximal operators as: 
$$ D_{\max} = (D^\dagger_{cc})^\ad \quad \text{and}\quad D^\dagger_{\max} = (D_{cc})^\ad,$$
where $\ad$ denotes the adjoint of an unbounded operator in the Hilbert space $\Lp{2}(M;E)$ or $\Lp{2}(M;F)$, respectively.
The  operator $D$ is said to be  \emph{complete} if the subspace of compactly supported sections in $\dom(D_{\max})$ is dense in $\dom(D_{\max})$ in the graph norm  of $D$ in $\Lp{2}(M;E)$.

\subsection{The framework and results}
\label{Sec:SR}

The following are the background assumptions we use throughout this article.
These are to be assumed throughout the paper unless specified otherwise.
\begin{enumerate}[(S1)] 
\label{StdSetup}
\item $M$ is a smooth manifold with smooth compact boundary $\Sigma = \partial M$;
\item $\vec{T}$ is a smooth interior vectorfield along $\partial M$;
\item $\mu$ is a smooth volume measure on $M$ and $\nu$  is the smooth volume measure on $\Sigma$, induced by $\mu$ and $\vec{T}$; 
\item $(E,h^E),\ (F,h^F) \to M$ are Hermitian vector bundles over $M$;
\item $D$ is a first-order elliptic differential operator from $E$ to $F$; 
\item $D$ and $D^\dagger$ are complete.  
\end{enumerate}

The interior co-vectorfield to $\vec{T}$ is given by $\tau$.
This is the $1$-form that satisfies $\tau(x) (T(x)) = 1$ and $\tau(x)\rest{T_x\Sigma} = 0$.

\begin{definition}[Adapted boundary operator]
An operator $A$ is   said to be an \emph{adapted boundary operator} for $D$ if its principal symbols is given by 
$$\sym_{A}(x,\xi) = \sym_{D}(x,\tau(x))^{-1} \comp \sym_{D}(x,\xi)$$
for all $x\in\partial M$ and $\xi\in T_x^*\partial M$.
By $\tilde{A}$, we denote an adapted boundary operator for $D^\dagger$ which satisfies the definition with $D$ replaced by $D^\dagger$.
\end{definition}

Such operators always exist and are elliptic differential operators of order $1$.
They are unique up to an operator of order zero. 
In what is to follow, we will always fix such an operator (unless specified otherwise).

\begin{definition}[Admissible spectral cut]
\label{Def:SpecCut}
If $r \in \R$ is a number such that the vertical line $l_r = \set{ \zeta \in \C: \Re \zeta = r}$ does \emph{not} intersect the spectrum of $A$, we say that $r$ is an \emph{admissible  cut}.
\end{definition} 

Any adapted operator $A$, being an elliptic operator on a compact manifold without boundary, has discrete spectrum.
Hence all real numbers but a countable set of exceptions are admissible spectral cuts.
Given such a cut $r$, we define the operator $A_r := A - r$ which is then invertible.

We are able to obtain spectral projectors 
$$\chi^{\pm}(A_r): \Lp{2}(\Sigma;E) \to \Lp{2}(\Sigma;E)$$ 
to the spectral subspaces corresponding to eigenvalues with positive and negative real parts respectively.
These are pseudo-differential operators of order zero  and therefore we obtain 
$$\SobH[\pm]{s}(\Sigma;E) = \chi^{\pm}(A_r)\SobH{s}(\Sigma;E)$$ 
as closed subspaces of the Sobolev spaces $\SobH{s}(\Sigma;E)$ for all $s \in \R$.
Note that these projectors are, in general, non-orthogonal.

The fundamental space for the analysis of boundary value problems is the following so-called \emph{check} space corresponding to $A$ 
$$ \checkH(A) := \SobH[-]{\frac{1}{2}}(A_r) \oplus \SobH[+]{-\frac{1}{2}}(A_r),$$
with norm
$$ \norm{u}_{\checkH(A)}^2 = \norm{\chi^-(A_r)u}_{\SobH{\frac{1}{2}}}^2 + \norm{\chi^+(A_r)u}_{\SobH{-\frac{1}{2}}}^2.$$
We write the left hand side of these definitions independent of the cut parameter $r$ since the subspace $\checkH(A)$ is independent of the cut parameter $r$ and any two norms corresponding to two admissible spectral cuts are comparable.
In practice,  it is often useful to choose a cut that is convenient to the problem at hand.

We set $\hatH(A) := \checkH(-A)$ and define $\sym_0: \Ck{\infty}(\Sigma;E) \to \Ck{\infty}(\Sigma;F)$ by $\sym_0(x):=\sigma_D(x,\tau(x))$. 
The  homomorphism field $(\sym_0^{-1})^\ast$ induces an isomorphism between $\hatH(A^\dagstar)$ and $\checkH(\tilde{A})$ and $\beta(u,v) := - \inprod{\sym_0 u, v}_{\Lp{2}(\Sigma;F)}$ for $u \in \Ck[c]{\infty}(\Sigma;E)$ and $v \in \Ck[c]{\infty}(\Sigma;F)$ extends to a perfect pairing between $\checkH(A)$ and $\checkH(\tilde{A})$. For convenience, we will denote this extension simply as $-\inprod{\sym_0 u,v}$. 

\begin{theorem} 
\label{Thm:Ell}
The following hold:
\begin{enumerate}[(i)]
\item 
\label{Thm:Ell1}
The space $\Ck[c]{\infty}(M;E)$ is dense in $\dom(D_{\max})$ with respect to the corresponding graph norm;
\item 
\label{Thm:Ell2}
The trace maps $\Ck[c]{\infty}(M;E) \to \Ck{\infty}(\Sigma;E)$ given by $u \mapsto u\rest{\Sigma}$ extend uniquely to surjective bounded linear maps $\dom( D_{\max}) \to \checkH(A)$; 
\item 
\label{Thm:Ell3}
The spaces
\begin{align} 
	\dom(D_{\max}) \intersect \SobH[loc]{1}(M;E) &= \set{ u \in \dom(D_{\max}): u\rest{\Sigma} \in \SobH{\frac{1}{2}}(\Sigma;E)}, \notag\\ 
\label{Eq:MaxDom}
\end{align}
\item 
\label{Thm:Ell4}
For all $u \in \dom(D_{\max})$ and $v \in \dom((D^\dagger)_{\max})$,
\begin{equation}
\label{Eq:MaxDInt}
\inprod{D_{\max} u, v}_{\Lp{2}(M;F)} - \inprod{u, (D^\dagger)_{\max}v}_{\Lp{2}(M;E)} = -\inprod{\sym_0 u\rest{\Sigma}, v\rest{\Sigma}}.
\end{equation}
\end{enumerate}
The corresponding statements hold for $D^\dagger$ on replacing $E$  by $F$. 
\end{theorem}

Higher elliptic regularity is then described by the following theorem:
\begin{theorem}
\label{Thm:HighReg}
The following holds:
\begin{multline*}
\dom(D_{\max}) \cap \SobH[loc]{k+1}(M;E) \\
	=  \set{u \in \dom(D_{\max}): Du \in \SobH[loc]{k}(M;F)\ \text{and}\ \chi^{+}(A_r)(u\rest{\Sigma}) \in \SobH{k + \frac{1}{2}}(\Sigma;E)}.
\end{multline*}
\end{theorem}

As a consequence of Theorem~\ref{Thm:Ell}~\ref{Thm:Ell2}, the following is a notion of boundary condition.
\begin{definition}[Boundary condition and the associated operator]
\label{Def:BC}
A \emph{boundary condition} for $D$ is a \emph{closed}  linear subspace $B \subset \checkH(A)$. 
The domains of the operators are
\begin{align*}
\dom(D_{B,\max}) &= \set{ u \in \dom(D_{\max}): u\rest{\Sigma} \in B},\\
\dom(D_{B}) &= \set{u \in \dom(D_{\max}) \intersect \SobH[loc]{1}(\Sigma;E): u\rest{\Sigma} \in B},
\end{align*}
and similarly for the formal adjoint $D^\dagger$ with $A$ replaced by $\tilde{A}$.
\end{definition}

For any boundary condition $D_B$, the operator $D_B$ is a closed map between $D_{cc}$ and $D_{\max}$.
Given any closed extension $D_c$ of $D_{cc}$, there exists a boundary condition $B$, given by $B = \set{ u\rest{\Sigma}: u \in \dom(D_c)}$, so that $D_c = D_{B,\max}$.
Moreover, a boundary condition $B$ satisfies  $B \subset \SobH{\frac{1}{2}}(\Sigma;E)$ if and only if $D_{B} = D_{B,\max}$.
See Proposition~\ref{Prop:ClosedExt} for an elaborate description of these statements as well as their proofs.

Given a boundary condition $B$, the associated \emph{adjoint boundary condition} is denoted by $B^\ad$ for the operator $D_{B,\max}^\ad$.
This space is given by:
$$B^\ad := \set{v \in \checkH(\tilde{A}): \inprod{ \sym_0 u, v} = 0\quad \forall u \in B},$$
and this is a closed subspace of $\checkH(\tilde{A})$.
As aforementioned, the symbol $\sym_0^\ast: \checkH(\tilde{A}) \to \hatH(A^\dagstar)$ is an isomorphism and we have that $\sym_0^\ast B^\ad \subset \hatH(A^\dagstar)$ is closed.
In applications, it is easier to work with this latter space so as to only have to consider the operator $A$ rather than both $A$ and $\tilde{A}$  
simultaneously. 

The following decomposition of a boundary condition is useful in characterising elliptic boundary conditions.

\begin{definition}[Elliptically decomposed boundary condition]
\label{Def:EllBC}
Let $B \subset \SobH{\frac{1}{2}}(\Sigma;E)$ be a boundary condition.
Let $r \in \R$ be an admissible spectral cut and suppose:
\begin{enumerate}[(i)]
\item  \label{Def:EllBC:First} \label{Def:EllBC:MutualComp}
$W_{\pm}$, $V_{\pm}$ are mutually complementary subspaces of $\Lp{2}(\Sigma;E)$ such that 
	$$V_\pm \oplus W_\pm = \chi^{\pm}(A_r) \Lp{2}(\Sigma;E)$$
\item  \label{Def:EllBC:FiniteDim}
$W_{\pm}$ are finite dimensional with $W_{\pm}, W_{\pm}^\ast \subset \SobH{\frac{1}{2}}(\Sigma;E)$, and 
\item  \label{Def:EllBC:Last} \label{Def:EllBC:Bddmap}
there exists a bounded linear map $g: V_- \to V_+$ with  $g(V_-^{\frac{1}{2}}) \subset V_+^{\frac{1}{2}}$ and $g^\ast((V_{+}^\ast)^\frac{1}{2}) \subset (V_-^\ast)^\frac{1}{2}$  
such that 
	$$B = W_+ \oplus \set{v + gv: v \in V_{-}^{\frac{1}{2}}}.$$
\end{enumerate}
Then, we say that $B$ can be \emph{elliptically decomposed with respect to~$r$}. 
\end{definition}
Here, given a subspace $V\subset\Lp{2}(\Sigma;E)$, we put $V^s:=V\cap \SobH{s}(\Sigma;E)$ for $s\in\R$.

\begin{remark}\label{Rem:*Spaces}
The spaces $W_\pm^\ast$ and $V_{\pm}^\ast$ are defined by $Q_{\pm}^\ast\Lp{2}(\Sigma;E)$ and $P_{\pm}^\ast \Lp{2}(\Sigma;E)$, where $Q_\pm$ and $P_\pm$ are the unique projectors with respect to the splitting of the space $\Lp{2}(\Sigma;E) = V_- \oplus W_- \oplus V_+ \oplus W_+$ so that $W_{\pm} = Q_{\pm}^\ast\Lp{2}(\Sigma;E)$ and $V_{\pm} = P_{\pm}^\ast \Lp{2}(\Sigma;E)$.
In particular, $\dim(W_\pm)=\dim(W_\pm^*)$.
\end{remark}

The notion of \emph{Fredholm pairs} is also useful in characterising elliptic boundary conditions.
Recall that a pair $(X, Y)$ of closed subspaces of a Hilbert space $Z$ is called a Fredholm pair if  $X + Y$ is closed in $Z$ and $X \cap Y$ and $\coker(X,Y) := Z/(X + Y)$ are finite dimensional.
The index of a Fredholm pair $(X,Y)$ is 
$$ \indx(X,Y) = \dim(X \cap Y) - \dim \coker(X,Y) \in \In.$$

\begin{definition}[Fredholm pair decomposition]
\label{Def:FP}
Let $B \subset \SobH{\frac12}(\Sigma;E)$ and let $r \in \R$ be an admissible spectral cut. 

Suppose: 
\begin{enumerate}[(i)] 
\item 
\label{Def:FP1} 
$B$ is closed subspace of $\SobH{\frac12}(\Sigma;E)$,
\item 
\label{Def:FP2}
$(\chi^+(A_r)\SobH{\frac12}(\Sigma;E), B)$ and $(\chi^-(A_r^\ast)\SobH{\frac12}(\Sigma;E), B^{\perp, \SobH{\frac12}(\Sigma;E)})$ are Fredholm pairs in $\SobH{\frac12}(\Sigma;E)$ (where $B^{\perp, \SobH{\frac12}(\Sigma;E)}$ denotes the annihilator of $B$ in $\SobH{\frac12}(\Sigma;E)$), and 
\item 
\label{Def:FP3} $\indx(\chi^+(A_r)\SobH{\frac12}(\Sigma;E), B) = -\indx(\chi^-(A_r^\ast)\SobH{\frac12}(\Sigma;E), B^{\perp, \SobH{\frac12}(\Sigma;E)}).$
\end{enumerate}
Then, we say that $B$ is \emph{Fredholm-pair decomposed with respect to $r$}.
\end{definition}

The following is an important theorem that provides useful criteria to determine elliptic boundary conditions. 
This theorem also illustrates that this notion of ellipticity agrees with previous notions in the literature.

\begin{theorem}
\label{Thm:EllEquiv}
Let $B \subset \SobH{\frac{1}{2}}(\Sigma;E)$ be a subspace. 
Then the following are equivalent:
\begin{enumerate}[(i)]
\item \label{Thm:EllEquiv:Ellbdy} 
$\dom(D_{B,\max}) \subset \SobH[loc]{1}(M;E)\quad\text{and}\quad \dom(D_{B^\ad,\max}^\dagger) \subset \SobH[loc]{1}(M;F)$.
\item \label{Thm:EllEquiv:Closed}
$B$ is a closed subset of $\checkH(A)$ and $B^{\ad} \subset \SobH{\frac{1}{2}}(\Sigma;F)$,
\item  \label{Thm:EllEquiv:Admiss}
with respect to every admissible spectral cut $r \in \R$, $B$ can be elliptically decomposed,
\item  \label{Thm:EllEquiv:Admiss2}
with respect to some admissible spectral cut $r \in \R$, $B$ can be elliptically decomposed,
\item \label{Thm:EllEquiv:Admiss3}
with respect to every admissible spectral cut $r \in \R$, $B$ can be Fredholm pair decomposed,
\item \label{Thm:EllEquiv:Admiss4}
with respect to some admissible spectral cut $r \in \R$, $B$ can be Fredholm pair decomposed.
\end{enumerate}
Moreover, whenever one of the equivalent statements \ref{Thm:EllEquiv:Closed}-\ref{Thm:EllEquiv:Admiss4} is satisfied, 
\begin{equation}
\sym_0^\ast(B^\ad) = W_-^\ast \oplus \set{ u - g^\ast u: u \in (V_+^\ast)^{\frac{1}{2}}}.
\label{eq:B*}
\end{equation}
\end{theorem} 

\begin{definition}[Elliptic boundary condition] 
\label{Def:EllBCOp}
If one and hence all assertions in Theorem~\ref{Thm:EllEquiv} hold true, we call $B$ an \emph{elliptic boundary condition}.
\end{definition}

Note that for an elliptic boundary condition $B$, we have that $B^{\ad}$ is an elliptic boundary condition for $D^\dagger$.

The following is a useful notion to allow for regularity of solutions up to the boundary. 
\begin{definition}[$(s+\frac{1}{2})$-(semi)regular boundary condition]
\label{Def:SemiRegular}
For $s \geq \frac{1}{2}$, we say an elliptic boundary condition $B$ is \emph{$(s+\frac{1}{2})$-semiregular} w.r.t.\ an admissible spectral cut $r$ if $W_+\subset \SobH{s}(\Sigma;E)$ and $g(V_-^s)\subset V_+^s$.
Here, $W_\pm$, $V_\pm$ and $g$ are as in Definition~\ref{Def:EllBC}.

If, in addition, $B^*$ is also $(s+\frac{1}{2})$-semiregular w.r.t.\ $r$, then we say that $B$ is \emph{$(s+\frac{1}{2})$-regular} w.r.t.\ $r$.
\end{definition}

It turns out that (semi)regularity is independent of the choice of admissible spectral cut $r$, cf.\ Lemma~\ref{Lem:RegEquiv}.

The main result along these lines is the following higher boundary regularity result.
\begin{theorem}[Higher boundary regularity]
\label{Thm:HBR}
Let $m \in \N$ and $B$ be an elliptic boundary condition that is $m$-semiregular. 
Then, for $k \in \set{0, 1, \dots, m-1}$ and $u \in \dom(D_B)$, we have that $D_{\max} u \in \SobH[loc]{k}(M;F)$ implies $u \in \SobH[loc]{k+1}(M;E)$.
\end{theorem}

The following are two significant types of boundary conditions that arise in applications. 
\begin{definition}[Local boundary condition]
\label{Def:LocalBC}
A boundary condition $B \subset \SobH{\frac{1}{2}}(\Sigma;E)$ is a \emph{local boundary condition} if there exists a subbundle $E' \subset E\rest{\Sigma}$ such that
$$ B = \close{\SobH{\frac{1}{2}}(\Sigma;E')}^{\checkH(A)}.$$
\end{definition}

\begin{definition}[Pseudo-local boundary condition]
\label{Def:PLocalBC}
Let $P$ be a classical pseudo-differential projector of order zero (not necessarily orthogonal).
Then, 
$$ B = \close{P\, \SobH{\frac{1}{2}}(\Sigma;E)}^{\checkH(A)} $$
is called a pseudo-local boundary condition.
\end{definition}

The following is a useful tool that characterises pseudo-local conditions.
\begin{theorem}
\label{Thm:PL}
For a pseudo-local boundary condition $B = \close{P\, \SobH{\frac{1}{2}}(\Sigma;E)}^{\checkH(A)} $, the following are equivalent: 
\begin{enumerate}[(i)]
\item 
\label{Thm:PL1}
$B$ is an elliptic boundary condition and $B = P\, \SobH{\frac12}(\Sigma;E)$.
\item 
\label{Thm:PL2}
For some/every admissible spectral cut $r \in \R$, the operator
	$$P - \chi^{+}(A_r): \Lp{2}(\Sigma;E) \to \Lp{2}(\Sigma;E)$$
	is a Fredholm operator.
\item 
\label{Thm:PL3}
For some/every admissible spectral cut $r \in \R$, the operator
	$$P - \chi^{+}(A_r): \Lp{2}(\Sigma;E) \to \Lp{2}(\Sigma;E)$$
	is elliptic.
\item 
\label{Thm:PL4} 
For every $\xi \in T_x^\ast \Sigma\setminus \set{0}$, $x \in \Sigma$, the principal symbol $\sym_{P}(x,\xi): E_x \to E_x$ restricts to an isomorphism from the sum of the generalised eigenspaces of $\imath\sym_{A_r}(x,\xi)$ to the eigenvalues with negative real part onto the image $\sym_{P}(x,\xi)(E_x)$ and, similarly, $\sym_{P^\ast}(x,\xi)$ restricts to an isomorphism from the sum of the generalised eigenspaces of $\imath\sym_{A_r^*}(x,\xi)$ to the eigenvalues with negative real part onto $\sym_{P^*}(x,\xi)(E_x)$.
\end{enumerate} 
\end{theorem}

In condition~\ref{Thm:PL4} we have to consider generalised eigenspaces rather than eigenspaces because the principal symbol $\sym_{A}(x,\xi)$ is not necessarily diagonalisable, cf.\ the example in \eqref{eq:nondiag} below.

An immediate and significant consequence is that every such condition is $\infty$-regular.
\begin{cor}
\label{Cor:InfReg}
Every pseudo-local elliptic boundary condition is $\infty$-regular.
In particular, if $D_B v \in \Ck{\infty}(M;F)$ then $v \in \Ck{\infty}(M;E)$. 
That is, $v$ is smooth \emph{up to} the boundary.
\end{cor}

By Remark~\ref{Rem:EigenvalueCounting}, $\imath\sym_{A_r}(x,\xi)$ has as many eigenvalues (counted with algebraic multiplicities) with positive real part as those with negative real part if $\dim M\ge3$.
Thus, if $P$ is the projector of an elliptic pseudo-local boundary condition, the rank of $\sym_{P}(x,\xi)$ must be precisely half the dimension of $E_x$.

There is a classical concept of \emph{local} elliptic boundary value problem known as Lopatinsky-Schapiro boundary conditions, see Section~20.1 in \cite{H94} for details.
For first order operators this condition reduces to considering a subbundle $E'\subset E|_{\Sigma}$ and setting $B=H^{\frac12}(\Sigma;E')$.
We consider the subbundle $F':=(\sym_0^{-1})^*(E')\subset F|_{\Sigma}$.
Then $B^*=\SobH{\frac12}(\Sigma;(F')^\perp)$.

\begin{cor}
\label{Cor:LS}
If $(D,B)$ and $(D^\dagger,B^*)$ form elliptic boundary value problems in the sense of Lopatinsky and Schapiro, then $B=H^{\frac12}(\Sigma;E')$ and $B^\perp=H^{\frac12}(\Sigma;(E')^\perp)$  are elliptic boundary conditions for $D$ in the sense of Definition~\ref{Def:EllBCOp} and so are $B^*=\SobH{\frac12}(\Sigma;(F')^\perp)$ and $(B^*)^\perp=\SobH{\frac12}(\Sigma;F')$ for $D^\dagger$.

Moreover, these boundary conditions are pseudo-local and hence $\infty$-regular.
\end{cor}

The following condition is required to yield Fredholm operators for elliptic boundary conditions. 
This condition automatically  holds for manifolds that are compact. 

\begin{definition}[Coercive at infinity]
The operator $D$ is said to be \emph{coercive at infinity} if there exists a compact subset $K \subset M$ and a constant $C$ such that 
$$ \norm{u}_{\Lp{2}(M;E)} \leq C \norm{D u}_{\Lp{2}(M;E)}$$
for all $u \in \Ck[c]{\infty}(M;E)$ with $\spt u \subset M \setminus K$.
\end{definition}

\begin{theorem}[Fredholmness]
\label{Thm:Fredholm}
Let $D$ and $D^\dagger$ be coercive at infinity. 
Suppose that  $B$ an elliptic boundary condition for $D$.
Then, the following hold. 
\begin{enumerate}[(i)]
\item  \label{Thm:Fredholm:1}
$D_B$ is a Fredholm operator and 
	$$ \Ind(D_B) = \dim \ker D_B - \dim \ker D^\dagger_{B^\ast} \in \In.$$
\item \label{Thm:Fredholm:2} 
Let $C$  be  a complementary subspace to $B$ in $\checkH(A)$ with an associated projection $\check{P}:\checkH(A) \to \checkH(A)$ with kernel $B$ and image $C$.
	Then
	$$\check{L}: \dom(D_{\max}) \to \Lp{2}(M;F) \oplus C,\quad \check{L}u := (D_{\max}u, \check{P}u\rest{\Sigma})$$
	is a Fredholm operator with the same index as $D_{B,\max} = D_B$.
\item  \label{Thm:Fredholm:3}
If $B' \subset B$ is another elliptic boundary condition, then $\dim \faktor{B}{B'} < \infty$ and 
	$$\Ind(D_B) = \Ind (D_{B'}) + \dim\faktor{B}{B'}.$$
\end{enumerate} 
\end{theorem} 

\section{Examples}
\label{Sec:Examples}

We look at a few examples in order to see how an adapted boundary operator $A$ can look like.
First we give a very simple example showing that the principal symbol of $A$ need not be diagonalisable.

Then we consider a Dirac operator on $M$.
For an orthonormal boundary transversal, the adapted boundary operator is essentially the Dirac operator of the boundary (up to zero-order terms).
In particular, it can be chosen to be selfadjoint.
We show in a simple example that for non-orthonormal transversal the adapted boundary operator no longer has real spectrum.

Finally, we consider the Rarita-Schwinger operator on $M$ and find that the eigenspaces of $A$ do not span $\Lp{2}(\Sigma;S^{\frac32}\Sigma)$; 
the generalised eigenspaces are larger than the true eigenspaces.

\subsection{Nondiagonalisable principal symbols}
\label{subsec:nondiag}

The following example shows that the principal symbol of a first-order elliptic operator need not be diagonalisable, not even if it is an induced boundary operator.
To start, let $M=\R^2$ whose Cartesian coordinates we denote by $t$ and $x$.
The operator acts on $\C^2$-valued functions, i.e.\ the bundle $E\to M$ is the trivial complex vector bundle of rank $2$.
Put
$$
D= \begin{pmatrix} -\imath & 1 \\ 0 & -\imath\end{pmatrix}\frac{\partial}{\partial t} + \begin{pmatrix} 1 & 0 \\ 0 & 1 \end{pmatrix}\frac{\partial}{\partial x}.
$$
Since 
$$
\det(\sym_D((t,x),(\xi_1,\xi_2)))
=
\det\begin{pmatrix} -\imath\xi_1+\xi_2 & \xi_1 \\ 0 & -\imath\xi_1 + \xi_2 \end{pmatrix}
=
(\xi_2-\imath\xi_1)^2
\ne
0
$$
for $(\xi_1,\xi_2)\neq(0,0)$, the operator $D$ is elliptic.

Now we restrict to $\{t\ge0\}$ and consider the operator $A$ induced on the boundary $\Sigma=\{t=0\}$.
We find
\begin{equation}
A
=
\begin{pmatrix} -\imath & 1 \\ 0 & -\imath\end{pmatrix}^{-1}\cdot\begin{pmatrix} 1 & 0 \\ 0 & 1 \end{pmatrix}\frac{\partial}{\partial x}
=
\begin{pmatrix} \imath & 1 \\ 0 & \imath\end{pmatrix}\frac{\partial}{\partial x}.
\label{eq:nondiag}
\end{equation}
We see that the principal symbol $\sym_A(\xi)=\xi\cdot\begin{pmatrix} \imath & 1 \\ 0 & \imath\end{pmatrix}$ is nondiagonalisable.

\subsection{Dirac operator and non-orthogonal boundary transversal}
Let $M$ be a compact Riemannian spin manifold with compact boundary $\Sigma$.
Let $\tau_0$ be the conormal along $\Sigma$ whose length is normalised to $1$ w.r.t.\ the Riemannian metric.
Let $D$ be the Dirac operator acting on spinors on $M$.
The adapted boundary operator $A$ on $\Sigma$ can be chosen in such a way that it is essentially the Dirac operator on $\Sigma$.
In particular, it is then selfadjoint and has real spectrum. 
The Atiyah-Patodi-Singer boundary conditions $B=\chi^-(A)\SobH{\frac12}(\Sigma;E)$ then are the most prominent example of an elliptic boundary condition.

More specifically, let $M=S^1\times [0,\infty)$ so that $\Sigma=S^1$ with length $2\pi$.
If we denote the standard coordinates on $M$ by $(\theta,t)$, then the interior unit normal covector field is given by $\tau_0=dt$.
With the appropriate choice of spin structure, the spinor bundle is the trivial $\C^2$-bundle and the Dirac operator is given by
$$
D= \begin{pmatrix} 0 & -1 \\ 1 & 0\end{pmatrix}\frac{\partial}{\partial t} + \begin{pmatrix} 0 & \imath \\ \imath & 0\end{pmatrix}\frac{\partial}{\partial \theta}.
$$
As an adapted boundary operator we can choose
$$
A_0
= 
\begin{pmatrix} 0 & -1 \\ 1 & 0\end{pmatrix}^{-1} \begin{pmatrix} 0 & \imath \\ \imath & 0\end{pmatrix}\frac{\partial}{\partial \theta}
=
\begin{pmatrix}  \imath & 0 \\ 0 & -\imath\end{pmatrix}\frac{\partial}{\partial \theta}.
$$
This operator has the eigenvalues $k\in\mathbb{Z}$ with multiplicity $2$.
The corresponding eigenspinors are given by 
$$
\theta\mapsto \begin{pmatrix} e^{-\imath k \theta} \\ 0 \end{pmatrix}
\text{ and }
\theta\mapsto \begin{pmatrix} 0 \\ e^{\imath k \theta}\end{pmatrix}.
$$
Now, fix a parameter $\alpha\in\R$ and put $\tau=\tau_0+\alpha d\theta$.
From 
\begin{align*}
\sym_D(x,d\theta)^{-1}\sym_D(x,\tau)
&=
\sym_D(x,d\theta)^{-1}\big(\sym_D(x,\tau_0)+\alpha\sym_D(x,d\theta)\big) \\
&=
\begin{pmatrix} 0 & \imath \\ \imath & 0\end{pmatrix}^{-1}\begin{pmatrix} 0 & -1 \\ 1 & 0\end{pmatrix} + \alpha \begin{pmatrix} 1 & 0 \\ 0 & 1\end{pmatrix} \\
&=
\begin{pmatrix} \alpha-\imath & 0 \\ 0 & \alpha+\imath \end{pmatrix}
\end{align*}
we get the adapted boundary operator
$$
A = \begin{pmatrix} \alpha-\imath & 0 \\ 0 & \alpha+\imath \end{pmatrix}^{-1}\frac{\partial}{\partial \theta}
= \frac{1}{\alpha^2+1}\begin{pmatrix} 1-\imath\alpha & 0 \\ 0 & 1+\imath\alpha \end{pmatrix} A_0.
$$
Thus, for $\alpha\neq0$, the operator $A$ has the eigenvalues $\frac{1-\imath\alpha}{\alpha^2+1}k$ and $\frac{1+\imath\alpha}{\alpha^2+1}k$ for $k\in\mathbb{Z}$ with multiplicity $1$.
\begin{center}
\includegraphics[scale=0.6]{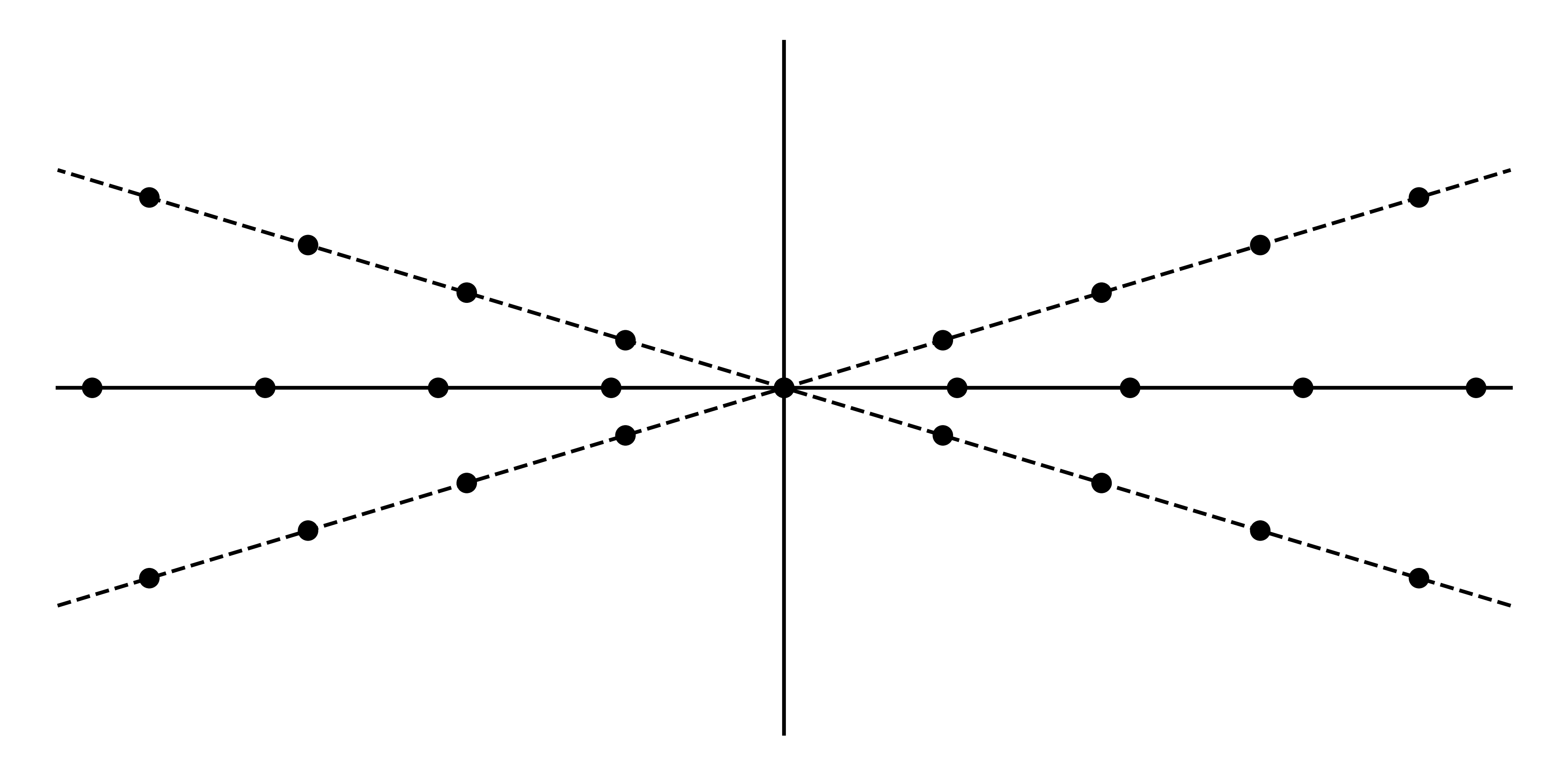}

\Fig{Spectrum of $A_0$ and of $A$ (dashed lines)}
\end{center}

\subsection{The Rarita-Schwinger operator}

If $D$ is the Dirac operator on $M$, then the adapted boundary operator can be chosen to be selfadjoint.
For other geometrically natural operators this is no longer the case.
We discuss the Rarita-Schwinger operator as an example.
Since this operator is much less known as compared to the Dirac operator, we include some basics in the discussion, see \cite{Bures1999,HS2018, Wang1991} for further aspects concerning this operator.

Let $M$ be a Riemannian spin manifold of dimension $n\ge3$.
We denote its complex spinor bundle by $SM$.
Clifford multiplication will be denoted by $\gamma:SM\otimes TM\to SM$, i.e.\ $\gamma(\phi\otimes\xi)=\xi\cdot\phi$.
Clifford multiplication is characterised by the relation $\eta\cdot\xi\cdot\phi+\xi\cdot\eta\cdot\phi+2\langle\xi,\eta\rangle\phi=0$.
For $\Phi=\phi\otimes\xi\in S_xM\otimes T_xM$ and $\eta\in T_xM$ we write $\Phi(\eta):=\langle\xi,\eta\rangle\phi$.
Now
$$
\iota: SM\to SM\otimes TM, \quad \iota(\phi) = -\frac1n \sum_{j=1}^n e_j\cdot\phi \otimes e_j
$$
defines an embedding of $SM$ into $SM\otimes TM$.
Here $e_1,\ldots,e_n$ is a local orthonormal tangent frame and $\iota$ is independent of the choice of this frame.
We then have
$$
\iota(\phi)(\eta) = -\frac1n \sum_{j=1}^n \langle e_j,\eta\rangle e_j\cdot\phi = -\frac1n \eta\cdot\phi
$$
and
$$
\gamma(\iota(\phi)) = -\frac1n \sum_{j=1}^n e_j\cdot e_j\cdot\phi = \phi.
$$
The bundle of $\frac32$-spinors is defined as
$$
\SS M := \ker(\gamma) = \iota(SM)^\perp \subset SM\otimes TM.
$$
It is naturally a Hermitian vector bundle over $M$.
Now, $\iota^*=\frac1n\gamma$ and $\iota\circ\gamma$ is the orthogonal projection $SM\otimes TM\to \iota(SM)$.
The complementary orthogonal projection $\PI:SM\otimes TM\to \SS M$ is given by 
$$
\PI(\Phi)=\Phi-\iota(\gamma(\Phi)), \text{ i.e. }\quad  \PI(\Phi)(\eta) = \Phi(\eta) + \frac1n \eta\cdot\gamma(\Phi).
$$
Now let $\DD:C^\infty(M;SM\otimes TM)\to C^\infty(M;SM\otimes TM)$ be the twisted Dirac operator with coefficients in $TM$.
We define the \emph{Rarita-Schwinger operator} by restricting to the subbundle $\SS M\subset SM\otimes TM$, i.e.
$$
D := \PI\circ\mathscr{D}|_{\SS M}.
$$
Its principal symbol is easily computed for $\Phi\in \SS_xM$ to be
\begin{equation}
\sym_D(x,\xi^\flat)(\Phi) 
= \PI(\sym_\mathscr{D}(x,\xi^\flat)(\Phi))
= \PI((\xi\otimes1)\Phi)
= (\xi\otimes1)\Phi + 2\iota(\Phi(\xi)).
\label{eq:RSsymbol}
\end{equation}
Here $\xi^\flat$ is the covector corresponding to the tangent vector $\xi$ under the isomorphism $TM\to T^*M$ induced by the Riemannian metric.
For $\xi\in T_xM\setminus\{0\}$ put 
\begin{align*}
\SS(\xi) 
&:= 
\{\Phi\in\SS_xM : \Phi(\xi)=0\},\\
\tilde S^{\frac32}(\xi)
&:=
\{\PI(\phi\otimes\xi) : \phi\in S_xM\}.
\end{align*}
Then $\SS_xM= \SS(\xi)\oplus\tilde S^{\frac32}(\xi)$ is an orthogonal decomposition which is respected by $\sym_D(x,\xi^\flat)$.
On $\SS(\xi)$ the principal symbol $\sym_D(x,\xi^\flat)(\Phi)$ simply acts by $\xi\otimes 1$ as one can see from \eqref{eq:RSsymbol}.
For $\PI(\phi\otimes\xi)\in\tilde{S}^{\frac32}$ we have 
$$
\sym_D(x,\xi^\flat)(\Phi)(\PI(\phi\otimes\xi))
=
\tfrac{n-2}{n}\,\PI(\xi\phi\otimes\xi).
$$
Thus, with respect to this splitting, the square of $\sym_D(x,\xi^\flat)$ is given by
$$
\sym_D(x,\xi^\flat)^2
=
-|\xi|^2\cdot
\begin{pmatrix}
1 & 0 \\ 0 & (\frac{n-2}{n})^2
\end{pmatrix}.
$$
In particular, $D$ is elliptic.

Next, we consider adapted boundary operators.
For the sake of simplicity, we restrict ourselves to the case $n=3$.
Let $\tau=\vec{T}^\flat\in T^*_xM$ be the interior unit conormal where $x\in\Sigma=\partial M$.
For $\xi \in T_x\Sigma$ the principal symbol of an adapted boundary operator is given by $\sym_A(x,\xi^\flat)=\sym_D(x,\tau)^{-1}\sym_D(x,\xi^\flat)$.
We first determine the spectrum of $\sym_A(x,\xi^\flat)$ and assume for now that $|\xi|=1$.
We choose $\eta\in T_x\Sigma$ in such a way that $\vec{T},\xi,\eta$ form an orthonormal basis of $T_xM$ and that $\vec{T}\cdot\xi\cdot\eta$ acts as $1$ by Clifford multiplication.

Let $\Phi\in S^{\frac32}_xM\setminus\{0\}$ be an eigenvector of $\sym_A(x,\xi^\flat)$ to the eigenvalue $\lambda\in\C$, i.e.\ 
\begin{equation}
\sym_D(x,\xi^\flat)\Phi=\lambda\sym_D(x,\tau)\Phi.
\label{eq:RSsymEV}
\end{equation}
We write $\Phi=\phi_0\otimes\vec{T} + \phi_1\otimes\xi + \phi_2\otimes \eta$ where $\phi_j\in S_xM$.
The relation $\gamma(\Phi)=0$ is equivalent to $\phi_2=\eta\vec{T}\phi_0+\eta\xi\phi_1$, i.e.\
$$
\Phi = \phi_0\otimes\vec{T} + \phi_1\otimes\xi - (\vec{T}\eta\phi_0+\xi\eta\phi_1) \otimes \eta .
$$
A straightforward computation yields
\begin{align*}
\sym_D(x,\xi^\flat)\Phi
&=
(\xi\phi_0-\tfrac23 \vec{T}\phi_1)\otimes \vec{T} + \tfrac13 \xi\phi_1\otimes \xi + (\phi_0+\tfrac13\eta\phi_1)\otimes \eta, \\
\sym_D(x,\tau)\Phi
&=
\tfrac13\vec{T}\phi_0\otimes\vec{T} + (\vec{T}\phi_1-\tfrac23\xi\phi_0)\otimes\xi + (\tfrac13\eta\phi_0-\phi_1)\otimes\eta .
\end{align*}
Comparing the coefficients of the $\vec{T}$ and $\xi$-terms, \eqref{eq:RSsymEV} implies
\begin{align*}
\phi_1 
&=
-\frac{3\vec{T}\xi+\lambda}{2}\phi_0, \\
\phi_0
&=
\frac{3\vec{T}\xi-\frac1\lambda}{2}\phi_1.
\end{align*}
This shows that if $\phi_0=0$, then $\phi_1=0$ and hence $\Phi=0$.
Thus $\phi_0\neq0$.
Combining the two equations gives 
$$
\phi_0 = \frac{10+(\frac3\lambda - 3\lambda)\vec{T}\xi}{4}\phi_0,
$$
hence
\begin{equation}
(2+(\lambda-\tfrac1\lambda)\eta)\phi_0=0.
\label{eq:RSEV1}
\end{equation}
In a suitable spinor basis, Clifford multiplication by $\eta$ is given by the matrix $\begin{pmatrix}0&\imath\\ \imath&0\end{pmatrix}$.
Thus \eqref{eq:RSEV1} says that the matrix
\begin{equation}
\begin{pmatrix}
2 & (\lambda-\tfrac1\lambda)\imath\\
(\lambda-\tfrac1\lambda)\imath & 2
\end{pmatrix}
\label{eq:RSsymbolMatrix}
\end{equation}
is singular.
Its determinant is given by $(\lambda+\tfrac1\lambda)^2$ and has to vanish.
Thus $\lambda=\pm\imath$.

If $\xi$ is no longer be assumed to have unit length, then by linearity the eigenvalues of $\sym_A(x,\xi^\flat)$ are $\pm |\xi|\imath$.
Note that for $\lambda=\pm\imath$, the matrix in \eqref{eq:RSsymbolMatrix} has rank $1$.
Hence, the geometric multiplicity of the eigenvalue $\lambda=\pm|\xi|\imath$ is $1$, while the algebraic multiplicities must add up to $4$, the dimension of $S^{\frac32}_xM$ (except for $\xi=0$).
Indeed, the algebraic multiplicities are $2$.

We compute the spectrum of $A$ for the example where $\Sigma$ is a flat $2$-torus.
Let $\Sigma=\R^2/\Gamma$ where $\Gamma$ is a lattice.
We assume that $\Sigma$ carries the flat metric induced from $\R^2$ and the spin structure with respect to which the spinor bundle $S\Sigma$ has trivial holonomy, cf.\ \cite{F84}.
Then $S^{\frac32}M$ has also trivial holonomy and the space of parallel sections has dimension $4$.
As adapted boundary operator, we choose 
$$
A = \sym_A(dx^1)\frac{\partial}{\partial x^1} + \sym_A(dx^2)\frac{\partial}{\partial x^2}.
$$
Let $k=(k_1,k_2)\in\Gamma^*$, the dual lattice.
Then $V_k = \{e^{2\pi\imath\langle k,x\rangle}\Phi : \Phi\text{ parallel}\}$ is an $A$-invariant subspace of $\Lp{2}(\Sigma;S^{\frac32}\Sigma)$.
Indeed, 
$$
A(e^{2\pi\imath\langle k,x\rangle}\Phi)=2\pi\imath\, e^{2\pi\imath\langle k,x\rangle}\sym_A(x,k_1dx^1+k_2dx^2)\Phi.
$$
Thus, restricted to the $4$-dimensional space $V_k$, the operator $A$ has the eigenvalues $\pm 2\pi|k|$, with multiplicity~$1$ each.
By Fourier analysis we know $\Lp{2}(\Sigma;S^{\frac32}\Sigma)=\bigoplus_{k\in\Gamma^*}V_k$.
Thus $A$ has real spectrum but the eigenspaces do not span all of $\Lp{2}(\Sigma;S^{\frac32}\Sigma)$.
The generalised eigenspaces, however,  do.

\section{Operator theory of the boundary adapted operator}
\label{Sec:OpTheory}

Recall the principal symbols $\sym_D$ and $\sym_A$ for the operator $D$ and the induced operator $A$ on the boundary from \S\ref{Sec:SR}.

Let $S _\omega = \set{\lambda  \in \C\setminus\{0\}: \arg \lambda \in [0,\omega)\cup(\pi-\omega,\pi+\omega)\cup(2\pi-\omega,2\pi)}$ be the open bisector of angle $\omega$ with vertex at the origin in $\C$.

\begin{center}
\includegraphics[scale=0.6]{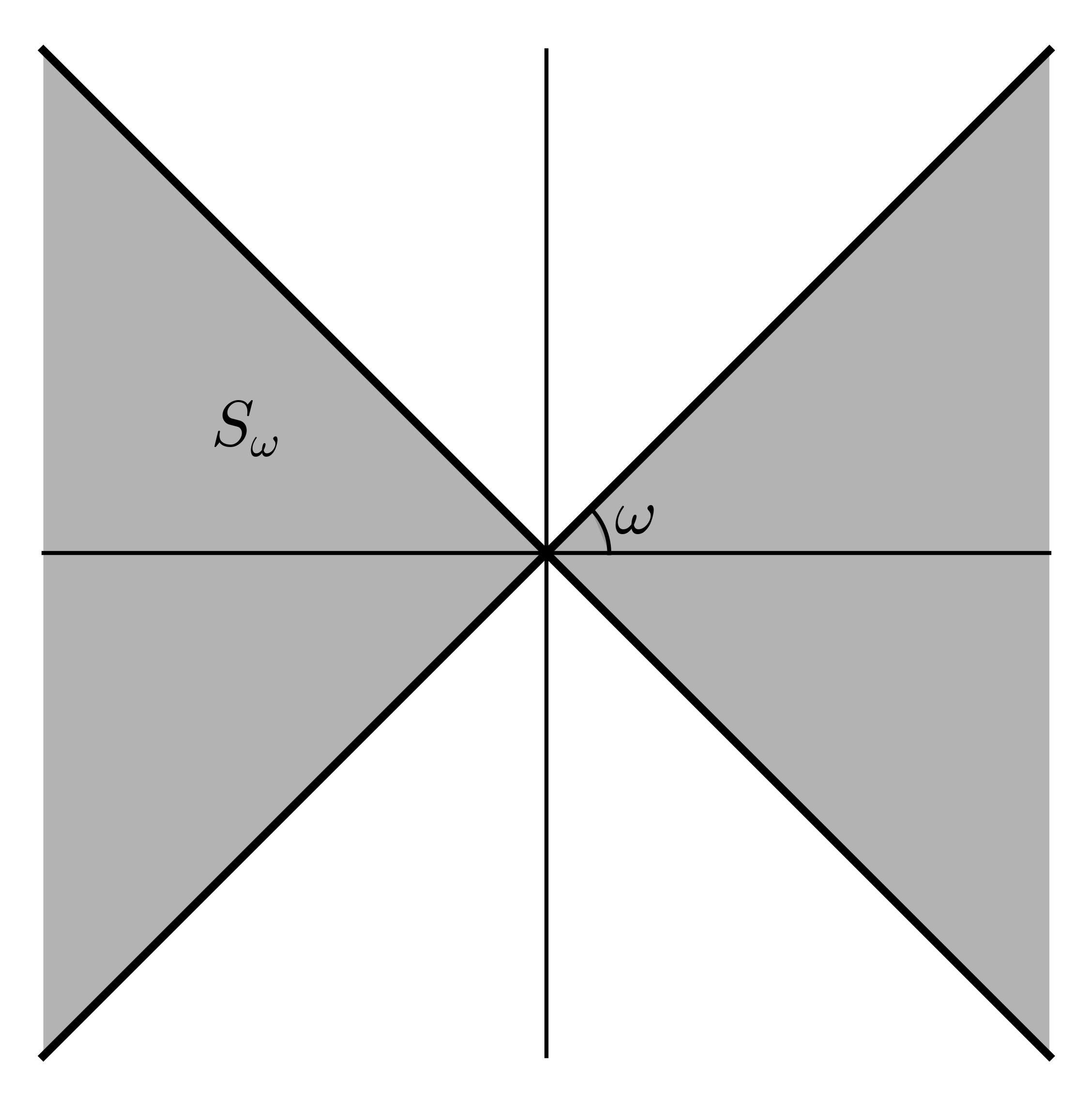}

\Fig{Bisector $S_\omega$}
\end{center}

We make the following central observation. 
\begin{lem}
\label{Lem:SymSpec}
There exists a $\nu>0$ such that 
$$\spec(\sym_A(x, \xi)) \cap S_{\nu} = \emptyset,$$
for all $(x,\xi) \in T^*\Sigma$ such that $\xi \neq 0$. 
\end{lem}
\begin{proof}
We first show, by contradiction, that the spectrum of the principal symbol of $A$ satisfies:
\begin{equation}
\spec(\sym_{A}(x, \xi)) \intersect \R = \emptyset
\label{eq:symAnoRealSpec}
\end{equation}
for any $0 \neq \xi\in T^*_x\Sigma$.

Suppose that the conclusion is false for some $\xi \neq 0$ . 
Then, there exists $t \in \R$ and $v \in E_x\setminus\{0\}$ such that $\sym_A(x, \xi)v = t v$. 
This is equivalent to $\sym_{D}(x, \xi)v = t \sym_D(x, \tau)v = \sym_D(x, t \tau)v$.
Then $\xi-t\tau\neq0$ and
$$ \sym_D(x, \xi-t\tau)v =  \sym_{D}(x, \xi)v -t\sym_{D}(x,\tau)v = 0.$$ 
But this contradicts the ellipticity of $D$ and proves \eqref{eq:symAnoRealSpec}.

Since the spectra of $\sym_{A}(x,\xi)$ depend continuously on $(x,\xi)$ and the unit cosphere bundle of $\Sigma$ is compact, the set of all eigenvalues of all $\sym_{A}(x,\xi)$ with $x\in\Sigma$ and $\|\xi\|=1$ is a compact subset of $\C$.
It avoids the real axis, hence we can find a $\nu>0$ such that $\bigcup_{\|\xi\|=1}\spec(\sym_A(x, \xi)) \cap S_{\nu} = \emptyset$.
Since $S_\nu$ is conic and $t\mapsto \sym_A(x, t\xi)$ is homogeneous of degree 1, the claim follows.
\end{proof} 

\begin{remark}
\label{Rem:EigenvalueCounting}
We note that if $\dim M\ge3$, for fixed $(x,\xi)$, the number of eigenvalues (counted with algebraic multiplicities) of $\sym_{A}(x,\xi)$ with positive imaginary part equals that of the eigenvalues with negative imaginary part.
Namely, in this case $\dim T_x^*\Sigma\ge2$ and hence we can joint $\xi$ with $-\xi$ by a continuous path $\xi(t)$ of nonvanishing covectors. 
Since $\sym_{A}(x,\xi(t))$ never has real eigenvalues, the number of eigenvalues with positive imaginary part does not change as $t$ varies.
But $\sym_{A}(x,-\xi)=-\sym_{A}(x,\xi)$ because $A$ is of first order.

In dimension $2$ this is no longer true, see Subsection~\ref{subsec:nondiag} for a counter-example.
\end{remark}

Throughout, we regard the operator $A$ as an unbounded operator on $\Lp{2}(\Sigma;E)$ and we obtain the following information on the spectrum of the operator $A$.
Let $B_R=\{\lambda\in\C: |\lambda|\le R\}$.

\begin{proposition} 
\label{Prop:ASpec} 
The $\Lp{2}$-spectrum of the operator  $A$  satisfies the following: 
\begin{enumerate}[(i)]
\item\label{ASpec1}
there exists $R > 0$ and $\omega \in [0, \pi/2)$ such that  $\spec(A) \subset \overline{S_{\omega}} \cup B_{R}$; 
\item\label{ASpec2}
there exists $C > 0$ such that for $\zeta \in \C \setminus (S_\omega \cup B_R)$ the resolvent bound $\modulus{\zeta}\norm{(A - \zeta)^{-1}} \leq C$ holds;
\item\label{ASpec3}
the spectrum $\spec(A)$ is discrete. 
\end{enumerate}
\end{proposition} 
\begin{proof} 
Let $\nu\in(0,\pi/2)$ be as in Lemma~\ref{Lem:SymSpec}.
Choose $\eta\in(0,\nu)$.
From $\spec(\imath\sym_A(x, \xi)) \cap \imath S_{\nu} = \emptyset$ for all $(x,\xi)$ we conclude $\spec(\imath\sym_A(x, \xi)) \cap \overline{\imath S_{\eta}} = \emptyset$ whenever $\xi\neq 0$.
Now \ref{ASpec1} and \ref{ASpec2} follow from Theorem~9.3 in \cite{Shubin} with $\omega=\pi/2-\eta$.

Assertion~\ref{ASpec3} simply follows from the ellipticity of $A$ along with the Rellich embedding: $\SobH{1}(\Sigma;E) \embed \Lp{2}(\Sigma;E)$ is compact, which gives that $(A - \lambda)^{-1}$ is a compact operator. 
\end{proof}  

To obtain useful operators associated to $A$ needed for the study of boundary value problems, we recall the important class of operators that are known as \emph{bi-sectorial}.
Let $T: \dom(T) \subset \mathcal{B} \to \mathcal{B}$ be a densely-defined and closed operator on a Banach space $\mathcal{B}$. 
Given $0 < \omega < \pi/2$, it is called \emph{$\omega$-bisectorial} if $\spec(T) \subset \overline{S_{\omega}}$ and for each $\mu > \omega$, there exists a $C_\mu$ such that for $\zeta \not\in \overline{S_{\mu}}$, 
$$ \norm{(\zeta - T)^{-1}} \leq \frac{C_\mu}{\modulus{\zeta}}.$$

Put $\C_+:=\{\lambda\in\C: \Re\lambda>0\}$.
If the condition holds with $S_\omega$ and $S_{\mu}$ replaced with $S_{\omega+} = S_\omega \cap \C_+$ and $S_{\mu+}$ respectively, then $T$ is said to be \emph{$\omega$-sectorial}.
To avoid confusion, let us remark that in older literature, the notion $\omega$-sectorial (in the context of a Hilbert space) contains the extra condition on the numerical range  of $T$.  
Such operators are now called \emph{Kato-sectorial}.
See Remark~\ref{Rem:KatoSec} for further details.

For any $r \in \R$ and $\epsilon > 0$, let $L_{\epsilon,r} = \union_{\modulus{s -r} \leq \epsilon} l_s$ be the closed vertical strip of width $2\epsilon$ centred about $l_r$. 
\begin{lem}
\label{Lem:l_r}
For any  admissible spectral cut $r$ there exists an $\epsilon_r > 0$ such that $L_{r,\epsilon} \subset \res(A)$.
\end{lem}
\begin{proof}
By Proposition~\ref{Prop:ASpec}~\ref{ASpec1} and \ref{ASpec3}, the intersection $\spec(A) \intersect L_{1,r}$ is finite, $\spec(A) \intersect L_{1,r} = \set{\lambda_1, \dots, \lambda_k}$. 
If the intersection is empty put $\epsilon_r:=1$, otherwise $\epsilon_r := \frac{1}{2}\min_{1\le i \le k}{ \modulus{\Re \lambda_i - r}}$ does the job.
\end{proof}

From this point onward, we fix an admissible spectral cut $r$.
\begin{proposition}
\label{Prop:A_rSec}
There exist $0<\omega_r < \pi/2$ and $\epsilon_r>0$ such that $A_r$ is invertible and $\omega_r$-bisectorial and $\spec(A_r) \subset \overline{S_{\omega_r}} \setminus L_{\epsilon_r,0}$.
\end{proposition}
\begin{proof}
Let $R$ and $\omega$ be as in Proposition~\ref{Prop:ASpec} and $r$ and $\epsilon_r$ as in Lemma~\ref{Lem:l_r}. 
Then the spectrum of $A$ is contained in the area shaded in grey in Figure~\ref{fig:omegar}.

\begin{center}
\includegraphics[scale=0.6]{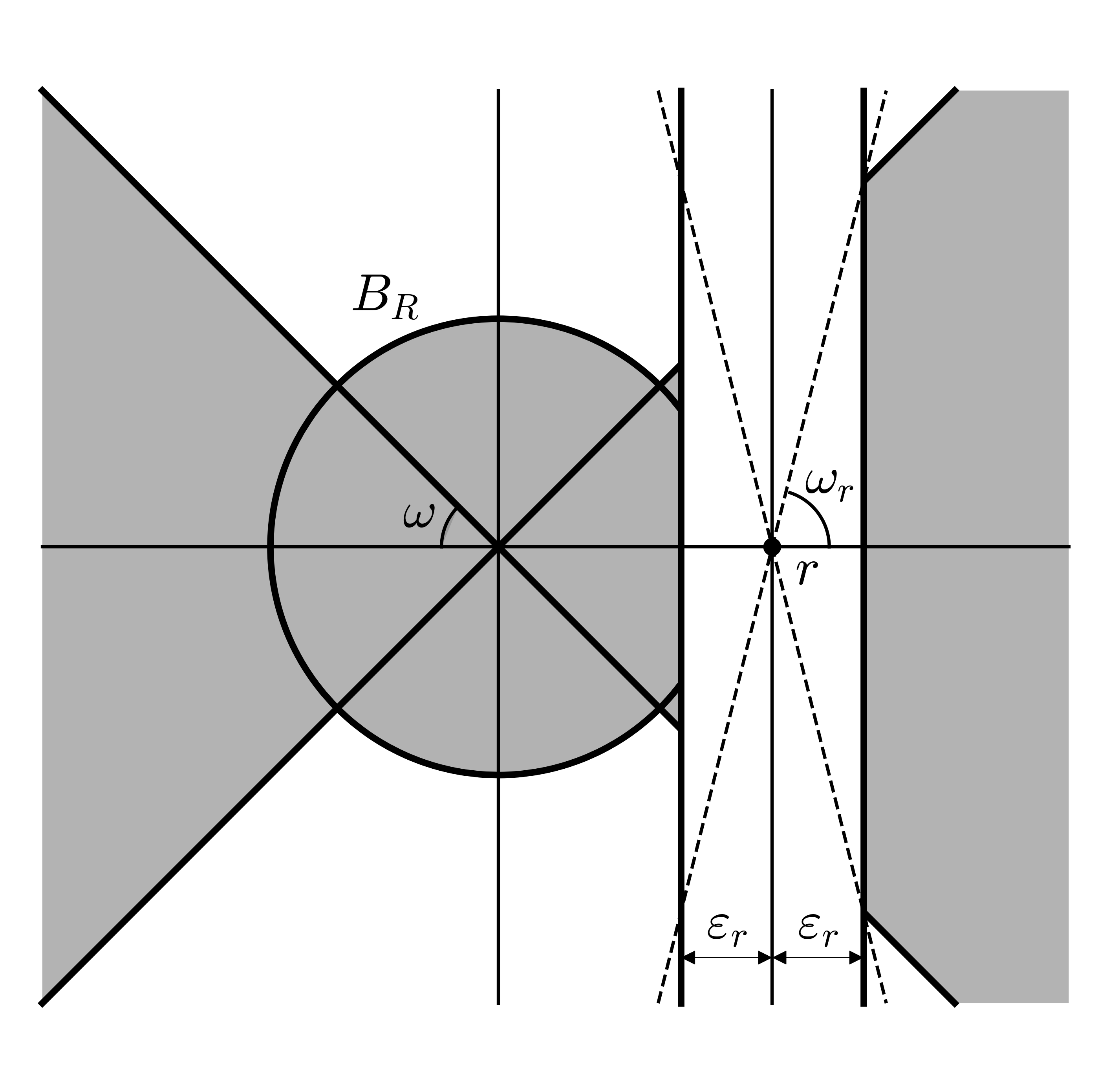}

\Fig{Spectrum of $A$}
\label{fig:omegar}
\end{center}

Thus the spectrum of $A_r$ is contained in $\overline{S_{\omega_r}} \setminus L_{\epsilon_r,0}$, with $\omega_r$ as indicated in the picture.
In particular, $0\notin\spec(A_r)$, hence $A_r$ is invertible.

It remains to show the resolvent estimate for $\zeta\in\C\setminus\overline{S_{\omega_r}}$.
If, in addition, $\zeta+r\notin \overline{S_\omega}\cup B_R$ then Proposition~\ref{Prop:ASpec}~\ref{ASpec2} yields the resolvent bound
$$
\norm{(\zeta - A_r)^{-1}} 
=
\norm{((\zeta+r) - A)^{-1}} 
\le
\frac{C}{|\zeta+r|}
\le
\frac{|\zeta+r|+|r|}{|\zeta+r|}\frac{C}{|\zeta|}
\le
\bigg(1+\frac{|r|}{R}\bigg)\frac{C}{|\zeta|}.
$$
On the other hand, $Z:=((\overline{S_\omega}\cup B_R)-r) \cap (\C\setminus\overline{S_{\omega_r}})$ is a compact set contained in the resolvent set of $A_r$.
Hence there exists a constant $C_1>0$ such that $|\zeta|\le C_1$ and $\norm{(\zeta - A_r)^{-1}}\le C_1$ for all $\zeta\in Z$.
This implies, for those $\zeta$:
\begin{equation*}
\norm{(\zeta - A_r)^{-1}}
\le
C_1
\le
\frac{C_1^2}{|\zeta|}.
\qedhere
\end{equation*}
\end{proof}

To define spectral projectors, consider 
$$ \chi^{\pm}(\zeta) = 
	\begin{cases} 	1,  	& \text{if }\pm \Re \zeta > 0, \\
			0,	& \text{otherwise.}
	\end{cases}
$$

\begin{proposition} 
\label{Prop:AProj} 
For $r$ an admissible spectral cut,  the projection $\chi^{+}(A_r)$ to the  spectral subspace associated to the eigenvalues of the right half open plane is a pseudo-differential operator of order $0$.
Similarly, $\chi^-(A_r)$, the projection to the spectral subspace of the left half open plane, is also a pseudo-differential operator of order~$0$.
\end{proposition}
\begin{proof}
By Proposition~\ref{Prop:A_rSec}, $A_r$ is an invertible $\omega_r$-sectorial operator and also that $l_0 \subset \res(A_r)$.
The main theorem of Grubb in \cite{Grubb} asserts that $\chi^+(A_r)$ can be defined as the contour integral
\begin{equation}
\chi^+(A_r)u = \frac{\imath}{2\pi}\oint_{\set{\imath r: \infty > r > -\infty}} \zeta^{^-1}A (\zeta - A)^{-1} u \ d\zeta
\label{eq:chi+Ar}
\end{equation}
for $u\in \Ck{\infty}(\Sigma;E)$ and it extends to a pseudo-differential operator of order zero. 

The conclusion for $\chi^-(A_r)$ follows by analogy or simply from $\chi^{-}(A_r) = 1 - \chi^+(A_r)$.
\end{proof}

Recall the function $\sgn: \C \to \C$ given by:
$$ \sgn(\zeta) = 
\begin{cases} 			+1, 			&\text{if }\Re \zeta > 0, \\
				-1, 			&\text{if }\Re \zeta < 0, \\
				0, 			&\text{if }\Re \zeta = 0.
\end{cases}$$
Whenever $\Re\zeta \neq 0$, we have that  
$$\sgn(\zeta) = \frac{\zeta}{\sqrt{\zeta^2}} = \frac{\sqrt{\zeta^2}}{\zeta}=\chi^+(\zeta)-\chi^-(\zeta).$$
Thus, on using that $\chi^{\pm}(A)$ is $\Lp{2}(E_{\Sigma})$ bounded by Proposition~\ref{Prop:AProj} and $l_o\subset \res(A_r)$, define
\begin{equation}
\label{Eq:Amod}
\modulus{A_r} = \sqrt{A_r^2} = \sgn(A_r)A_r = (\chi^+(A_r)-\chi^-(A_r))A_r.
\end{equation}

For the remainder of this paper, we use the analyst's inequality $a \lesssim b$ to mean that $a \leq C b$, there exists some $C$.
Typically, the dependence of the implicit constant will be clear from context or made explicitly clear in the analysis.
We write $a \simeq b$ to mean $a \lesssim b$ and $b \lesssim a$.

\begin{lem}
\label{Lem:OpSect}
The operator $\modulus{A_r}$ is invertible and $\omega_r$-sectorial with $\dom(\modulus{A_r}) = \dom(A_r) = \SobH{1}(\Sigma;E)$ with the norm estimate $\norm{A_r u} \simeq \norm{\modulus{A_r}u}$,
\end{lem} 
\begin{proof}
Since the projectors $\chi^{\pm}(A_r)$ are bounded, the domain equality $\dom(\modulus{A_r}) = \dom(A_r) = \SobH{1}(\Sigma;E)$ follows.
For the norm estimate, first note
$$
\norm{|A_r|u}
=
\norm{(\chi^+(A_r)-\chi^-(A_r))A_ru}
\le
\norm{\chi^+(A_r)A_ru} + \norm{\chi^-(A_r)A_ru}
\lesssim
\norm{A_ru}.
$$
The reverse inequality follows similarly from $A_r=\sgn(A_r)|A_r|$.

The operators $A_r$ and $|A_r|$ have the same eigenvectors and if $\lambda$ is an eigenvalue of $A_r$ then $\sgn(\lambda)\lambda$ is the corresponding eigenvalue of $|A_r|$. 
This shows $\omega_r$-sectoriality.
\end{proof}

Since we now have that $\modulus{A_r}$ is an $\omega_r$-sectorial operator, we obtain the existence of a bounded holomorphic semigroup via standard semigroup theory (c.f.\ \cite{Kato}).
We require this in later parts to define the boundary extension operator.

\begin{proposition}
\label{Prop:ExpHol}
The operator family $\zeta \mapsto \exp(-\zeta \modulus{A_r})$ is holomorphic on the sector $S_{\frac{\pi}{2} - \omega_r}\cap\C_+$ and for any $\epsilon \in (0, \pi/2 - \omega_r)$, it is uniformly bounded on $S_{\frac{\pi}{2} - \omega_r - \epsilon}\cap\C_+$.
For $u \in \Lp{2}(E_{\Sigma})$, $U(\zeta) = \exp(-\zeta \modulus{A_r})u$ solves the heat equation $\partial_\zeta U(\zeta) = - \modulus{A_r} U(\zeta)$ in $\Lp{2}(E_{\Sigma})$ with $\lim_{|\zeta| \to 0} U(\zeta) = u$.
\end{proposition}
\begin{proof}
By Lemma~\ref{Lem:OpSect}, we have that $\modulus{A_r}$ is a $\omega_r$-sectorial operator.
From Chapter~9, \S6 in \cite{Kato}, the semigroup $\zeta\mapsto\exp(-\zeta\modulus{A_r})$ is holomorphic on $S_{\frac\pi2-\omega_r}$ and is uniformly bounded on $S_{\frac\pi2 - \omega_r - \epsilon}$ for each $\epsilon \in (0, \pi/2 - \omega_r)$.

By holomorphicity of the semigroup, $U$ solves the heat equation.
Equation~(1.58) on page~491 of \cite{Kato} yields $\lim_{|\zeta| \to 0} U(\zeta) = U(0) = u$.
\end{proof} 

\begin{remark}
\label{Rem:KatoSec}
In \cite{Kato}, the definition of a sectorial operator is what is known in the literature today as \emph{Kato-sectorial}, which further assumes that the numerical range of the operator is also valued inside the sector. 
However, the hypothesis on operator in Chapter 9 \S6 of \cite{Kato} (c.f. page 490) is precisely the modern notion of sectoriality which we have assumed and this is what allows us to assert the conclusion in this proposition.
See Remark 7.3.3 in \cite{Haase} which further outlines these differences. 
\end{remark}

\section{$H^{\infty}$-functional calculus and fractional Sobolev spaces}
\label{Sec:Hinfty}

While in the previous section we showed the existence of the holomorphic semigroup, to treat certain estimates and properties of the semigroup in later parts, we will show that the operator $\modulus{A_r}$ has an $H^\infty$-functional calculus.
Although this takes some effort to establish, we will see that our labours are well worth the fruit they bear as this is a fundamental property that affords us with the ability to overcome some of the kerfuffles that would otherwise arise in our analysis.

To recall the notion of an $H^\infty$-functional calculus, let  $T$ be an $\omega$-sectorial operator on a Hilbert space $\cH$. 
For $\mu > \omega$, let $\Psi(S_{\mu}\cap\C_+)$ denote bounded the algebra of holomorphic functions $\psi: S_{\mu}\cap\C_+ \to \C$ for each of which there exists $\alpha > 0$ and $C > 0$ satisfying
$$\modulus{\psi(\zeta)} \leq C \min \set{ \modulus{\zeta}^\alpha, \modulus{\zeta}^{-\alpha}}.$$
Integrating along the unbounded contour 
$$
\gamma 
=
\set{ r e^{\imath \theta}: \infty > r > 0} \union \set{ r e^{-\imath \theta}: 0 < r < \infty}
$$
for $\theta \in (\omega, \mu)$ defines the following functional calculus: 
$$ 
\psi(T)u 
:= 
\frac{\imath}{2\pi} \oint_{\gamma} \psi(\zeta) (\zeta - T)^{-1}u\ d\zeta,
$$
for $u\in\cH$.
This is an absolutely convergent integral due to the decay of $\psi$ and the resolvent bounds on $T$.
The operator $\psi(T)$ is a bounded operator on $\cH$.
If there exists $C > 0$ such that for every $\psi \in \Psi(S_{\mu}\cap\C_+)$ satisfies: 
$$  \norm{\psi(T)} \leq C\,\norm{\psi}_\infty,$$
then we say that $T$ has an $H^\infty(S_{\mu}\cap\C_+)$, or simply an $H^\infty$-functional calculus.
This is due to the fact that for functions $f:(S_{\mu}\cap\C_+) \union\set{0} \to \C$  bounded and holomorphic on $S_{\mu}\cap\C_+$, we can define $f(T)$ as a bounded operator on $\cH$ when $T$ has an $H^\infty$-functional calculus. 
More significantly, via the pioneering work of McIntosh in \cite{Mc86}, $T$ has an $H^\infty$-calculus if and only if the following so-called \emph{quadratic estimates}
\begin{equation}
\label{Eq:QEst}
\int_{0}^\infty \norm{\psi(tT)u}^2\ \frac{dt}{t} \simeq \norm{u}^2
\end{equation}
hold for all $u \in \close{\ran(T)}$ and some non-zero $\psi \in \Psi(S_{\mu}\cap\C_+)$, or equivalently, for all such $\psi$.
Our goal is to establish quadratic estimates~\eqref{Eq:QEst} for $T=|A_r|$ and $\cH=\Lp{2}(\Sigma;E)$ by demonstrating the $H^\infty$-functional calculus by other means.
\begin{lem}
\label{Lem:AdjProj}
The following hold:
\begin{enumerate}[(i)]
\item\label{AdjProj1}
$\chi^{\pm}(A_r)^\ast$ commute with $A_r^\ast$ on $\dom(A_r^\ast)$, 
\item\label{AdjProj2} 
$\dom(A_r^\ast) = \SobH{1}(\Sigma;E)$ and $\chi^{\pm}(A_r^\ast)$ are bounded projectors,
\item\label{AdjProj3} 
$\chi^{\pm}(A_r)^\ast = \chi^{\pm}(A_r^\ast)$ and $\sgn(A_r)^*=\sgn(A_r^*)$, and 
\item\label{AdjProj4} 
$\chi^\pm(A_r^*)|_{\chi^\pm(A_r)\Lp{2}(\Sigma;E)}:\chi^\pm(A_r)\Lp{2}(\Sigma;E) \to \chi^\pm(A_r^*)\Lp{2}(\Sigma;E)$ and $\chi^\pm(A_r)|_{\chi^\pm(A_r^*)\Lp{2}(\Sigma;E)}:\chi^\pm(A_r^*)\Lp{2}(\Sigma;E) \to \chi^\pm(A_r)\Lp{2}(\Sigma;E)$ are isomorphisms.
\end{enumerate}
\end{lem}
\begin{proof} 
To prove \ref{AdjProj1}, note that for $u \in \dom(A_r^\ast)$ and $v \in \dom(A_r)$, 
\begin{align*}
\inprod{\chi^{\pm}(A_r)^\ast A_r^\ast u, v} 
&= 
\inprod{ A_r^\ast u, \chi^{\pm}(A_r) v}  = \inprod{ u, A_r \chi^{\pm}(A_r) v}\\
&=
\inprod{ u, \chi^{\pm}(A_r) A_r v} = \inprod{ \chi^{\pm}(A_r)^\ast u, A_r v},
\end{align*}
which shows that $\chi^{\pm}(A_r)^\ast u \in \dom(A_r^\ast)$.
On using the density of $\dom(A_r)$, we obtain \ref{AdjProj1}.

Assertion~\ref{AdjProj2} simply follows by applying the theory from section~\ref{Sec:OpTheory} to the operator $A_r^\ast$ in place of $A_r$ which is again elliptic.
In particular, by elliptic regularity theory, we have that $\dom(A_r^\ast) = \SobH{1}(\Sigma;E)$ and we can apply Proposition~\ref{Prop:AProj} to assert that $\chi^{\pm}(A_r^\ast)$ are bounded projectors.

We prove \ref{AdjProj3} for $\chi^{+}(A_r)$. 
From \eqref{eq:chi+Ar} we obtain
\begin{align*}
\chi^{+}(A_r)u 
&= \frac{\imath}{2\pi} \int_{\set{\imath r: \infty > r > -\infty}} \zeta^{-1} A_r(A_r - \zeta)^{-1}u\ d\zeta  \\ 
&= \frac{\imath}{2\pi} \int_{\infty}^{-\infty} (\imath r)^{-1} A_r(A_r - \imath r)^{-1}u\ d(\imath r) \\
&= \frac{-\imath}{2\pi} \int_{-\infty}^{\infty} A_r(A_r - \imath r)^{-1}u\ \frac{dr}{r}.
\end{align*}
On sending $\imath r \mapsto (-\imath r)$ and replacing $A_r$ with $A_r^\ast$, 
\begin{align*}
\chi^{+}(A_r^\ast) u 
&= \frac{\imath}{2\pi} \int_{\set{-\imath r: -\infty< r < \infty}} \zeta^{-1} A_r^\ast (A_r^\ast - \zeta)^{-1}u\ d\zeta \\
&= \frac{\imath}{2\pi} \int_{-\infty}^\infty (-\imath r)^{-1} A_r^\ast (A_r^\ast + \imath r)^{-1}u\ d(-\imath r) \\
&= \frac{\imath}{2\pi} \int_{-\infty}^\infty A_r^\ast(A_r^\ast + \imath r)^{-1}u\ \frac{dr}{r}.
\end{align*}
Then, observe that for $u \in \dom(A_r)$ and $v \in \dom(A_r^\ast)$,
\begin{align*}
\inprod{\chi^{+}(A_r)u,v} 
&= 
-\frac{\imath}{2\pi} \int_{-\infty}^\infty \inprod{A_r(A_r - \imath r)^{-1}u,v}\ \frac{dr}{r}  \\
&= -\frac{\imath}{2\pi} \int_{-\infty}^\infty \inprod{u, A_r^\ast (A_r^\ast + \imath r)^{-1}v}\ \frac{dr}{r}.
\end{align*}
Also,
$$\inprod{u, \chi^{+}(A_r^\ast)v} = -\frac{\imath}{2\pi} \int_{-\infty}^\infty \inprod{u, A_r^\ast (A_r^\ast + \imath r)^{-1}v}\ \frac{dr}{r},$$
which shows that $\inprod{\chi^{+}(A_r)u,v} = \inprod{u, \chi^{+}(A_r^\ast)v}$ and by the density of $\dom(A_r)$ and $\dom(A_r^\ast)$ coupled with \ref{AdjProj1}, we obtain $\chi^+(A_r)^*=\chi^+(A_r^*)$.
Moreover, 
$$
\chi^-(A_r)^*=(1-\chi^+(A_r))^*=1-\chi^+(A_r^*)=\chi^-(A_r^*).
$$
Since $\sgn(A_r)=\chi^+(A_r)-\chi^-(A_r)$, \ref{AdjProj3} follows.

It suffices to prove \ref{AdjProj4} for $\chi^+(A_r^*)|_{\chi^+(A_r)\Lp{2}(\Sigma;E)}:\chi^+(A_r)\Lp{2}(\Sigma;E) \to \chi^+(A_r^*)\Lp{2}(\Sigma;E)$, the other cases being analogous.
On the one hand, we have $\Lp{2}(\Sigma;E)=\chi^+(A_r)\Lp{2}(\Sigma;E)\oplus\chi^-(A_r)\Lp{2}(\Sigma;E)$ and on the other hand,
\begin{align*}
\Lp{2}(\Sigma;E)
&=
\ker(\chi^-(A_r)^*) \oplus \chi^-(A_r)\Lp{2}(\Sigma;E)\\
&=
\ker(\chi^-(A_r^*)) \oplus \chi^-(A_r)\Lp{2}(\Sigma;E)\\
&=
\chi^+(A_r^*)\Lp{2}(\Sigma;E)\oplus\chi^-(A_r)\Lp{2}(\Sigma;E).
\end{align*}
Lemma~\ref{Lem:SubIsom} with $\cH=\Lp{2}(\Sigma;E)$, $\cH_1=\chi^+(A_r^*)\Lp{2}(\Sigma;E)$, $\cH_1'=\chi^+(A_r)\Lp{2}(\Sigma;E)$, and $\cH_2=\chi^-(A_r)\Lp{2}(\Sigma;E)$ implies \ref{AdjProj4}.
\end{proof}

With this device in hand, we  are able to assert the existence of an $H^\infty$-functional calculus for both $\modulus{A_r}$ and $\modulus{A_r^\ast}$. 

\begin{proposition}
\label{Prop:ModAFC}
The operators $\modulus{A_r}$ and $\modulus{A_r}^\ast$ are invertible elliptic pseudo-differential operators of first order and they admit an $H^\infty$-functional calculus.
Moreover, $\modulus{A_r}^\ast = \modulus{A_r^\ast}$.
\end{proposition}
\begin{proof}
By Proposition~\ref{Prop:AProj} $\modulus{A_r}$ and $\modulus{A_r}^\ast$ are pseudo-differential operators of first order.
Since they are invertible they are elliptic.

In particular, $\dom(\modulus{A_r}) = \dom({\modulus{A_r}}^\ast)=\SobH{1}(\Sigma;E)$ and $\norm{\modulus{A_r}^\ast u} \simeq \norm{\modulus{A_r}u}\simeq \norm{u}_{\SobH{1}}$.
Corollary~5.5 in \cite{AMcN} (with $s = t = 1$) yields that $\modulus{A_r}$ and $\modulus{A_r}^\ast$ have an $H^\infty$-functional calculus.

From Lemma~\ref{Lem:AdjProj}, we have that 
\begin{equation*}
\modulus{A_r}^\ast = (\sgn(A_r)A_r)^\ast = A_r^\ast\sgn(A_r)^\ast  =\sgn(A_r)^\ast  A_r^\ast = \modulus{A_r^\ast}.
\qedhere
\end{equation*}
\end{proof}

\begin{cor}
\label{Cor:SobFrac} 
\label{Cor:NegPowers}
For any $s\in\R$ we have that 
$$
\SobH{s}(\Sigma;E) = \dom(\modulus{A_r}^s) = \dom({\modulus{A_r^\ast}}^s)= \dom(({\modulus{A_r}}^s)^\ast)= \dom(({\modulus{A_r}}^\ast)^s)
$$
with 
$\norm{u}_{\SobH{s}}  \simeq \norm{\modulus{A_r}^s} \simeq \norm{{\modulus{A_r^\ast}}^s}$.
\hfill$\Box$
\end{cor}

\section{The model operator and boundedness of the extension}
\label{Sec:ModelOp}

In this section, we proceed by using a key lemma from \cite{BB12}  (c.f.\ Lemma~\ref{Lem:BdyDef}), which allows us to localise our considerations to a cylinder $[0, T) \times \Sigma$, for some $T > 0$  determined by \ref{StdSetup}.
As in Section~5 in \cite{BB12}, we analyse a simpler operator, the \emph{model operator}, on the infinite cylinder $[0, \infty) \times \Sigma$.
In later sections, we will relate this analysis back to the original operator $D$.

For convenience, define
$Z_{[0,r)} = [0, r) \times \Sigma$ for $r \in (0, \infty]$.

\begin{lem}[Lemma~2.4 in \cite{BB12}]
\label{Lem:BdyDef}
There exists a neighbourhood $U$ around $\Sigma$ in $M$, a constant $T_{c} > 0$ and a diffeomorphism $\Phi = (t,\phi): U \to Z_{[0,T_c)}$ such that 
\begin{enumerate}[(i)]
\item $\Sigma = t^{-1}(0)$,
\item $\Phi_{\Sigma} = \mathrm{id}_{\Sigma}$,
\item $d \Phi(\vec{T}) = \partial_t$ along $\Sigma$,
\item $\tau = dt$ along $\Sigma$, and 
\item ${\Phi_\ast}(\mu) = \modulus{dt} \otimes \nu$.
\end{enumerate} 
\end{lem}
In effect, this lemma allows us to localise our problem to the cylinder $Z_{[0,T_c)}$, but we routinely consider the infinite cylinder $Z_{[0,\infty)}$ in order to not worry about the boundary at the end corresponding to the value $T_c$.
Define the model operator $D_0$ associated with $D$ given the adapted boundary operator $A$ in $Z_{[0,\infty)}$ is given by
\begin{equation}
\label{Eq:ModelOp} 
D_0 = \sym_0( \partial_t + A),
\end{equation}
where we recall $\sym_0(x) = \sym_D(x,\tau(x))$.
It will also be useful to consider the following operator
\begin{equation}
\label{Eq:ModelDual}
(\sym_0^{-1}D_0)^\dagger = -(\partial_t - A^\dagstar)
\end{equation}

Now, letting $r$ denote an admissible spectral cut as before, we define
\begin{equation}
\label{Eq:SobProj}
\SobH[\pm]{s}(A_r) := \chi^{\pm}(A_r) \SobH{s}(\Sigma;E) \subset \SobH{s}(\Sigma;E) 
\end{equation} 
as well as the space 
\begin{equation}
\label{Eq:HCheck}
\checkH(A_r) := \SobH[-]{\frac{1}{2}}(A_r) \oplus \SobH[+]{-\frac{1}{2}}(A_r),
\end{equation}
normed by
\begin{equation}
\label{Eq:HCheckNorm}
\norm{u}^{2}_{\checkH(A_r)} := \norm{\chi^{-}(A_r) u}^2_{\SobH{\frac{1}{2}}} + \norm{\chi^{+}(A_r)u}^2_{\SobH{-\frac{1}{2}}}.
\end{equation}
Similarly, define the space 
\begin{equation} 
\label{Eq:HHat}
\hatH(A_r) := \SobH[-]{-\frac{1}{2}}(A_r) \oplus \SobH[+]{\frac{1}{2}}(A_r) = \checkH(-A_r).
\end{equation} 

Analogous definitions can be made with $A_r^\ast$ in place of $A_r$.
\begin{lem}
\label{Lem:DualSob} 
For $\alpha \in \R$, the $\Lp{2}$-inner product induces a perfect pairing 
$$ \inprod{\cdot,\cdot}:  \SobH[\pm]{-\alpha}(A_r^\ast) \times \SobH[\pm]{\alpha}(A_r) \to \C$$ 
and $\SobH[\pm]{-\alpha}(A_r^\ast) \cong (\SobH[\pm]{\alpha}(A_r))^\ast$.
\end{lem} 
\begin{proof}
Without loss of generality, we assume that $\alpha > 0$ and by the density of $\chi^{\pm}(A_r^\ast) \Lp{2}(\Sigma;E)$ in $\chi^{\pm}(A_r^\ast) \SobH{-\alpha}(\Sigma;E)$, assume that $u \in \chi^{\pm}(A_r^\ast)\Lp{2}(\Sigma;E)$ and $v \in \chi^{\pm}(A_r)\SobH{\alpha}(\Sigma;E)$.
Then, by Cauchy-Schwarz inequality and Corollary~\ref{Cor:SobFrac},
$$\modulus{\inprod{u,v}} = \modulus{ \inprod{\modulus{A_r}^{-\alpha}u, \modulus{A_r^\ast}^{\alpha} v} }  
	\lesssim \norm{u}_{\SobH{-\alpha}} \norm{v}_{\SobH{\alpha}}.$$
Now, note that $\inprod{\cdot,\cdot}: \SobH{-\alpha}(\Sigma;E) \times \SobH{\alpha}(\Sigma;E) \to \C$ is a duality, and in particular, 
$$ 
\norm{u}_{\SobH{-\alpha}} \simeq \sup_{0 \neq y \in \SobH{\alpha}(\Sigma;E)} \frac{\modulus{\inprod{u,y}}}{\norm{y}_{\SobH{\alpha}}}\quad\text{and}\quad  \norm{v}_{\SobH{\alpha}} \simeq \sup_{0 \neq x \in \SobH{-\alpha}(\Sigma;E)} \frac{\modulus{\inprod{x,v}}}{\norm{x}_{\SobH{-\alpha}}}.
$$
Note that by Section~4 page 156 in \cite{Kato}, we have that 
$$
\chi^{\pm}(A_r^\ast)\Lp{2}(\Sigma;E) = \ker(\chi^{\mp}(A_r^\ast)) \perp \chi^{\mp}(A_r)\Lp{2}(\Sigma;E) ,
$$ 
and therefore we obtain that $ \inprod{u,y} = \inprod{u, y + y'}$ for all $y' \in \chi^{\mp}(A_r)\SobH{\alpha}(\Sigma;E)$.
Then,
$$
\norm{u}_{\SobH{-\alpha}} \simeq \sup_{0 \neq y \in \SobH{\alpha}(\Sigma;E)} \frac{\modulus{\inprod{u,y}}}{\norm{y}_{\SobH{\alpha}}} = \sup_{0 \neq y \in \chi^{\pm}\SobH{\alpha}(\Sigma;E)} \frac{\modulus{\inprod{u,y}}}{\norm{y}_{\SobH{\alpha}}}.
$$
Similarly, we obtain the estimate
$$
\norm{v}_{\SobH{\alpha}} \simeq \sup_{0 \neq x \in \chi^{\pm}\SobH{-\alpha}(\Sigma;E)} \frac{\modulus{\inprod{x,v}}}{\norm{x}_{\SobH{-\alpha}}}.
$$
This shows that  $\inprod{\cdot,\cdot}: \chi^{\pm}(A_r^\ast) \SobH{-\alpha}(\Sigma;E) \times \chi^{\pm}(A_r) \SobH{\alpha}(\Sigma;E) \to \C$ is a duality.
Therefore $\chi^{\pm}(A_r^\ast) \SobH{-\alpha}(\Sigma;E)$ is embedded in $(\chi^{\pm}(A_r)\SobH{\alpha}(\Sigma;E))^\ast$.
However, since  $\SobH{-\alpha}(\Sigma;E)$ and $\SobH{\alpha}(\Sigma;E)$ are reflexive, and $\chi^{\pm}(A_r)$ and $\chi^{\pm}(A_r^\ast)$ are projectors, we have that their ranges $\chi^{\pm}(A_r)\SobH{-\alpha}(\Sigma;E)$ and $\chi^{\pm}(A_r)\SobH{\alpha}(\Sigma;E)$, as well as their counterparts with $A_r^\ast$ in place of $A_r$ are closed, and therefore reflexive. 
Hence the pairing induces an isomorphism $ \chi^{\pm}(A_r^\ast) \SobH{-\alpha}(\Sigma;E) \cong (\chi^{\pm}(A_r) \SobH{\alpha}(\Sigma;E))^\ast$.
\end{proof}

\begin{cor}
\label{Cor:Pairing} 
The $\Lp{2}$-inner product induces a perfect pairing 
$$
\inprod{\cdot,\cdot}:\checkH(A_r) \times \hatH(A_r^\ast) \to \C
$$
and $\checkH(A_r)^\ast \cong \hatH(A_r^\ast)$.
\hfill$\Box$
\end{cor}

Despite the fact that it may seem that the space $\checkH(A_r)$ depends on the cut $r$, in the following, we show that it is independent of the cut. 

\begin{proposition}
\label{Prop:CheckEquiv}
Let $r, q \in \R$ be two admissible spectral cuts.
Then, we have that $\checkH(A_r) = \checkH(A_q)$ with the norm equivalence $\norm{u}_{\checkH(A_r)} \simeq \norm{u}_{\checkH(A_q)}.$
\end{proposition}
\begin{proof}
Without loss of generality, assume that $q < r$ and fix  $u \in \Ck{\infty}(\Sigma;E)$ and recall that
$$ 
\norm{u}_{\checkH(A_q)}^2 
= 
\norm{\chi^{-}(A_q)u}^2_{\SobH{\frac{1}{2}}}  + \norm{\chi^{+}(A_q)u}^2_{\SobH{-\frac{1}{2}}}.
$$
On noting that $A_r, A_q, \chi^{\pm}(A_r), \chi^{\pm}(A_q)$ all commute on the domain $\dom(A_r) = \dom(A_q) = \SobH{1}(\Sigma;E)$ by functional calculus, and 
$$\chi^{-}(A_r)u
=
\chi^{-}(A_r)\cbrac{\chi^{-}(A_q)u + \chi^+(A_q)u} 
= 
\chi^{-}(A_q)u + \chi^{-}(A_r)\chi^+(A_q)u.
$$
Moreover,  $\chi^{-}(A_r)\chi^+(A_q)$ is a pseudo-differential projector of order zero and $\chi^{-}(A_r)\chi^+(A_q)\SobH{\alpha}(\Sigma;E)$ and $\chi^{-}(A_q)\SobH{\alpha}(\Sigma;E)$ are complementary subspaces for any $\alpha \in \R$. 
Note that these projectors are in general non-orthogonal in $\Lp{2}(\Sigma;E)$ and therefore, 
$$\norm{\chi^{-}(A_r)u}_{\SobH{\frac{1}{2}}} 
\simeq 
\norm{\chi^{-}(A_q)u}_{\SobH{\frac{1}{2}}} + \norm{\chi^{-}(A_r)\chi^+(A_q)u}_{\SobH{\frac{1}{2}}}.$$ 

Since the projector $\chi^{-}(A_r)\chi^+(A_q)$ has finite rank, 
\begin{align*}
\norm{\chi^{-}(A_r) \chi^{+}(A_q)u}_{\SobH{\frac{1}{2}}}
	&\simeq \norm{\chi^{-}(A_r) \chi^{+}(A_q)u}_{\SobH{-\frac{1}{2}}}.
\end{align*}

Since $q < r$, we have
$$
\chi^{+}(A_q)u
=
\chi^{+}(A_q)\cbrac{\chi^{-}(A_r)u + \chi^+(A_r)u} 
= 
\chi^{-}(A_r)\chi^{+}(A_q)u + \chi^+(A_r)u, 
$$
and similar to our earlier observation,
$$ 
\norm{\chi^{+}(A_q)u}_{\SobH{-\frac{1}{2}}} 
\simeq 
\norm{\chi^{+}(A_r)u}_{\SobH{-\frac{1}{2}}} + \norm{\chi^{-}(A_r)\chi^+(A_q)u}_{\SobH{-\frac{1}{2}}}.$$ 

Combining these estimates, 
\begin{align*}
\norm{u}_{\checkH(A_r)} &\simeq \norm{\chi^{-}(A_q)u}_{\SobH{\frac{1}{2}}} +  \norm{\chi^{-}(A_r)\chi^+(A_q)u}_{\SobH{\frac{1}{2}}} +  \norm{\chi^{+}(A_r)u}_{\SobH{-\frac{1}{2}}} \\
	&\simeq \norm{\chi^{-}(A_q)u}_{\SobH{\frac{1}{2}}} +  \norm{\chi^{-}(A_r)\chi^+(A_q)u}_{\SobH{-\frac{1}{2}}} +  \norm{\chi^{+}(A_r)u}_{\SobH{-\frac{1}{2}}} \\
	&\simeq \norm{\chi^{-}(A_q)u}_{\SobH{\frac{1}{2}}}  + \norm{\chi^{+}(A_q)u}_{\SobH{-\frac{1}{2}}} \\
	&\simeq \norm{u}_{\checkH(A_q)}.
\end{align*}
Since  $\Ck{\infty}(\Sigma;E)$ is dense in both $\checkH(A_q)$ and $\checkH(A_r)$, we have that $\checkH(A_q) = \checkH(A_r)$ with equivalence of norms.
\end{proof}

As a result of this Proposition, we simply write $\checkH(A)$ rather than $\checkH(A_r)$ and $\norm{\cdot}_{\checkH(A)}$ for the norm.

Now, recall  $T_c > 0$ given by Lemma~\ref{Lem:BdyDef}, and fix a smooth function $\eta:\R\to[0,1]$ to satisfy:
\begin{equation}
\label{Eq:Cutoff}
\eta(t) := \begin{cases} 	1,	&t \in [0,T_c/2], \\
				0, 	&t \in [2T_c/3, \infty).
	\end{cases}
\end{equation}
Then, for $u \in \Ck{\infty}(\Sigma;E)$, define the boundary extension operators by
\begin{equation}
\label{Eq:BdyExt}
(\ext u)(t,x) := \eta(t) (\exp(-t \modulus{A_r})u)(x)\quad \text{and}\quad (\ext^\ast u)(t,x) := \eta(t) (\exp(-t\modulus{A_r^\ast})u)(x) 
\end{equation}

The following lemma, which is a direct consequence of the fact that the operators $\modulus{A_r}$ and $\modulus{A_r^\ast}$ have an $H^\infty$-functional calculus, is the key observation required in the estimates we will prove.
\begin{lem}
\label{Lem:ModAFC}
For any $\alpha > 0$ and $\mu \in (0, \pi/2)$, the function
$$\zeta \mapsto \psi_\alpha(\zeta) = \zeta^\alpha \exp(-\zeta) \in  \Psi(S_\mu)$$ 
 and
$$ \int_{0}^\infty \norm{ t^\alpha \modulus{A_r}^\alpha \exp(-t \modulus{A_r}) u}^2\ \frac{dt}{t} \simeq  \norm{u}^2$$
for all $u \in \Lp{2}(\Sigma;E)$.
\end{lem}
\begin{proof}
We note that for $\modulus{\zeta} \to 0$ or $\modulus{\zeta} \to \infty$, $\modulus{\psi_\alpha(\zeta)} \to 0$ and it is easy to see that it is holomorphic. 
This shows that $\psi_\alpha \in \Psi(S_\mu)$ for any $\mu \in (0, \pi/2)$. 
By Proposition~\ref{Prop:ModAFC}, the $\omega_r$-sectorial invertible operators $\modulus{A_r}$ and $\modulus{A_r^\ast}$ have an $H^\infty$-functional calculus and so we obtain the required quadratic estimate in the conclusion.
\end{proof} 

Recall that for an operator $S$, the graph norm is given by $\norm{\cdot}_S^2 = \norm{S\cdot}^2 + \norm{\cdot}^2.$
\begin{proposition}
\label{Prop:ExtEst}
For $u \in \Ck{\infty}(\Sigma;E)$, 
$$ 
\norm{ \ext u}_{D_0} \lesssim \norm{u}_{\checkH(A)}
\quad \text{and}\quad 
\norm{\ext^\ast u}_{(\sym_0^{-1} D_0)^\dagger} \lesssim \norm{u}_{\hatH(A^\dagstar)}.
$$
\end{proposition}
\begin{proof}
First, write $D_{0,r} = \sym_0(\partial_t + A_r)$ and note that $D_0 = D_{0,r} + \sym_0 r$.
It is easy to see that $\norm{u}_{D_0} \simeq \norm{u}_{D_{0,r}}$ and therefore, we use the latter norm to establish the required estimate. 
The norm $\norm{\ext u}_{D_{0,r}}$ is then given by 
$$
\norm{\ext u}_{D_{0,r}}^2 
= 
\int_{0}^{T_c} \norm{D_{0,r}( \eta(t) \exp(-t\modulus{A_r})u)}^2\ dt + \int_{0}^{T_c} \norm{ \eta(t) \exp(-t\modulus{A_r})u}^2\ dt
$$
since $\spt \eta \subset [0, T_c]$.

We first estimate the latter term:
\begin{align*}
\int_0^{T_c} \norm{\eta(t)\exp(-t\modulus{A_r})u}^2\ dt 
&\lesssim
\int_0^\infty \norm{\exp(-t\modulus{A_r})u}^2\ dt  \\
&= \int_0^\infty \norm{ t^{\frac{1}{2}} \modulus{A_r}^{\frac{1}{2}} \exp(-t\modulus{A_r})\ \modulus{A_r}^{-\frac{1}{2}}u}^2\ \frac{dt}{t} \\
&\simeq \norm{\modulus{A_r}^{-\frac{1}{2}}u}^2 \\
&\simeq 
\norm{u}_{\SobH{-\frac{1}{2}}}^2\\
&\le
\norm{u}_{\checkH}^2,
\end{align*} 
where we used Lemma~\ref{Lem:ModAFC}.

Now, note that
\begin{align}
\partial_t (\ext u)(t,\cdot) 
&=  \eta'(t) \exp(-t\modulus{A_r}) u - \eta(t) \modulus{A_r} \exp(-t\modulus{A_r})u , \notag\\
A_r (\ext u)(t,\cdot) 
&= 
\eta(t) A_r \exp(-t\modulus{A_r})u.
\label{Eq:HEComp}  
\end{align}
It suffices to consider the two cases $u=\chi^\pm(A_r)u$. 
If $u = \chi^{-}(A_r)u$, then we have that $\modulus{A_r}u = - A_r$ and therefore,
$$ D_{0,r} (\ext u)(t, \cdot) = \sym_0(\eta'(t) - 2\eta(t)\modulus{A_r})  \exp(-t\modulus{A_r})u.$$
The norm of this quantity is dominated by two terms as follows: 
\begin{multline*}
\int_0^\infty \norm{D_{0,r} (\ext u)(t, \cdot)}^2\ dt \lesssim \\  
	\int_{0}^\infty \eta'(t)^2 \norm{\exp(-t\modulus{A_r})u}^2\ dt +  \int_0^\infty 4\eta(t)^2 \norm{\modulus{A_r}\exp(-t\modulus{A_r})u}^2\ dt.
\end{multline*}
The first term is estimated as before, where as the second term is estimated by 
\begin{align*}
\int_0^\infty \norm{\modulus{A_r}\exp(-t\modulus{A_r})u}^2\ dt
&= 
\int_0^\infty \norm{t^{\frac{1}{2}}\modulus{A_r}^{\frac{1}{2}} \exp(-t \modulus{A_r}) \modulus{A_r}^{\frac{1}{2}}u}^2 \ \frac{dt}{t}\\
&\lesssim 
\norm{\modulus{A_r}^{\frac{1}{2}}u}^2 \\
&\simeq 
\norm{u}_{\SobH{\frac{1}{2}}}^2,
\end{align*}
again by Lemma~\ref{Lem:ModAFC} and Corollary~\ref{Cor:SobFrac}.

Next, consider the case that $u = \chi^{+}(A_r)u$.
In this case, we have that $\modulus{A_r}u = A_r u$ and therefore,
$$ D_{0,r}(\ext u)(t, \cdot) = \eta'(t) \exp(-t\modulus{A_r})u.$$
Thus, 
\begin{align*}
\int_0^\infty \norm{D_{0,r} (\ext u)(t, \cdot)}^2\ dt 
&\lesssim 
\int_{0}^\infty \norm{\exp(-t\modulus{A_r})t^{\frac{1}{2}}\modulus{A_r}^{\frac{1}{2}} \modulus{A_r}^{-\frac{1}{2}}}^2\ \frac{dt}{t} \\
&\lesssim 
\norm{\modulus{A_r}^{-\frac{1}{2}} u}^2 \\
&\simeq 
\norm{u}_{\SobH{-\frac{1}{2}}}^2.
\end{align*}
On combining these estimates, we obtain the first inequality. 

Finally, to consider the operator $\ext^\ast$, as before, we write $(\sym_0^{-1}D_{0})^\dagger_r = -(\partial_t - A_r^\dagger)$ and run the exact same argument with $-A_r^\ast$ in place of $A_r$ for $v \in \Ck{\infty}(\Sigma;E)$.
This gives us that $\norm{\ext^\ast v}_{(\sym_0^{-1}D_0)^\dagger} \lesssim \norm{v}_{\checkH(-A_r^\ast)}$ and the conclusion follows on noting that $\checkH(-A_r^\ast) = \hatH(A_r^\ast)$.
\end{proof} 

\subsection{Maximal regularity and regularity in the cylinder}
\label{Sec:HighReg}

To establish the results we obtain in this subsection, the approach taken in \cite{BB12} is to use the Fourier expansion afforded by the selfadjointness of the boundary operator and use  ODE theory to obtain the desired results.
We are unable to take this approach in our setting since our boundary operator is non-selfadjoint in general and hence the eigenspaces need not span all of $\Lp{2}(\Sigma;E)$ and the generalised eigenspaces need not be orthogonal.
We use Banach-valued ODE theory instead.

Fix $\rho \in (0, T_c)$, where $T_c$ is from Lemma~\ref{Lem:BdyDef}. 
Given $f \in \Lp{2}(Z_{[0,\rho]}; E)$, we consider the following Banach-valued Cauchy problem 
\begin{equation}
\label{Eq:RegODE}
\begin{aligned} 
&\partial_t W(t;f) + \modulus{A_r} W(t;f) = f(t),\\
&\lim_{t \to 0} W(t;f) = 0.
\end{aligned}
\end{equation}

\begin{lem}
\label{Lem:WReg}
For a given $f \in \Lp{2}(Z_{[0,\rho]}; E))$, a unique solution $W(t; f): \Lp{2}(\Sigma;E) \to \Ck{\infty}(\Sigma;E)$, $ t > 0$ to \eqref{Eq:RegODE} exists and it is given by 
\begin{equation}
\label{eq:Duhamel}
W(t;f) = \int_{0}^t \exp(-(t - s)\modulus{A_r}) f(s)\ ds.
\end{equation}
It satisfies the estimate
\begin{equation}
\label{eq:DuhamelEst}
\int_{0}^\rho \norm{ \partial_t W(t;f)}^2_{\Lp{2}(\Sigma)}\, dt + \int_{0}^\rho \norm{ \modulus{A_r} W(t; f)}^2_{\Lp{2}(\Sigma)} \, dt \lesssim \int_{0}^\rho \norm{f(t)}^2_{\Lp{2}(\Sigma)} \, dt.
\end{equation}
The implicit constant depends on $T_c$ but not on $\rho$.
\end{lem}
\begin{proof}
First note the identification $\Lp{2}(Z_{[0,\rho]}; E) \cong \Lp{2}([0, \rho]; \Lp{2}(\Sigma;E))$. 
Since $\Lp{2}(\Sigma;E)$ is a Hilbert space, it is $\gamma$-convex, which means that the Hilbert transform is bounded on $\Lp{p}(\R; \Lp{2}(\Sigma;E))$ for $p \in (1, \infty)$, cf.\ Remark 2.7 in \cite{DV87}. 
Also, we have that $\modulus{A_r}$ has bounded holomorphic functional calculus by Proposition~\ref{Prop:ModAFC}.
This is equivalent to the fact that there exist $K \geq 1$ and $\theta > 0$ so that $\norm{\modulus{A_r}^{\imath s}} \leq K e^{\theta \modulus{s}}$ for all $s \in \R$ by the theorem in Section~8 in \cite{Mc86}.
These facts verify the hypotheses of Theorem~2.1 in \cite{GS91} and therefore, we obtain that the solution $W(t;f)$ to \eqref{Eq:RegODE} exists uniquely for a given $f \in \Lp{2}(Z_{[0,\rho]}; E)$ satisfying the estimate \eqref{eq:DuhamelEst}.
Formula~\eqref{eq:Duhamel} is standard.
\end{proof}

Using this,  we define the following operator that is the crucial device for our treatment of regularity.
\begin{equation}
\label{Eq:RegOp}
S_{0,r}u(t) :=  \int_{0}^t \exp(-(t-s)\modulus{A_r}) \chi^+(A_r) u(s)\ ds - \int_{t}^\rho \exp(-(s - t)\modulus{A_r})  \chi^{-}(A_r) u(s)\ ds.
\end{equation}
Throughout, let us fix $D_{0,r} := \sym_0(\partial_t + A_r)$ and $(C_\rho u)(s) := u(\rho -s)$.

\begin{lem}
\label{Lem:SProp} 
The operator $S_{0,r}$ satisfies the following:
\begin{enumerate}[(i)]
\item 
\label{Lem:SProp1}
$S_{0,r} u(t) = W(t;  \chi^+(A_r) u) - W(\rho - t;  \chi^{-}(A_r) C_\rho u)$,
\item 
\label{Lem:SProp2}
$\chi^{+}(A_r) (S_{0,r} u)(0)  = \chi^{-}(A_r)(S_{0,r} u)(\rho) = 0$, and 
\item 
\label{Lem:SProp3}
$\sym_0^{-1} D_{0,r} S_{0,r}  = 1$.
\end{enumerate}
\end{lem} 
\begin{proof}
Applying the coordinate transformation $s \to \rho-s$ to the second term in the definition of $S_{0,r}$ in \eqref{Eq:RegOp} yields the formula in \ref{Lem:SProp1}.
Claim~\ref{Lem:SProp2} follows simply from the fact that the projector commutes with the functional calculus of the operator.

For \ref{Lem:SProp3}, write $u = u^+ + u^-$, where $u^\pm = \chi^\pm(A_r)u$.
Then it suffices to show that $\sym_0^{-1} D_{0,r}S_{0,r} u^\pm = u^\pm$. 
Noting that $S_{0,r}u^+(t) = W(t;  u^+)$ which solves \eqref{Eq:RegODE}, and since $\modulus{A_r} = A_r \sgn(A_r)$,
$$ \partial_t S_{0,r}u^+(t) = - A_r \sgn(A_r) S_{0,r}u^{+}(t) + u^+(t) = - A_r S_{0,r}u^+(t) + u^+(t).$$
This is exactly that $\sym_0^{-1} D_{0,r}S_{0,r} u^+(t) = u^+(t)$.
Now, for $u^-$, note that
\begin{align*}
\partial_t W(\rho -t;  C_\rho u^{-}) &=  + A_r \sgn(A_r) W(\rho -t;  C_\rho u^{-}) -  C_\rho u^{-}(\rho - t) \\
	&=  - A_r W(\rho -t;  C_\rho u^{-}) -   u^{-}(t)
\end{align*}
since $\chi^-(A_r)\sgn(A_r) = -\chi^-(A_r).$
Then, 
\begin{equation*}
\sym_0^{-1}D_{0,r}S_{0,r}u^-(t) = -(\partial_t + A_r)W(\rho -t;  C_\rho u^-) = u^-(t).
\qedhere
\end{equation*}
\end{proof}

Furthermore, we note that $S_{0,r}$ increases regularity.

\begin{lem}
\label{Lem:SReg}
The operator $S_{0,r}$ maps $\SobH{k}(Z_{[0,\rho]}) \to \SobH{k+1}(Z_{[0,\rho]}; E)$ boundedly, with the bound independent of $\rho$ (but dependent on $T_c$). 
\end{lem}
\begin{proof}
First, we consider the case $k = 0$, and note that 
$$\norm{u}_{\SobH{1}(Z_{[0,\rho]})}^2 \simeq \int_{0}^\rho \norm{\partial_t u}^2_{\Lp{2}(\Sigma)}\ dt + \int_{0}^\rho \norm{\modulus{A_r} u}^2_{\Lp{2}(\Sigma)}\ dt.$$
Now, 
\begin{align*}
\|&S_{0,r} u\|_{\SobH{1}(Z_{[0,\rho]})}^2 \\
&\lesssim  
\int_0^\rho \norm{\partial_t W(t, \chi^+(A_r) u)}^2_{\Lp{2}(\Sigma)}\ dt + \int_0^\rho \norm{\modulus{A_r} W(t, \chi^+(A_r) u)}^2_{\Lp{2}(\Sigma)}\ dt \\ 
&\quad+\int_0^\rho \norm{\partial_t W(\rho - t, \chi^-(A_r) C_\rho u)}^2_{\Lp{2}(\Sigma)}\ dt 
+ \int_0^\rho \norm{\modulus{A_r} W(\rho - t, \chi^-(A_r) C_\rho u)}^2_{\Lp{2}(\Sigma)}\ dt \\ 
&\lesssim 
\int_{0}^\rho \norm{ \chi^+(A_r) u}^2_{\Lp{2}(\Sigma)}\ dt + \int_{0}^\rho \norm{ \chi^-(A_r) C_\rho u}^2_{\Lp{2}(\Sigma)}\ dt \\
&\lesssim 
\int_0^\rho \norm{u}^2_{\Lp{2}(\Sigma)}\ dt 
= 
\norm{u}_{\Lp{2}(Z_{[0,\rho]})}^2, 
\end{align*}
where  the first inequality follows from \ref{Lem:SProp1} in Lemma~\ref{Lem:SProp}, the second from equation \eqref{eq:DuhamelEst} in Lemma~\ref{Lem:WReg}, and the third from noting that $C_\rho$ is an isometry of $\Lp{2}(Z_{[0,\rho]}; E)$.

For higher regularity, fix $f \in \dom(\modulus{A_r}^l)$.
Differentiating \eqref{eq:Duhamel} we get that
\begin{equation}
\label{Eq:HighReg1} 
\partial_t^l W(t; f) =  (-1)^{l} W(t; \modulus{A_r}^l f) + \sum_{m=0}^{l-1} \partial_t^{l-1-m} \modulus{A_r}^{m} f(t)
\end{equation} 
and
\begin{equation}
\label{Eq:HighReg2} 
\partial_t^l W(\rho - t; f) =  (-1)^{l+1} W(\rho - t; \modulus{A_r}^l f) + \sum_{m=0}^{l-1} \partial_t^{l-1-m} \modulus{A_r}^{m} f(t).
\end{equation}
Moreover, since the metric on $Z_{[0,\rho]}$ is of product type, we have that
$$
\norm{u}_{\SobH{k}(Z_{[0,\rho]})} 
\simeq 
\sum_{l=0}^k \int_0^\rho \norm{\partial_t^{k-l} u}_{\SobH{l}(\Sigma)} \, dt 
\simeq 
\sum_{l = 0}^k \int_0^\rho  \norm{ \modulus{A_r}^l \partial_t^{k-l} u }_{\Lp{2}(\Sigma)} \, dt.
$$
Therefore, when $f \in \SobH{k}(Z_{[0,\rho]}; E)$, then $f \in \cap_{l=0}^k  \dom( \partial_t^{k-l} \modulus{A_r}^{l})$. 
We show that whenever $0 \leq l \leq k+1$, 
$$\norm{\partial_t^l \modulus{A_r}^{k+1-l} W(\cdot; f)}_{\Lp{2}(Z_{[0,\rho]})} \lesssim \norm{f}_{\SobH{k}(Z_{[0,\rho]})}.$$
For $l = 0$, 
\begin{align*} 
\norm{\partial_t^l \modulus{A_r}^{k+1-l} W(\cdot ;f)}_{\Lp{2}(Z_{[0,\rho]})} 
&= 
\norm{\modulus{A_r} W(\cdot; \modulus{A_r}^{k} f)}_{\Lp{2}(Z_{[0,\rho]})} \\
&\lesssim 
\norm{\modulus{A_r}^{k} f}_{\Lp{2}(Z_{[0,\rho]})} \\
&\lesssim 
\norm{f}_{\SobH{k}(Z_{[0,\rho]})}
\end{align*}
by functional calculus and \eqref{eq:DuhamelEst} in Lemma~\ref{Lem:WReg}.
When $l \geq 1$,
$$
\partial_t^l \modulus{A_r}^{k+1 - l} W(t; f) 
= 
(-1)^{l} \modulus{A_r} W(t; \modulus{A_r}^k f) + \sum_{m=0}^{l-1} \partial_t^{l-1-m}\modulus{A_r}^{k+1+m - l} f(t)$$
by \eqref{Eq:HighReg1}.
Therefore, 
\begin{align*} 
\|\partial_t^l &\modulus{A_r}^{k+1 - l} W(\cdot; f)\|_{\Lp{2}{(Z_{[0,\rho]}})} \\
&\le 
\norm{(\modulus{A_r} W(t; \modulus{A_r}^k f)}_{\Lp{2}{(Z_{[0,\rho]}})} + \sum_{m=0}^{l-1} \norm{\partial_t^{l-1-m}\modulus{A_r}^{k+1+m - l} f(t)}_{\Lp{2}(Z_{[0,\rho]})} \\
&\lesssim  \norm{\modulus{A_r}^k f}_{\Lp{2}(Z_{[0,\rho]})} + \sum_{m=0}^{l-1} \norm{f}_{\SobH{k}(Z_{[0,\rho]})} \\
&\simeq \norm{f}_{\SobH{k}(Z_{[0,\rho]})},
\end{align*}
via \eqref{eq:DuhamelEst} in Lemma~\ref{Lem:WReg} and since $(l - 1 - m) + (k+1+m - l)=k$.

Replicating this argument and using \eqref{Eq:HighReg2} instead of \eqref{Eq:HighReg1} and noting that $C_\rho$ is bounded on $\SobH{k}(\Sigma;E)$, we get a similar estimate for $t \mapsto W(C_\rho(\cdot), f)$.
Thus, we obtain 
\begin{align}
\label{Eq:HighReg3}
&\norm{W(\cdot, f)}_{\SobH{k+1}(Z_{[0,\rho]})} \lesssim \norm{f}_{\SobH{k}(Z_{[0,\rho]})},\ \text{and} \notag \\
&\norm{W(C_\rho(\cdot), f)}_{\SobH{k+1}(Z_{[0,\rho]})} \lesssim \norm{f}_{\SobH{k}(Z_{[0,\rho]})}.
\end{align}

To finish the proof, let $u \in \SobH{k}(Z_{[0,\rho]})$ and note
\begin{align*}
\norm{S_{0,r} u}_{\SobH{k+1}(Z_{[0,\rho]})} &\leq \norm{W(\cdot,  \chi^+(A_r)u)}_{\SobH{k+1}(Z_{[0,\rho]})}  \\
	&\qquad\qquad\qquad+ \norm{W(C_\rho(\cdot),  \chi^-(A_r)C_\rho u)}_{\SobH{k+1}(Z_{[0,\rho]})}  \\
	&\lesssim \norm{\chi^+(A_r)u}_{\SobH{k}(Z_{[0,\rho]})}  + \norm{\chi^-(A_r) u}_{\SobH{k}(Z_{[0,\rho]})} \\
	&\simeq \norm{u}_{\SobH{k}(Z_{[0,\rho]})}, 
\end{align*}
where we use \ref{Lem:SProp1} in Lemma~\ref{Lem:SProp} in the first inequality, the estimate \eqref{Eq:HighReg3} in the second, the  fact that $C_\rho$ commutes with $\chi^{\pm}(A_r)$ and is bounded on $\SobH{k}(Z_{[0,\rho]})$ in the third, and the fact that $\chi^{\pm}(A_r)$ are pseudo-differential operators of order zero in the last.
\end{proof}

Define
\begin{equation}
\label{Eq:B0}
B_0 := H_{-}^{\frac{1}{2}}(A) \oplus H_{+}^{\frac{1}{2}}(A)
\end{equation}
where we recall that $H_{\pm}^\frac{1}{2}(A) = \chi^{\pm}(A_r)\SobH{\frac{1}{2}}(\Sigma;E)$.
Moreover, for $k \geq 1$, define
$$\SobH{k}(Z_{[0,\rho]}; E; B_0) := \set{u \in \SobH{k}(Z_{[0,\rho]}; E): (u(0), u(\rho)) \in B_0}.$$

In the following proposition, we use some facts from the theory of Banach-valued calculus. 
A standard reference is the  book \cite{CH}, but an excellent overview of this topic is contained in \cite{Kreuter}.

\begin{proposition}
\label{Prop:RegInv}
Regard $D_{0,r}$ as an operator on $\Lp{2}(Z_{[0,\rho]}; E)$.
Then:
\begin{enumerate}[(i)] 
\item
\label{Prop:RegInv1}  
For all $u \in \dom((D_{0,r})_{\max})$ satisfying $\chi^{-}(A_r)(u(\rho)) = 0$, 
$$ (1 - S_{0,r}\sym_0^{-1} D_{0,r})u(t,\cdot) = \exp(-t \modulus{A_r})(\chi^{+}(A_r)u(0)).$$
\item 
\label{Prop:RegInv2} 
For all $k \geq 0$, the operator
$$ D_{0,r}: \SobH{k+1}(Z_{[0,\rho]}; E; B_0) \to \SobH{k}(Z_{[0,\rho]};F)$$
is an isomorphism with inverse $S_{0,r}\sym_0^{-1}$.
\end{enumerate}
\end{proposition}
\begin{proof}
Write $D_{0,r}u(s) = D_{0,r}u^-(s) + D_{0,r}u^{+}(s)$ where $u^{\pm} = \chi^{\pm}(A_r)u$ and note 
$$ \sym_0^{-1} D_{0,r} u^\pm(s) = \cbrac{\partial_s u^\pm(s) \pm \modulus{A_r} u^\pm(s)}.$$
Moreover, from 
$$S_{0,r}\sym_0^{-1} D_{0,r} u^+(t) 
	= \int_{0}^t \exp({-(t-s)\modulus{A_r}})(\partial_s u^+(s) + \modulus{A_r}u^+(s))\ ds$$
and
$$ \partial_s (\exp({-(t-s)\modulus{A_r}}) u^+(s)) = \modulus{A_r} \exp({-(t-s)\modulus{A_r}}) u^+(s) + \exp({-(t-s)\modulus{A_r}}) \partial_s u^+(s),$$
we obtain 
\begin{align}
S_{0,r}\sym_0^{-1} D_{0,r} u^+(t) &= 
\int_0^t \partial_s (\exp({-(t-s)\modulus{A_r}}) u^+(s))\ ds \notag\\
&= 
u^+(t) + \exp({-t\modulus{A_r}})(u^+(0)),
\label{eq:SSDU+}
\end{align}
via the Banach-valued fundamental theorem of calculus.

For the remaining term, via \ref{Lem:SProp1}  in Lemma~\ref{Lem:SProp}, first observe that
\begin{equation}
S_{0,r}\sym_0^{-1} D_{0,r} u^{-}(t) 
= 
- \int_0^{\rho - t} \exp({-(\rho - t - s)\modulus{A_r}}) C_\rho(\partial_s u^-(s) - \modulus{A_r}u^{-}(s))\ ds.
\label{eq:u-}
\end{equation}
Also,
\begin{align*}
 &\partial_s (\exp({-(\rho - t - s)\modulus{A_r}})u^-(\rho -s))  \\ 
	&\qquad= \modulus{A_r} \exp({-(\rho - t - s)\modulus{A_r}})u^{-}(\rho -s) - \exp({-(\rho - t - s)\modulus{A_r}}) \partial_{s'}u^{-}(s')\rest{s' = \rho -t}   \\
	&\qquad=\modulus{A_r} \exp({-(\rho - t - s)\modulus{A_r}})(C_\rho u^{-})(s) - \exp({-(\rho - t - s)\modulus{A_r}}) (C_\rho \partial_s u^{-})(s).
\end{align*} 
On substituting this into \eqref{eq:u-} and  again using the fundamental theorem of calculus for Banach-valued functions,
\begin{align} 
S_{0,r}\sym_0^{-1} D_{0,r} u^-(t)  &= - \int_0^{\rho -t} \Big( \modulus{A_r} \exp({-(\rho -t -s)\modulus{A_r}}) (C_\rho u^-)(s) \notag \\
	&\qquad\qquad\qquad- \partial_s (\exp({-(\rho - t - s)\modulus{A_r}})(C_\rho u^-)(s)) \notag\\
	&\qquad\qquad\qquad- \modulus{A_r} \exp({-(\rho -t -s)\modulus{A_r}}) (C_\rho u^-)(s) \Big)\ ds \notag\\
	&= u^{-}(t) - \exp(-{(\rho - t)\modulus{A_r}}) u^-(\rho)\notag \\ 
	&= u^{-}(t),
\label{eq:SSDU-}
\end{align}
where the penultimate equality follows from the assumption $u^-(\rho) = 0$.
The formula in the conclusion then follows by adding \eqref{eq:SSDU+} and \eqref{eq:SSDU-}.

Noting $S_{0,r}: \SobH{k}(Z_{[0,\rho]}; E) \to\SobH{k+1}(Z_{[0,\rho]}; E)$ boundedly by Lemma~\ref{Lem:SReg}, $\sym_0^{-1}$ is smooth, and combining with what we have just proved, the second assertion follows.
\end{proof}

\begin{remark}
For a general $u$, without assuming that $\chi^-(A_r)(u(\rho)) = 0$, the same argument would yield the equation: 
$$ (1 - S_{0,r}\sym_0^{-1} D_{0,r})u = \exp({-t \modulus{A_r}})(\chi^{+}(A_r)u(0)) - \exp({-(\rho-t) \modulus{A_r}})(\chi^{-}(A_r)u(\rho)).$$
\end{remark}

\section{The maximal domain}
\label{Sec:MaximalDom}

In this section, we use the results from Section~\ref{Sec:ModelOp} to analyse the maximal domain of $D$.
This done via an adaption of Lemma~4.1 in \cite{BB12} (c.f.\ Lemma~\ref{Lem:Lem4.1BB}) to obtain normal form for the operators $D$ and $D^\dagger$.
Using this, in Lemma~\ref{Lem:ExtBd}, we show that the extension operators are bounded for the operator $D$ in place of $D_0$.
There, we arguing in a similar way to Lemma~6.1 in \cite{BB12}, but using the $H^\infty$-functional calculus instead of Fourier series. 
This allows us to obtain \ref{Lem:RestBd}, which demonstrates the boundedness of the boundary restriction map. 
Moreover, in order to move between the operators $D$ and $D_0$ in the later analysis, we prove a relatively boundedness result in Lemma~\ref{Lem:D-D0}.

The argument here is in the spirit  of Lemma~5.2 in \cite{BB12}, which crucially relies on Lemma~\ref{Lem:BdyEll}.
While this latter lemma is similar in conclusion to Lemma~5.1 in \cite{BB12}, its proof deviates significantly due to the fact that $A$ is not necessarily selfadjoint in our setting.
The proof here obtained via maximal regularity considerations on the cylinder, which we demonstrated in Subsection~\ref{Sec:HighReg}.
Having gathered the required ingredients, we prove Theorems~\ref{Thm:Ell} and \ref{Thm:HighReg}. 

\begin{lem}
\label{Lem:GreensDensity}
The space $\Ck[c]{\infty}(M;E)$ is dense in $\dom(D_{\max})$ and $\Ck[c]{\infty}(M;F)$ is dense in $\dom(D^\dagger_{\max})$.
Moreover, the operators $D$ and $D^\dagger$ satisfy the following Green's formula: 
\begin{equation} 
\label{Eq:Greens} 
\inprod{D u, v}_{\Lp{2}(M;F)} - \inprod{u, D^\dagger v}_{\Lp{2}(M;E)} = -\inprod{\sym_0 u\rest{\Sigma}, v\rest{\Sigma}}_{\Lp{2}(\Sigma;F)},
\end{equation} 
for all $u \in \Ck[c]{\infty}(M;E)$ and $v \in \Ck[c]{\infty}(M;F)$.
\end{lem} 

\begin{proof}
The proof is identical to those of Lemmas~2.6 and 6.4 in \cite{BB12}.
\end{proof}

\begin{lem}[Lemma~4.1 in \cite{BB12}]
\label{Lem:Lem4.1BB}
In coordinates $\Phi$ given by \eqref{Lem:BdyDef}, over the cylinder $Z_{[0,T_c)}$, we have that
\begin{equation}
\label{Eq:NormalForm}
\begin{aligned}
D &= \sym_t (\partial_t + A + R_t), \\ 
D^\dagger &= -\sym_t^\ast ( \partial_t + \tilde{A} + \tilde{R}_t),
\end{aligned} 
\end{equation}
for any pair of adapted boundary operators $A$ and $\tilde{A}$ for $D$ and $D^\dagger$, respectively.
The remainder terms are pseudo-differential operators of order at most one, with coefficients depending smoothly on $t$ and which satisfy the estimates:
\begin{equation}
\label{Eq:NormalRemainders}
\begin{aligned}
&\norm{R_t u}_{\Lp{2}(\Sigma)} \lesssim t\norm{A u}_{\Lp{2}(\Sigma)} + \norm{u}_{\Lp{2}(\Sigma)}, \text{ and} \\
&\norm{\tilde{R}_t v}_{\Lp{2}(\Sigma)} \lesssim t\norm{\tilde{A} v}_{\Lp{2}(\Sigma)} + \norm{v}_{\Lp{2}(\Sigma)}.
\end{aligned} 
\end{equation}
for $u \in \Ck{\infty}(\Sigma;E)$ and $v \in \Ck{\infty}(\Sigma;F)$.
\end{lem} 
\begin{proof}
Write $D = \sym_t( \partial_t + P_t)$ where $P_t$ are a family of elliptic differential operators of order one whose coefficients depend smoothly on $t$. 
Therefore, writing $R_t = P_t - A$ given such an $A$, we note that this is an operator whose coefficients depend smoothly on $t$ and is of order at most one. 
Noting that $R_0$ is of order $0$ and $\Sigma$ is closed,
$$ 
\norm{R_t u}_{\Lp{2}(\Sigma)} \lesssim t \norm{u}_{\SobH{1}(\Sigma)} + \norm{u}_{\Lp{2}(\Sigma)}  \lesssim t\norm{Au}_{\Lp{2}(\Sigma)} + \norm{u}_{\Lp{2}(\Sigma)}
$$
where the second inequality follows from elliptic estimates.
The conclusion for $\tilde{R}_t$ follows in exactly the same way.
\end{proof}

Given that we have the normal form $D = \sym_t(\partial_t + A + R_t)$ we obtain
\begin{equation}
\label{Eq:DAdj}
(\sym^{-1}D)^\dagger = -(\partial_t - A^\dagstar - R_t^\dagger).
\end{equation}

\begin{lem}
\label{Lem:ExtBd} 
For all $u \in \Ck{\infty}(\Sigma;E)$, the sections $\ext u, \ext^\ast u \in \dom(D_{\max})$ and
$$
\norm{\ext u}_{D} \lesssim \norm{u}_{\checkH(A)}
\quad\text{and}\quad 
\norm{\ext^\ast u}_{(\sym^{-1}D)^\dagger} \lesssim \norm{u}_{\hatH(A^\dagstar)}.
$$
\end{lem}
\begin{proof}
Via a simple calculation, we obtain that 
$$ \norm{D(\ext u)}_{\Lp{2}(M)} \lesssim \norm{\ext u}_{D_0} + \norm{R_t(\ext u)}_{\Lp{2}(Z_{[0,T_c)})}.$$
Since we have estimated the first term on the right by Proposition~\ref{Prop:ExtEst}, it suffices to bound the latter term. 
Now, note that 
$$ \norm{R_t \ext u}_{\Lp{2}(\Sigma)} \lesssim \norm{t A \ext u} + \norm{ \ext u}\simeq \norm{t \modulus{A_r} \ext u} + \norm{\ext u}$$
and therefore, 
\begin{align*}
\int_{0}^\infty \norm{R_t \ext u}^2_{\Lp{2}(\Sigma)} \ dt &\lesssim  \int_0^\infty \eta(t)^2 \norm{t \modulus{A_r} \exp(-t \modulus{A_r})u}^2_{\Lp{2}(\Sigma)}\ dt\\
	&\qquad\qquad+\int_{0}^\infty \eta(t)^2 \norm{\exp(-t\modulus{A_r})u}_{\Lp{2}(\Sigma)}^2\ dt. \\
	&\lesssim \int_{0}^\infty \norm{t^{\frac{3}{2}} \modulus{A_r}^{\frac{3}{2}} \exp(-t \modulus{A_r}) \modulus{A_r}^{-\frac{1}{2}}u}_{\Lp{2}(\Sigma)}^2\ \frac{dt}{t} \\
	&\qquad\qquad+\int_{0}^\infty \norm{t^\frac{1}{2}\modulus{A_r}^{\frac{1}{2}} \exp(-t \modulus{A_r}) \modulus{A_r}^{-\frac{1}{2}}u}_{\Lp{2}(\Sigma)}^2\ \frac{dt}{t} \\
	&\lesssim \norm{u}_{\SobH{-\frac{1}{2}}(\Sigma)}^2 + \norm{u}_{\SobH{-\frac{1}{2}}(\Sigma)}^2 \\
	&\lesssim \norm{u}_{\checkH(A)}^2,
\end{align*}
where the penultimate inequality follows from Lemma~\ref{Lem:ModAFC}.
The estimate for $\norm{\ext^\ast u}_{(\sym^{-1}D)^\ast}$ is similar on replacing $A$ with $-A^\ast$ and $\modulus{A_r}$ with $\modulus{A_r^\ast}$ and on noting that $\checkH(-A^\ast) = \hatH(A^\dagstar)$.
\end{proof}

Using this lemma,  we obtain the following. 
\begin{lem}
\label{Lem:RestBd}
For all $u \in \Ck[c]{\infty}(Z_{[0,T_c)}; E)$, we have the bounds
$$ 
\norm{u \rest{\Sigma}}_{\checkH(A)} \lesssim \norm{u}_{D}
\quad\text{and}\quad 
\norm{u \rest{\Sigma}}_{\hatH(A^\dagstar)} \lesssim \norm{u}_{(\sym^{-1}D)^\dagger}.
$$
\end{lem}
\begin{proof}
From the Green's formula \eqref{Eq:Greens}, we have for $\phi\in \Ck[c]{\infty}(M;E)$ and $\psi\in \Ck[c]{\infty}(M;F)$,
$$
\inprod{D \phi, \psi}_{\Lp{2}(M)} - \inprod{\phi, {D}^\dagger\psi}_{\Lp{2}(M)} = - \inprod{\sym \phi, \psi}_{\Lp{2}(\Sigma)}.$$
On fixing $v \in \Ck{\infty}(\Sigma;E)$ and setting $\psi = (\sym^{-1})^\ast \ext^\ast v$ and $\phi = u$,  we have that
$$
\inprod{\sym^{-1} D u, \ext^\ast v}_{\Lp{2}(M)} - \inprod{u, (\sym^{-1}D)^\ast \ext^\ast v}_{\Lp{2}(M)} = - \inprod{u\rest{\Sigma}, v}_{\Lp{2}(\Sigma)},$$
since $v = (\ext^\ast v)\rest{\Sigma}$. 
Therefore,
\begin{align*}
\modulus{\inprod{u\rest{\Sigma}, v}_{\Lp{2}(\Sigma;E)}} &\lesssim \norm{(\sym^{-1}D)u} \norm{\ext^\ast v} + \norm{u} \norm{(\sym^{-1}D)^\dagger \ext^\ast v} \\ 
	&\lesssim \norm{(\sym^{-1}D)u} \norm{\ext^\ast v}_{(\sym^{-1}D)^\dagger} + \norm{u} \norm{\ext^\ast v}_{(\sym^{-1}D)^\dagger} \\
	&\lesssim \norm{u}_D \norm{\ext^\ast v}_{(\sym^{-1}D)^\dagger}	\\
	&\lesssim \norm{u}_D \norm{v}_{\hatH(A^\dagstar)},
\end{align*}
where the second inequality follows from the fact that $\norm{\cdot}_{\Lp{2}} \lesssim \norm{\cdot}_{(\sym^{-1}D)^\dagger}$ and the last inequality from Lemma~\ref{Lem:ExtBd}. 

By taking a supremum over $v \in \Ck{\infty}(\Sigma;E)$, which is dense in $\hatH(A^\dagstar)$, and since this latter space is dual to $\checkH(A)$ via Corollary~\ref{Cor:Pairing}, we obtain the first estimate.
The remaining estimate is obtained by a similar calculation after making a choice of $\phi = \ext v$ and $\psi = (\sym^{-1})^\ast u$. 
\end{proof}

This demonstrates the well-posedness of boundary value problems and shows that the definition of $\checkH(A)$ is the correct function space for considering boundary value problems.
Now we set ourselves the task of identifying the maximal domain. 
Unlike in selfadjoint case, this situation is slightly more subtle and later we are forced to calculate via a different, but equivalent norm on $\Lp{2}(\Sigma;E)$.

\begin{lem} 
\label{Lem:BdyEll}
For  all $\rho \in (0,T_c)$ and $ u \in \Ck[c]{\infty}(Z_{[0,\rho)};E)$ with $\chi^+(A_r)(u(0)) = 0$, we have that 
$$ \norm{u}_{\SobH{1}}^2 \lesssim \norm{ (\partial_t +A)u}^2 + \norm{u}^2.$$
The implicit constant depends on $T_c$ but not on $\rho$.
\end{lem}
\begin{proof}
This statement follows from the  inverse $S_{0,r}$ which we constructed  for \mbox{$(\partial_t + A_r)$} on $Z_{[0,\rho]}$ in Subsection~\ref{Sec:HighReg}.
This operator satisfies $S_{0,r}: \SobH{k}(Z_{[0,\rho]}; E) \to \SobH{k+1}(Z_{[0,\rho]};E)$ and  in Proposition~\ref{Prop:RegInv}, we obtain that
$$ (1 - S_{0,r}(\partial_t + A_r))v(t) = e^{-t\modulus{A_r}}(\chi^+(A_r) v(0))$$
whenever $\chi^-(A_r)(v(0)), \chi^+(A_r) (v(\rho)) \in \SobH{\frac{1}{2}}(\Sigma;E)$.
Since $\spt u \subset Z_{[0,\rho)}$ the desired conclusion follows.
\end{proof} 

Since we need to consider the maximal domain of $D$ on subsets of $X \subset M$, we use the notation $\dom(D_{\max}; X)$ to denote the maximal domain of $D$ on $X$.

\begin{lem}
\label{Lem:D-D0}
There exists $C_1 > 0$ such that for all $\rho \in (0, T_c)$ and $u \in \Ck{\infty}(Z_{[0,\rho)}; E)$,
$$ \norm{(D - D_0)u}_{\Lp{2}(Z_{[0,\rho)}; E)} \leq C_1 \rho \norm{D_0 u}_{\Lp{2}(Z_{[0,\rho)}; E)} + \norm{u}_{\Lp{2}(Z_{[0,\rho)}; E)}.$$
Moreover, if $\rho < 1/C_1$, then   $\dom(D_{\max};Z_{[0,\rho)}) = \dom(D_{0,\max}; Z_{[0,\rho)})$.
\end{lem}
\begin{proof}
The key is to write  
$$D - D_0 = (\sym_t - \sym_0)\sym_0^{-1} D_0 + \sym_t R_t$$
via Lemma~\ref{Lem:Lem4.1BB} and note that it suffices to estimate $t \norm{A u}_{\Lp{2}(\Sigma)} = \norm{A (tu)}_{\Lp{2}(\Sigma)}$.
But $v = t u$ satisfies $v(0) = 0$ and so by Lemma~\ref{Lem:BdyEll}, the asserted inequality follows.
The existence of $C_1$ independent of $\rho$ is also guaranteed by Lemma~\ref{Lem:BdyEll}.

Now, for the maximal domains, note that for the choice of $\rho < 1/C_1$, 
$\norm{(D - D_{0})u} \leq \norm{D_{0}u} + \norm{u},$
and so  we have that $D - D_{0}$ and $D_{0}$ are relatively bounded with constant less than $1$.
Using Theorem~1.1 on page 190 in \cite{Kato},  and since $\Ck[c]{\infty}(Z_{[0,\rho)};E)$ are dense in both $\dom(D_{\max};Z_{[0,\rho)})$ and $\dom(D_{0,\max}; Z_{[0,\rho)})$ by the completeness assumption, we conclude that $D$ and $D_{0}$ have the same maximal domain on $\Lp{2}(Z_{[0,\rho)}; E)$.
\end{proof}

Now we study and characterise the maximal domain.
Put $\tilde{T}_c:=\min\{T_c,1/C_1\}$.
Define the space
\begin{equation}
\label{Eq:H1D} 
\SobH[D]{1}(M;E) := \dom(D_{\max}) \intersect \SobH[loc]{1}(M;E)
\end{equation}
with norm
\begin{equation}
\label{Eq:H1DNorm}
\norm{u}_{\SobH[D]{1}}^2 := \norm{\tilde{\eta} u}_{\SobH{1}(M)}^2 + \norm{D u}^2_{\Lp{2}(M)} + \norm{u}^2_{\Lp{2}(M)},
\end{equation} 
where $\tilde{\eta}:\R\to[0,1]$ is a smooth cutoff function satisfying: 
$$
\tilde{\eta}(t) = \begin{cases} 1, &t \in [0, \tilde{T}_c/2], \\ 
				0, &t \in [2\tilde{T}_c/3, \infty).
		\end{cases}$$

With this, we obtain the following.
\begin{lem}
\label{Lem:H1D-D}
Whenever $u \in \Ck[c]{\infty}(M;E)$ with $\chi^{+}(A_r) (u\rest{\Sigma}) = 0$, we have that $ \norm{u}_{D} \simeq \norm{u}_{\SobH[D]{1}}$.
\end{lem} 
\begin{proof}
The easy direction is simply given by 
$$
\norm{u}_{D}^2 
\leq 
\norm{\tilde{\eta} u}_{\SobH{1}}^2 + \norm{u}_{D}^2 
= 
\norm{u}_{\SobH[D]{1}}^2.
$$
For the reverse direction, we have that $\tilde{\eta} u = u$ for $t < \tilde{T}_c/2$ and therefore, $\chi^+(A)((\tilde{\eta} u)\rest{\Sigma}) = 0$.
Then,
$$ \norm{\tilde{\eta}u}_{\SobH{1}} \lesssim \norm{\tilde{\eta} u}_{D_0} \lesssim \norm{\tilde{\eta} u}_{D} \lesssim \norm{u}_{D},$$
where the first inequality follows from Lemma~\ref{Lem:BdyEll} and the uniform boundedness of $\sym_0$, the second from Lemma~\ref{Lem:D-D0}, and the last from the fact that $\tilde{\eta} \in \Ck[c]{\infty}([0,\tilde{T}_c))$.
By Lemma~\ref{Lem:GreensDensity}, we obtain that $\Ck[c]{\infty}(M;E)$ is dense in $\SobH[D]{1}(M;E)$ and $\Ck[cc]{\infty}(M;E)$ is dense in $\set{u \in \SobH[D]{1}(M;E): u\rest{\Sigma} = 0}.$
This completes the proof. 
\end{proof} 

We have the following corollary that identifies the minimal domain of the operator.

\begin{cor}
We have that $\dom(D_{\min}) = \set{u \in \SobH[D]{1}(M;E): u \rest{\Sigma} = 0}$. 
\hfill$\Box$
\end{cor} 

Before presenting the key theorem of this section, we note the following which allows us to
dualise over the boundary defining subspaces of $D$ and $D^\dagger$.
Recall that $\sym_0(x) = \sym_{D}(x,\tau(x))$.
 
\begin{lem}
\label{Lem:HomFieldPairing}
Over the boundary $\Sigma$, the homomorphism field $(\sym_0^{-1})^\ast: E\rest{\Sigma} \to F\rest{\Sigma}$ induces an isomorphism $\hatH(A^\dagstar) \to \checkH(\tilde{A})$ where $\tilde{A}$ is the adapted boundary operator for $D^\dagger$.
Moreover, $\beta(u,v) = - \inprod{\sym_0 u, v}_{\Lp{2}(\Sigma;F)}$ for $u \in \Ck[c]{\infty}(\Sigma;E)$ and $v \in \Ck[c]{\infty}(\Sigma;F)$ extends to a perfect pairing $\beta: \checkH(A) \times \checkH(\tilde{A}) \to \C$. 
\end{lem}
\begin{proof}
Fix $u \in \Ck{\infty}(\Sigma;E)$. 
We have that $(\sym_0^{-1})^\ast u \in \Ck{\infty}(\Sigma;F)$.
Then, on noting that $(\sym_0^{-1})^\ast u = (\sym^{-1})^\ast (\ext^\ast u)\rest{\Sigma}$,
\begin{align*}
\norm{(\sym_0^{-1})^\ast u}_{\checkH(\tilde{A})}^2 &\lesssim \norm{(\sym^{-1})^\ast \ext^\ast u}_{D^\dagger}^2 \\ 
	&= \norm{D^\dagger (\sym^{-1})^\ast \ext^\ast u}^2_{\Lp{2}(Z_{[0,T_c)})} +\norm{(\sym^{-1})^\ast\ext^\ast u}^2_{\Lp{2}(Z_{[0,T_c)})} \\
	&\lesssim \norm{\ext^\ast u}_{(\sym^{-1}D)^\dagger}^2 \\
	&\lesssim \norm{u}_{\hatH(A^\dagstar)}^2,
\end{align*}
where the first inequality is obtained from invoking Lemma~\ref{Lem:RestBd} applied to the operator $D^\dagger$ and its associated adapted boundary  operator $\tilde{A}$, whereas the last inequality follows from Lemma~\ref{Lem:ExtBd}. 

Next, let $\tilde{\ext}$ be the extension operator associated to $\tilde{A}_r$, that is explicitly $\tilde{\ext} v(t, \cdot) = \eta(t) \exp(-t\modulus{\tilde{A}_r})v$. 
Since $\sym_0^\ast v  = (\sym^\ast \tilde{\ext} v)\rest{\Sigma}$, 
\begin{equation*}
\norm{\sym_0^\ast v}_{\hatH(A^\dagstar)} \lesssim \norm{\sym^\ast \tilde{\ext} v}_{(\sym^{-1}D)^\dagger}  \lesssim \norm{\tilde{\ext} u}_{D^\dagger} \lesssim \norm{u}_{\checkH(\tilde{A})},
\end{equation*}
where the penultimate inequality follows from the fact that $(\sym^{-1}D)^\dagger \sym^\ast = D^\dagger$ and the final inequality from applying Lemma~\ref{Lem:ExtBd} to $D^\dagger$.

Therefore, $\beta$ induces a perfect pairing as stated in the conclusion since  the usual inner product in $\Lp{2}$ induces a perfect paring between $\checkH(A)$ and $\hatH(A^\dagstar)$.
\end{proof} 

Putting these facts together, we now obtain a proof of Theorem~\ref{Thm:Ell}.
\begin{proof}[Proof of Theorem~\ref{Thm:Ell}]
The proof of \ref{Thm:Ell1} is stated in Lemma~\ref{Lem:GreensDensity}. 

To prove \ref{Thm:Ell2}, note that $\ext u\in\dom(D_{\max})$ by Lemma~\ref{Lem:RestBd}.
By this same lemma and \ref{Thm:Ell1}, the trace map is a bounded map $\dom(D_{\max})\to \checkH(A)$.
Since $\ext:\checkH(A) \to \dom(D_{\max})$ is a right inverse of the trace map, surjectivity follows. 

We prove \ref{Thm:Ell3}. 
Clearly, if $u\in \SobH[loc]{1}(M;E)$ then $u\rest{\Sigma}\in \SobH{\frac12}(\Sigma;E)$ by the trace theorem.

For the converse, let $u_\Sigma \in \Ck{\infty}(\Sigma;E)$.
Here, we resort to the $H^\infty$-functional calculus of $\modulus{A_r}$.
Then $u^{+}_\Sigma := \chi^{+}(A_r)(u\rest{\Sigma})\in \Ck{\infty}(\Sigma;E)$. 
Now 
\begin{align*}
\norm{\ext u^+_\Sigma}_{\SobH[D]{1}}^2 &= \norm{\tilde{\eta} \ext u^+_\Sigma}_{\SobH{1}}^2 + \norm{\ext u^+_{\Sigma}}_{\Lp{2}}^2 + \norm{D \ext u^+_\Sigma}_{\Lp{2}}^2 \\
	&\lesssim  \int_{0}^{\tilde{T}_c} \norm{ (\tilde{\eta} \ext u^+_\Sigma)'}_{\Lp{2}(\Sigma)}^2\ dt + \int_{0}^{\tilde{T}_c} \norm{ \modulus{A_r} \ext u^+_\Sigma}_{\Lp{2}(\Sigma)}^2\ dt,
\end{align*}
since $D \ext u^+_{\Sigma} = (\eta'(t) + \eta(t)r)\exp(-t \modulus{A_r})u^+_{\Sigma}$ given that $\chi^{-}(A_r)u^+_\Sigma = 0$ and also since $\norm{\ext u^+_\Sigma} \lesssim \norm{\modulus{A_r} \ext u^+_\Sigma}$ by the invertibility of $\modulus{A_r}$. 
Thus, 
\begin{align*}
\int_{0}^{\tilde{T}_c} &\norm{ (\tilde{\eta} \ext u^+_\Sigma)'}_{\Lp{2}(\Sigma)}^2\ dt \lesssim \int_0^{\tilde{T}_c} \norm{ (\ext u^+_\Sigma)'}_{\Lp{2}(\Sigma)}^2 dt + \int_0^{\tilde{T}_c} \norm{ \ext u^+_\Sigma}_{\Lp{2}(\Sigma)}^2 dt \\
	&\lesssim \int_0^{\tilde{T}_c} \norm{ \eta' \exp(-t\modulus{A_r}) u^+_\Sigma}_{\Lp{2}(\Sigma)}^2 dt + \int_0^{\tilde{T}_c} \norm{ \eta \modulus{A_r} \exp(-t \modulus{A_r})u^+_\Sigma}_{\Lp{2}(\Sigma)}^2 dt \\
	&\lesssim \int_0^\infty \norm{ \modulus{A_r} \exp(-t \modulus{A_r})u^+_\Sigma}_{\Lp{2}(\Sigma)}^2.
\end{align*}
On noting that $\norm{ \modulus{A_r} \ext u^+_\Sigma}_{\Lp{2}(\Sigma)}^2 \le \norm{ \modulus{A_r} \exp(-t\modulus{A_r})u^+_\Sigma}_{\Lp{2}(\Sigma)}^2$, we obtain
\begin{align*} 
\norm{\ext u^+_\Sigma}_{\SobH[D]{1}}^2  
&\lesssim 
\int_0^\infty \norm{ \modulus{A_r} \exp(-t \modulus{A_r})u^+_\Sigma}_{\Lp{2}(\Sigma)}^2\ dt \\
&\lesssim 
\int_0^\infty \norm{ t^{\frac{1}{2}} \modulus{A_r}^{\frac{1}{2}} \exp(-t \modulus{A_r}) \modulus{A_r}^{\frac{1}{2}}u^+_\Sigma}_{\Lp{2}(\Sigma)}^2\ \frac{dt}{t} \\
&\lesssim 
\norm{\modulus{A_r}^{\frac{1}{2}} u^+_\Sigma}^2\\ 
&\simeq  
\norm{u^+_\Sigma}_{\SobH{\frac{1}{2}}}^2,
\end{align*}
where the third inequality follows from Lemma~\ref{Lem:ModAFC}.
Thus $\ext$ restricts to a bounded map $\chi^+(A)\SobH{\frac12}(\Sigma;E)\to \SobH[D]{1}(M;E)$.

Now let $u\in\dom(D_{\max})$ such that $u_\Sigma:=u\rest{\Sigma} \in \SobH{\frac12}(\Sigma;E)$.
Then $u - \ext u^+_\Sigma \in \dom(D_{\max})$ and $\chi^+(A_r)(u - \ext u^+_\Sigma) = 0$.
By Lemma~\ref{Lem:H1D-D}, we obtain that $u - \ext u^+_\Sigma \in \SobH[D]{1}(M;E)$.
Since $\ext u^+_\Sigma\in\SobH[D]{1}(M;E)$ we have that $u \in \SobH[D]{1}(M;E)$.

The assertions for the operator $D^\dagger$ mirrors this argument exactly, but with $A$ replaced by $\tilde{A}$, the adapted boundary operator for $D^\dagger$ on $F$, and moreover, the integration by parts formula follows directly. 
\end{proof}

\begin{lem}
\label{Lem:RegInv3}
There exists $T_R < T_c$ such that 
$$(D - \sym_0 R_0): \SobH{k+1}(Z_{[0,T_R]}; E; B_0) \to \SobH{k}(Z_{[0,T_R]};F)$$
is an isomorphism.
\end{lem} 
\begin{proof}
Note that from Lemma~\ref{Lem:Lem4.1BB}, we obtain that for any given $\epsilon > 0$, there exists a $\rho < T_c$ such that 
\begin{equation} 
\label{Eq:Relbnd2}
\norm{D_{0,r} - (D - \sym_0 R_0)u}_{\SobH{k}(Z_{[0,\rho]})} \leq \epsilon \norm{u}_{\SobH{k+1}(Z_{[0,\rho]})},
\end{equation} 
for $u \in \SobH{k+1}(Z_{[0,\rho]};E)$.
Moreover, as noted in Lemma~\ref{Lem:SReg}, the norm on the operator $S_{0,r}$ mapping $\SobH{k}(Z_{[0,\rho]};E) \to \SobH{k+1}(Z_{[0,\rho]};E)$ is independent of $\rho$.
On letting this constant be $C$, choose $\epsilon = 1/2C$ and let $T_R$ be the corresponding $\rho$.
We conclude the proof by combining this with \eqref{Eq:Relbnd2}.
\end{proof} 

Using this lemma, we prove the following main theorem, Theorem~\ref{Thm:HighReg}, of this section. 

\begin{proof}[Proof of Theorem~\ref{Thm:HighReg}]
The inclusion
\begin{multline*}
\dom(D_{\max}) \cap \SobH[loc]{k+1}(M;E) \\
	\subset \set{u \in \dom(D_{\max}): Du \in \SobH[loc]{k}(M;F)\ \text{and}\ \chi^{+}(A_r)(u\rest{\Sigma}) \in \SobH{k + \frac{1}{2}}(\Sigma;E)}
\end{multline*} 
is immediate.
So we prove the opposite containment.
Assume that $u \in \dom(D_{\max})$ with $Du \in \SobH[loc]{k}(M;F)$ and that $\chi^+(A_r) u|_\Sigma \in \SobH{k + \frac{1}{2}}(\Sigma;E)$.

Multiplying with a cutoff function, interior elliptic regularity allows us to assume that $\spt u \subset Z_{[0, \rho)}$ for, say  $\rho := \min\{1/(2C_1), T_R\}$, where $C_1$ is the constant in Lemma~\ref{Lem:D-D0} and $T_R$ is given by Lemma~\ref{Lem:RegInv3}.
By induction on $k$ we can further assume that $u \in \SobH[loc]{k}(M;E)$.

Since $Z_{[0,\rho]}$ is compact, $u \in \SobH{k}(Z_{[0,\rho]}; E)$ and $Du \in \SobH{k}(Z_{[0,\rho]}; F)$.
By our choice of $\rho$, $\dom((D_{0,r})_{\max};Z_{[0,\rho]}) = \dom(D_{\max};Z_{[0,\rho]})$. 
As $u \rest{ \set{\rho} \times \Sigma} = 0$, we apply Proposition~\ref{Prop:RegInv} to write
$$u(t)  = S_{0,r}\sigma_0^{-1} D_{0,r} u(t) + \exp({-t\modulus{A_r}})(\chi^+(A_r) u(0)) =: u_0(t) + u_1(t).$$

For $0 \leq l \leq k+1$, 
	\begin{align*}
	\norm{\partial_t^l \modulus{A_r}^{k+1 - l} u_1}_{\Lp{2}(Z_{[0,\rho]})}^2 &= \norm{\modulus{A_r}^{k+1} u_1}_{\Lp{2}(Z_{[0,\rho]})}^2  \\
		&= \int_{0}^\rho \norm{t^{\frac{1}{2}}\modulus{A_r}^{\frac{1}{2}}\exp(-t \modulus{A_r}) \modulus{A_r}^{k+ \frac{1}{2}}(\chi^{+}(A_r)u(0))}_{\Lp{2}(\Sigma)}^2 \frac{dt}{t} \\
		&\lesssim \norm{\modulus{A_r}^{k+ \frac{1}{2}}(\chi^{+}(A_r)u(0))}_{\Lp{2}(\Sigma)}^2
	\end{align*}
by Lemma~\ref{Lem:ModAFC}.
Thus, we conclude that $u_1 \in \SobH{k+1}(Z_{[0,\rho]}; E)$.

Moreover,  $\sym_0 R_0$ is of order $0$ and hence bounded as an operator $\SobH{k}(Z_{[0,\rho]};E) \to \SobH{k}(Z_{[0,\rho]};E)$.
Therefore
$$ (D - \sym_0 R_0)u_0 =  (D - \sym_0 R_0)u -  (D - \sym_0 R_0)u_1 \in \SobH{k}(Z_{[0,\rho]};F).$$
By the construction of $S_{0,r}$, we have $u_0 \in \SobH{k}(Z_{[0,\rho]};E;B_0)$.
Again by our choice of $\rho$, Lemma~\ref{Lem:RegInv3} implies that
$$ (D - \sym_0 R_0): \SobH{k+1}(Z_{[0,T_R]}; E; B_0) \to \SobH{k}(Z_{[0,T_R]};F)$$ 
is an isomorphism.
Therefore, there is a $\tilde{u}_0 \in  \SobH{k+1}(Z_{[0,T_R]}; E; B_0)$ such that $(D - \sym_0 R_0)\tilde{u}_0 = (D - \sym_0 R_0)u_0.$
However, we have that $\tilde{u}_0 = u_0$ because $(D - \sym_0 R_0):\SobH{k}(Z_{[0,T_R]}; E; B_0) \to \SobH{k-1}(Z_{[0,T_R]};F)$ is injective.
Therefore, $u \in \SobH{k+1}(Z_{[0,\rho]};E)$.
\end{proof} 

\section{Boundary value problems}
\label{Sec:BVPs}

In this section, we obtain a  generalisation of statements in Section~7 in \cite{BB12}.
The results of the previous sections can be considered a suitable substitute for those of Sections~4--6 in \cite{BB12} which furnishes us with the ability to define boundary conditions as well as consider a range of significant questions surrounding  them. 
We begin with Subsection~\ref{Sec:BC}, due to Theorem~\ref{Thm:Ell},  we are able to formulate notions of boundary conditions quite simply mimicking \cite{BB12}. 
We are then also able to understand the associated closed operators to boundary conditions as well as all closed extensions of the minimal operator via boundary conditions.
Moreover, we also obtain a description of the adjoint of an operator with a particular boundary condition via the formal adjoint with an associated adjoint boundary condition.

Then, in Subsection~\ref{Sec:PfEllEquiv}, we move on to the proof of Theorem~\ref{Thm:EllEquiv}, which is at the heart of this section.
Despite the tools made available to us in the previous sections, establishing this theorem is considerably harder than the version found in \cite{BB12} as Theorem~7.11.
There, the authors were able to enjoy the luxury of orthogonality  due to the selfadjointness of their adapted boundary operators.
It is precisely this which we cannot afford and which is at the heart of our complications.
A particular aspect of our troubles lie in the notion of an elliptic decomposition of a boundary condition with respect to an admissible cut parameter. 
In \cite{BB12}, this definition involved four subspaces, but in our more general case, there are eight.
Therefore, much of the effort in this section is to prove Theorem~\ref{Thm:EllEquiv}, which although is similar in spirit to the proof of Theorem~7.11, is quite different in implementation.

We also consider boundary regularity in Subsection~\ref{Sec:BdyReg}, again taking inspiration from Section~7.4 in \cite{BB12}.
There, we provide a proof of Theorem~\ref{Thm:HBR}, which is the corresponding replacement and generalisation for Theorem~7.17 in \cite{BB12}. 

An important class of boundary conditions are the so-called \emph{local} and \emph{pseudo-local} boundary conditions.
We consider these in Subsection~\ref{Sec:LPL} and provide a proof of Theorem~\ref{Thm:PL}.
This theorem generalises Theorem~7.20 in \cite{BB12}, with some of the directions argued exactly the same as they are simply consequences of well known results by Hörmander.

\subsection{Boundary conditions and closed extensions}
\label{Sec:BC}
For a subspace $U \subset \union_{r \in \R} \SobH{r}(\Sigma;E)$, let $U^s$ be the closure of $U \intersect \SobH{s}(\Sigma;E)$ in $\SobH{s}(\Sigma;E)$, and similarly define $\check{U}$ and $\hat{U}$.

We say $B \subset \checkH(A)$ is a boundary condition for $D$ if it is a closed subspace in $\checkH(A)$.
This is justified as a consequence of \ref{Thm:Ell2} in Theorem~\ref{Thm:Ell}. 
The domains of the associated operators are given by: 
\begin{align*}
\dom(D_{B,\max}) &= \set{ u \in \dom(D_{\max}): u\rest{\Sigma} \in B}\\
\dom(D_{B}) &= \set{u \in \SobH[D]{1}(M, E): u\rest{\Sigma} \in B}.
\end{align*}

\begin{proposition}
\label{Prop:ClosedExt} 
We have that: 
\begin{enumerate}[(i)]
\item 
\label{Prop:ClosedExt1} 
If $D_c$ is a closed extension of $D$ between $D_{cc}$ and $D_{\max}$, then there exists a boundary condition $B \subset \checkH(A)$ such that  $D_c = D_{B,\max}$.
\item 
\label{Prop:ClosedExt2} 
A boundary condition $B$ is contained in $\SobH{\frac{1}{2}}(\Sigma;E)$ if and only if $D_{B} = D_{B,\max}$.
In this case, $\norm{u}_{\SobH[D]{1}} \simeq \norm{u}_{D}$ for all $u \in \dom(D_B)$.
\end{enumerate} 
\end{proposition}
\begin{proof}
By Theorem~\ref{Thm:Ell}~\ref{Thm:Ell2} the restriction map induces an isomorphism $\dom(D_{\max})/\dom(D_{\min}) \to \checkH(A)$.
Since $D_c$ is closed, $\dom(D_c)$ is a closed subspace of $\dom(D_{\max})$ and $B$ is its image in $\checkH(A)$.

The first part of assertion~\ref{Prop:ClosedExt2} follows from Theorem~\ref{Thm:Ell}~\ref{Thm:Ell3}.
As to the estimate, we assume w.l.o.g.\ that $\spt u\subset Z_{[0,\rho]}$ for $\rho$ so small that Lemma~\ref{Lem:D-D0} and Proposition~\ref{Prop:RegInv} are valid.
Then 
\begin{align*}
\norm{u}_{H^1}
&\le
\norm{S_{0,r}\sym_0^{-1} D_{0,r}u}_{H^1} + \norm{\exp(-t|A_r|)\chi^+(A_r)u|_\Sigma}_{H^1} \\
&\lesssim
\norm{D_{0,r}u}_{L^2} + \norm{\chi^+(A_r)u|_\Sigma}_{H^{\frac12}}\\
&\lesssim
\norm{u}_D + \norm{u|_\Sigma}_{\checkH(A)}\\
&\lesssim \norm{u}_D .
\qedhere
\end{align*}
\end{proof} 

Define: 
\begin{equation}
\label{Eq:AdjBCond}
B^\ad := \set{v \in \checkH(\tilde{A}): \inprod{ \sym_0 u, v} = 0\quad \forall u \in B}.
\end{equation}

The subspace $B^\ad$ is called the \emph{adjoint  boundary condition} and this nomenclature is justified by the following proposition.
\begin{proposition}\label{Prop:B0}
We have that 
$$ \dom( (D_{B,\max})^\ad ) = \set{v \in \dom( (D^\dagger)_{\max} ): v\rest{\Sigma} \in B^\ad}.$$
If further $B \subset \SobH{\frac{1}{2}}(\Sigma;E)$, then $\sym_0^\ast B^\ad = B^{\perp, \SobH{-\frac12}} \intersect \hatH(A^\dagstar)$ where
$$ B^{\perp, \SobH{-\frac12}} := \set{ w \in \SobH{-\frac{1}{2}}(\Sigma;E): \inprod{u,w} = 0\quad \forall u \in B}.$$
Here $\inprod{\cdot,\cdot}$ denotes the pairing between $\SobH{\frac{1}{2}}(\Sigma;E)$ and $\SobH{-\frac{1}{2}}(\Sigma;E)$, that is,  $B^{\perp,\SobH{-\frac12}}$ is the annihilator of $B$ in $\SobH{-\frac{1}{2}}(\Sigma;E)$. 
\end{proposition}
\begin{proof}
Note $(D_{B,\max})^{\ad} \subset (D_{cc})^\ad = (D^\dagger)_{\max}$ where $(D_{B,\max})^\ad$ is the adjoint of $D_{B,\max}$ in $\Lp{2}(M;E)$.
This satisfies the equation
$$ \inprod{D_{B,\max}u,v} = \inprod{u, (D_{B,\max})^{\ad} v}$$
for all $u \in \dom(D_{B,\max})$ and $v \in \dom((D_{B,\max})^{\ad})$.
Therefore, by \eqref{Eq:MaxDInt} in Theorem~\ref{Thm:Ell}, we have that 
$$ \inprod{\sym_0  u\rest{\Sigma}, v\rest{\Sigma}} = 0,$$
where we recall that this is the pairing between $\checkH(A)$ and $\checkH(\tilde{A})$ from Lemma~\ref{Lem:HomFieldPairing}.
Thus, the characterisation of the domain of $(D_{B,\max})^\ad$ as given in the conclusion follows.

Now, assume that $B \subset \SobH{\frac{1}{2}}(\Sigma;E)$.
Then $\inprod{\cdot,\cdot}$ agrees with the pairing $\SobH{\frac{1}{2}}(\Sigma;E) \times \SobH{-\frac{1}{2}}(\Sigma;E) \to \C$.
Thus, 
\begin{align*}
\sym_0^\ast B^\ad 
&= 
\set{ \sym_0^\ast v \in \hatH(A^\dagstar): v \in \checkH(\tilde{A})\ \text{and}\ \inprod{\sym_0 u, v} = 0\quad \forall u \in B} \\
&= 
\set{ w \in \hatH(A^\dagstar): \inprod{u, w} = 0\quad \forall u \in B}  \\
&=
B^{\perp, \SobH{-\frac12}} \cap \hatH(A^\dagstar).
\qedhere
\end{align*}
\end{proof}

\subsection{Proof of Theorem~\ref{Thm:EllEquiv}}
\label{Sec:PfEllEquiv}

Our goal in this subsection is to prove Theorem~\ref{Thm:EllEquiv}, which states the equivalence of several criteria for ellipticity of a boundary condition.

We proceed by first proving the equivalence between \ref{Thm:EllEquiv:Ellbdy}--\ref{Thm:EllEquiv:Admiss2}. 
Indeed, it is immediate that in Theorem~\ref{Thm:EllEquiv} assertion~\ref{Thm:EllEquiv:Admiss} implies \ref{Thm:EllEquiv:Admiss2} and that \ref{Thm:EllEquiv:Ellbdy} implies \ref{Thm:EllEquiv:Closed}.

\begin{proof}[Proof of \ref{Thm:EllEquiv:Admiss2} $\implies$ \ref{Thm:EllEquiv:Ellbdy} and equation \eqref{eq:B*} in Theorem~\ref{Thm:EllEquiv}]
It suffices to prove \eqref{eq:B*}.
The annihilators in $\SobH{-\frac12}{(\Sigma;E)}$ of the following subspaces of $\SobH{\frac12}{(\Sigma;E)}$ are given by
\begin{align*}
W_\pm^{\perp,\SobH{-\frac12}{}}
&=
W_\mp^* \oplus (V_-^\ast)^{-\frac12} \oplus (V_+^\ast)^{-\frac12}, \\
\set{v + gv: v \in V_{-}^{\frac{1}{2}}}^{\perp,\SobH{-\frac12}{}}
&=
W_-^* \oplus W_+^* \oplus \set{v - g^*v: v \in (V_{+}^*)^{-\frac{1}{2}}} .
\end{align*}
Since, by assumption, $g$ restricts to a map $V_-^{\frac12}\to V_+^{\frac12}$ the dual map naturally extends to a map $(V_+^\ast)^{-\frac12}\to (V_-^\ast)^{-\frac12}$.
From $B = W_+ \oplus \set{v + gv: v \in V_{-}^{\frac{1}{2}}}$ we have
\begin{align*}
B^{\perp,\SobH{-\frac12}}
&=
W_+^{\perp,\SobH{-\frac12}{}} \cap \set{v + gv: v \in V_{-}^{\frac{1}{2}}}^{\perp,\SobH{-\frac12}{}} \\
&=
W_-^\ast \oplus \set{v - g^*v: v \in (V_{+}^*)^{-\frac{1}{2}}} .
\end{align*}
Since $g^*$ preserves $\SobH{\frac12}$-regularity, 
$$
B^{\perp,\SobH{-\frac12}} \cap \hatH(A^\dagstar) 
=
W_-^\ast \oplus \set{v - g^*v: v \in (V_{+}^*)^{\frac{1}{2}}}.
$$
From Proposition~\ref{Prop:B0} we recall that $\sym_0^\ast B^\ad = B^{\perp,\SobH{-\frac12}} \intersect \hatH(A^\dagstar)$ which concludes the proof.
\end{proof} 

Next, we demonstrate that \ref{Thm:EllEquiv:Closed} implies \ref{Thm:EllEquiv:Admiss}.
This  is considerably more involved than the proof of the corresponding result in the selfadjoint case given as Theorem~7.11 in \cite{BB12}.

\begin{lem}
\label{Lem:AdjDecomp}
Suppose $W_{\pm}$ and $V_{\pm}$ are mutually complementary subspaces such that $\chi^\pm(A_r) \Lp{2}(\Sigma;E) = W_{\pm} \oplus V_{\pm}$ and that $W_{\pm}$ are finite dimensional.
Let $P_{\pm}:\Lp{2}(\Sigma;E) \to V_{\pm}$ and $Q_{\pm}: \Lp{2}(\Sigma;E) \to W_{\pm}$ be associated projectors that respect the decomposition 
$$ \Lp{2}(\Sigma;E) = W_- \oplus V_- \oplus W_+ \oplus V_+.$$
Then, writing $W_{\pm}^\ast := Q_{\pm}^\ast \Lp{2}(\Sigma;E)$ and $V_{\pm}^\ast := P_{\pm}^\ast\Lp{2}(\Sigma;E) $, we have that: 
\begin{enumerate}[(i)] 
\item 
\label{Lem:AdjDecomp1}
$V_{-}^\ast \oplus W_{-}^\ast = \chi^{-}(A_r^\ast)\Lp{2}(\Sigma;E) $ and $V_{+}^\ast \oplus W_{+}^\ast = \chi^{+}(A_r^\ast)\Lp{2}(\Sigma;E) $, 
\item 
\label{Lem:AdjDecomp2}
$\nul(P_{\pm}^\ast) = V_{\mp}^\ast \oplus W_{-}^\ast \oplus W_{+}^\ast$ and $\nul(Q_{\pm}^\ast) = W_{\mp}^\ast \oplus V_{-}^\ast \oplus V_+^\ast$, 
\item 
\label{Lem:AdjDecomp3}
$\dim(W_{\pm}^\ast)=\dim(W_\pm)$.
\end{enumerate} 
\end{lem}

\begin{proof}
From the observation $W_{\pm} \intersect V_{\pm} = \set{0}$, we have
$$
\nul(P_{\pm} + Q_{\pm}) = \nul(P_{\pm}) \intersect \nul(Q_{\pm}) =W_{\mp} \oplus V_{\mp}= \chi^{\mp}(A_r) \Lp{2}(\Sigma;E).
$$ 
This proves $\chi^{\pm}(A_r) = P_{\pm} + Q_{\pm}$.
Therefore, on taking adjoints, we obtain assertion~\ref{Lem:AdjDecomp1}.

Now, note that $1 = P_{-}^\ast + Q_{-}^\ast + P_{+}^\ast + Q_{+}^\ast$ and by the observation $\nul(P_{\pm}^\ast) = (1 - P_{\mp}^\ast)\Lp{2}(\Sigma;E)$ and $\nul(Q_{\pm}^\ast) = (1 - Q_{\mp}^\ast)\Lp{2}(\Sigma;E)$, we obtain \ref{Lem:AdjDecomp2}

Since the ranks of $Q_\pm$ and $Q_\pm^\ast$ coincide, assertion~\ref{Lem:AdjDecomp3} holds.
\end{proof}

\begin{lem}
\label{Lem:Vprop}
Let $V_{\pm}$ and $W_{\pm}$ be closed subspaces of $\Lp{2}(\Sigma;E)$ as Lemma~\ref{Lem:AdjDecomp} with corresponding projectors $P_\pm$ and $Q_\pm$.
In addition, assume $W_{\pm} \subset \SobH{\frac{1}{2}}(\Sigma;E)$.
Then, 
\begin{enumerate}[(i)]
\item 
\label{Lem:Vprop1}
$V_{\pm}^\frac{1}{2} = V_{\pm} \intersect \SobH{\frac{1}{2}}(\Sigma;E)$ are closed in $\SobH{\frac{1}{2}}(\Sigma;E)$ and 
	$$\SobH{\frac{1}{2}}(\Sigma;E) = V_-^\frac{1}{2} \oplus W_- \oplus V_{+}^\frac{1}{2} \oplus W_+,$$
\item 
\label{Lem:Vprop2}
the spaces $V_{\pm}^\frac{1}{2}$ are dense in $V_{\pm}$, 
\item 
\label{Lem:Vprop3}
the projectors $P_{\pm}$ restrict to bounded projectors 
	$$P_{\pm}^{\frac{1}{2}}: \SobH{\frac{1}{2}}(\Sigma;E) \to V_{\pm}^\frac{1}{2}.$$ 
\end{enumerate}
If $W_{\pm}^\ast \subset \SobH{\frac{1}{2}}(\Sigma;E)$, the same conclusions hold for $V_\pm^\ast$, $W_\pm^\ast$ and $P_{\pm}^\ast$ and $Q_{\pm}^\ast$ in place of $V_{\pm},\ W_{\pm}, P_{\pm}, Q_{\pm}$.
\end{lem}
\begin{proof}
Since $\chi^{\pm}(A_r) = P_\pm + Q_\pm$, we have that $P_\pm  = \chi^{\pm}(A_r) - Q_\pm$.
The first operator is a pseudo-differential operator of order $0$ and the latter has range in $\SobH{\frac{1}{2}}(\Sigma;E)$ and finite rank.
Hence, $P_\pm^\frac{1}{2}: \SobH{\frac{1}{2}}(\Sigma;E) \to \SobH{\frac{1}{2}}(\Sigma;E)$ is bounded.
This shows \ref{Lem:Vprop3}.

Let $v \in V_{\pm}$.
By density of $\SobH{\frac{1}{2}}(\Sigma;E)$ in $\Lp{2}(\Sigma;E)$, there exists $v_n\in\SobH{\frac12}{(\Sigma;E)}$ such that $v_n \to v$. 
Then $P_\pm v_n \to P_\pm v=v$.
This shows \ref{Lem:Vprop2}.

The spaces $W_\pm$ and $V_{\pm}^\frac{1}{2}$ are ranges of $H^{\frac12}$-bounded projectors and hence closed in $\SobH{\frac{1}{2}}(\Sigma;E)$.
For $u \in \SobH{\frac{1}{2}}(\Sigma;E)$, 
$$u = P_{-}u + Q_{-} u + P_{+} u + Q_{+} u$$
with $P_{\pm} u \in V_{\pm}^\frac{1}{2}$ as we have proved earlier.
This shows \ref{Lem:Vprop1}. 
\end{proof} 

\begin{lem} 
\label{Lem:ClosedNorm}
Let $X$ be one of $\checkH(A)$, $\hatH(A)$, $\checkH(A^\dagstar)$ or $\hatH(A^\dagstar)$ and let $Z \subset X$ be a closed subspace of $X$.
If $Z \subset \SobH{\frac{1}{2}}(\Sigma;E)$, then it is a closed subspace of $\SobH{\frac{1}{2}}(\Sigma;E)$ and moreover, $\norm{u}_{\SobH{\frac{1}{2}}} \simeq \norm{u}_{X}$. 
\end{lem}
\begin{proof}
Set $Y = \SobH{\frac{1}{2}}(\Sigma;E)$ and note that we have $\norm{u}_{X} \lesssim \norm{u}_{\SobH{\frac{1}{2}}}$.
All possible choices of $X$ are Banach spaces and hence, we can apply Lemma~\ref{Lem:Abs1} to obtain that $(Z, \norm{\cdot}_{\SobH{\frac{1}{2}}})$ is a Banach space. 
Thus, this satisfies the hypothesis of Lemma~\ref{Lem:Abs2} and hence, we obtain the desired conclusion.
\end{proof}

Define the following spaces:
\begin{equation}
\label{Eq:Spaces} 
\begin{aligned}
&W_-^\ast := \chi^{-}(A_r^\ast)\Lp{2}(\Sigma;E) \intersect \sym_0^\ast B^\ad 	&&W_- := \chi^{-}(A_r) W_-^\ast  \\
&W_+ := \chi^+(A_r) \Lp{2}(\Sigma;E) \intersect B 					&&W_+^\ast := \chi^+(A_r^\ast) W_+ \\
&V_-^\ast := \chi^{-}(A_r^\ast)\Lp{2}(\Sigma;E) \intersect (W_-^\ast)^\perp 	&&V_- := \chi^{-}(A_r)V_-^\ast \\
&V_+ := \chi^+(A_r)\Lp{2}(\Sigma;E) \intersect W_+^\perp\qquad 			&&V_+^\ast := \chi^+(A_r^\ast) V_+.
\end{aligned}
\end{equation}

\begin{proposition}
\label{Prop:Spaces}
Let $B$ and $B^*$ be as in Theorem~\ref{Thm:EllEquiv}~\ref{Thm:EllEquiv:Closed}.
Then, the following properties hold for the spaces listed in \eqref{Eq:Spaces}:
\begin{enumerate}[(i)]
\item \label{Spaces1}
they are all closed subspaces of $\Lp{2}(\Sigma;E)$,
\item \label{Spaces2}
$\chi^\pm(A_r)\Lp{2}(\Sigma;E)  = V_\pm \oplus W_\pm$ and $\chi^\pm(A_r^\ast)\Lp{2}(\Sigma;E)  = V_\pm^\ast \oplus W_\pm^\ast$,
\item \label{Spaces3}
$W_\pm, W_\pm^\ast$ are finite dimensional and contained in $\SobH{\frac{1}{2}}(\Sigma;E)$, 
\item \label{Spaces4}
$\chi^-(A_r)B$ and $\chi^+(A_r^\ast) \sym_0^\ast B^\ad$ are closed subspaces of $\SobH{\frac{1}{2}}(\Sigma;E)$.
\end{enumerate}
\end{proposition} 
\begin{proof}
We first prove \ref{Spaces3} and \ref{Spaces4}. 
For $u \in B$, 
$$\norm{u}_{\SobH{\frac{1}{2}}} \simeq \norm{u}_{\checkH(A)} \simeq \norm{\chi^{-}(A_r) u}_{\SobH{\frac{1}{2}}} + \norm{\chi^{+}(A_r) u}_{\SobH{-\frac{1}{2}}}.$$
The first equivalence follows from Lemma~\ref{Lem:ClosedNorm} by the assumption that $B$ is a boundary condition (i.e. $B \subset \checkH(A_r)$ is closed) and since $B \subset \SobH{\frac{1}{2}}(\Sigma;E)$.

Set $X = B$ with norm $\norm{\cdot}_X = \norm{\cdot}_{\SobH{\frac{1}{2}}}$ and so we have that this is a Banach space by Lemma~\ref{Lem:ClosedNorm}.
Let $Y = \chi^{-}(A_r) \SobH{\frac{1}{2}}(\Sigma;E)$ and $Z = \chi^+(A_r)\SobH{-\frac{1}{2}}(\Sigma;E)$. 
Each of these are Banach spaces.
Moreover, $\chi^{+}(A_r)|_B: B \to \SobH{\frac{1}{2}}(\Sigma;E) \embed \SobH{-\frac{1}{2}}(\Sigma;E)$ is a compact map since the latter embedding is compact. 
By Proposition~\ref{Prop:HormPeet}, we obtain that $\chi^{-}(A_r) \rest{B}$ has closed range and finite dimensional kernel.
Now, one easily sees that $\nul(\chi^{-}(A_r)|_B) = B \intersect \chi^{+}(A_r) \Lp{2}(\Sigma;E)$.

To obtain the corresponding conclusion for $W_-^\ast$, we note that for $v \in \sym_0^\ast B^\ad$,
$$ \norm{v}_{\SobH{\frac{1}{2}}} \simeq \norm{v}_{\hatH(A^\dagstar)}  \simeq \norm{\chi^{-}(A_r^\ast) v}_{\SobH{-\frac{1}{2}}} + \norm{\chi^{+}(A_r^\ast) v}_{\SobH{\frac{1}{2}}}.$$
By invoking Proposition~\ref{Prop:HormPeet} with $X = \sym_0^\ast B^\ad$ with norm $\norm{\cdot}_{X} = \norm{\cdot}_{\SobH{\frac{1}{2}}}$, $Y = \chi^+(A_r^\ast)\SobH{\frac{1}{2}}(\Sigma;E)$, and $Z = \chi^-(A_r^\ast) \SobH{-\frac{1}{2}}(\Sigma;E)$, we obtain that $\chi^+(A_r^*)|_{\sym_0^\ast B^\ad}$ has closed range and finite dimensional kernel.

This proves \ref{Spaces4} and that $W_{+}$ and $W_{-}^\ast$ are finite dimensional and contained in $\SobH{\frac12}(\Sigma;E)$.
The latter fact for $W_{-}$ and $W_{+}^\ast$ is simply from  $\chi^{-}(A_r)$ and $\chi^+(A_r^\ast)$ both being pseudo-differential operators of order zero.

Observing that $V_+$ is the orthogonal complement of $W_+$ in $\chi^+(A_r)\Lp{2}(\Sigma;E)$ shows \ref{Spaces1}--\ref{Spaces3} for $V_+$ and $W_+$ and similarly for $V_-^*$ and $W_-^*$.
Lemma~\ref{Lem:AdjProj}~\ref{AdjProj4} then implies these assertions for the remaining spaces.
\end{proof} 

Recall from before that we use the notation $X^{\perp, Y}$ to mean the annihilator of $X$ in the space $Y$.
\begin{lem}
\label{Lem:Gspaces}
Let $B$ and $B^*$ be as in Theorem~\ref{Thm:EllEquiv}~\ref{Thm:EllEquiv:Closed}.
Then the spaces 
$$
\chi^{-}(A_r) B =  V_-^{\frac{1}{2}}
\quad\text{and}\quad 
\chi^+(A_r^\ast) \sym_0^\ast B^\ad = (V_+^\ast)^{\frac{1}{2}}.
$$
\end{lem}
\begin{proof} 
Let $\tilde{V}_-^\frac{1}{2} = \chi^-(A_r)B$.
We prove the assertion in the following steps.
\begin{enumerate}[a)]
\item 
\label{Lem:Gspacea}
\emph{Claim:} $W_-^\ast = \chi^{-}(A_r)^\ast \SobH{-\frac{1}{2}}(\Sigma;E) \intersect B^{\perp, \SobH{-\frac12}}$. \\

The containment $\subset$ is clear by Proposition~\ref{Prop:B0}, so we prove the reverse containment. 
Fix $u \in \chi^-(A_r^\ast) \SobH{-\frac{1}{2}}(\Sigma;E) \intersect B^{\perp, \SobH{-\frac12}}$.
Then $u \in \hatH(A^\dagstar)$ and hence $u \in B^{\perp, \SobH{-\frac12}} \intersect \hatH(A^\dagstar) = \sym_0^\ast B^\ad$.
By assumption, $\sym_0^\ast B^\ad \subset \SobH{\frac{1}{2}}(\Sigma;E)$ and therefore, $\chi^{-}(A_r)^\ast \SobH{-\frac{1}{2}}(\Sigma;E) \intersect B^{\perp, \SobH{-\frac12}} \subset \chi^{-}(A_r^\ast) \Lp{2}(\Sigma;E) \intersect \sym_0^\ast B^\ad = W_-^\ast$.
\hfill\checkmark
\item 
\label{Lem:Gspaceb}
\emph{Claim:} $W_-^\ast = (\tilde{V}_-^{\frac{1}{2}})^{\perp, \chi^{-}(A_r^\ast) \SobH{-\frac{1}{2}}}$. \\

Using \ref{Lem:Gspacea} note that:
\usetagform{simple}
\begin{align*} 
w \in W_-^\ast 
	&\iff w \in \chi^-(A_r^\ast) \SobH{-\frac{1}{2}}(\Sigma;E)\quad \text{and}\quad \inprod{w, b} = 0,\ \forall b \in B \notag\\ 
	&\iff w \in \chi^-(A_r^\ast) \SobH{-\frac{1}{2}}(\Sigma;E)\quad \text{and}\quad \inprod{\chi^{-}(A_r^\ast) w, b} = 0,\ \forall b \in B \notag\\ 
	&\iff w \in \chi^-(A_r^\ast) \SobH{-\frac{1}{2}}(\Sigma;E)\quad \text{and}\quad \inprod{w, \chi^-(A_r) b} = 0,\ \forall b \in B \notag\\ 
	&\iff w \in \chi^-(A_r^\ast) \SobH{-\frac{1}{2}}(\Sigma;E)\quad \text{and}\quad \inprod{w, v} = 0,\ \forall v \in \tilde{V}_-^\frac{1}{2} \tag{\checkmark}
\end{align*} 
\usetagform{default}
 
\item 
\label{Lem:Gspacec}
\emph{Claim:} $\tilde{V}_-^\frac{1}{2} = (W_-^\ast)^{\perp, \chi^{-}(A_r)\SobH{\frac{1}{2}}}.$\\

By Lemma~\ref{Lem:DualSob}, we have that $\chi^{-}(A_r^\ast) \SobH{-\frac{1}{2}}(\Sigma;E) \cong (\chi^{-}(A_r)\SobH{\frac{1}{2}}(\Sigma;E))^\ast$, and therefore by \ref{Lem:Gspaceb} and Proposition~\ref{Prop:Spaces}~\ref{Spaces4}: 
\usetagform{simple}
\begin{align*}
(W_-^\ast)^{\perp, \chi^{-}(A_r)\SobH{\frac{1}{2}}}  
&= 
\cbrac{(\tilde{V}_-^{\frac{1}{2}})^{\perp, \chi^{-}(A_r^\ast) \SobH{-\frac{1}{2}}}}^{\perp, \chi^{-}(A_r)\SobH{\frac{1}{2}}} 
= 
\close{(\tilde{V}_-^{\frac{1}{2}})}^{\chi^{-}(A_r)\SobH{\frac{1}{2}}}  
= 
\tilde{V}_-^{\frac{1}{2}}.
\tag{\checkmark}
\end{align*}
\usetagform{default}

\item 
\label{Lem:Gspaced}
\emph{Claim:} $(W_-^\ast)^{\perp, \chi^{-}(A_r) \SobH{\frac{1}{2}}} = (W_-^\ast)^{\perp, \Lp{2}}\intersect \chi^-(A_r) \Lp{2}(\Sigma;E) \intersect \SobH{\frac{1}{2}}(\Sigma;E) $.\\

This is clear since $\chi^-(A_r) \Lp{2}(\Sigma;E) \intersect\SobH{\frac{1}{2}}(\Sigma;E) =\chi^-(A_r) \SobH{\frac12}(\Sigma;E)$.
\hfill\checkmark
\item 
\label{Lem:Gspacee}
\emph{Claim:} $(W_-^\ast)^{\perp, \Lp{2}}  \intersect \chi^-(A_r)\Lp{2}(\Sigma;E) = \chi^-(A_r) \bbrac{(W_-^\ast)^{\perp, \Lp{2}} \intersect  \chi^-(A_r^\ast)\Lp{2}(\Sigma;E) }.$\\

Let $u \in \chi^-(A_r) \bbrac{(W_-^\ast)^{\perp, \Lp{2}} \intersect  \chi^-(A_r^\ast)\Lp{2}(\Sigma;E)}$. 
That is, $u = \chi^-(A_r) u'$, where $u' \in (W_-^\ast)^{\perp, \Lp{2}} \intersect  \chi^-(A_r^\ast)\Lp{2}(\Sigma;E)$.
Then, for $w \in W_-^\ast$,
$$\inprod{u,w} = \inprod{\chi^-(A_r)u',w} = \inprod{u',\chi^-(A_r^*)w} = \inprod{u', w} = 0,$$
which shows that $u \in (W_-^\ast)^{\perp, \Lp{2}}  \intersect \chi^-(A_r)\Lp{2}(\Sigma;E)$.

For the reverse inclusion, let $u \in (W_-^\ast)^{\perp, \Lp{2}}  \intersect \chi^-(A_r)\Lp{2}(\Sigma;E)$.
Then by Lemma~\ref{Lem:AdjProj}~\ref{AdjProj4}, we have some $u^\ast \in \chi^-(A_r^\ast)\Lp{2}(\Sigma;E)$ such that $u = \chi^-(A_r) u^\ast$.
Therefore, for $w \in W_{-}^\ast$
$$0 = \inprod{u, w} = \inprod{\chi^{-}(A_r) u^\ast, w} = \inprod{u^\ast, \chi^{-}(A_r^*) w} = \inprod{u^\ast, w}$$
and therefore, $u^\ast \in \chi^{-}(A_r^\ast)\Lp{2}(\Sigma;E) \intersect (W_{-}^\ast)^{\perp, \Lp{2}}$ and $u = \chi^-(A_r)u^\ast$.
\hfill\checkmark
\item
\label{Lem:Gspacef}
\emph{Claim:} $\chi^-(A_r)B=V_- \intersect \SobH{\frac{1}{2}}(\Sigma;E)$. 
\begin{align*}
\chi^-(A_r)B 
&= 
\tilde{V}_-^\frac{1}{2}  \\
&= 
(W_-^\ast)^{\perp, \chi^{-}(A_r)\SobH{\frac{1}{2}}} \\
&= 
(W_-^\ast)^{\perp, \Lp{2}} \intersect \chi^-(A_r)\Lp{2}(\Sigma;E) \intersect \SobH{\frac{1}{2}}(\Sigma;E) \\ 
&= 
\chi^-(A_r) \bbrac{(W_-^\ast)^{\perp, \Lp{2}} \intersect  \chi^-(A_r^\ast)\Lp{2}(\Sigma;E) } \intersect \SobH{\frac{1}{2}}(\Sigma;E)\\
&= 
\chi^-(A_r) V_-^* \intersect \SobH{\frac{1}{2}}(\Sigma;E)\\
&= 
V_- \intersect \SobH{\frac{1}{2}}(\Sigma;E),
\end{align*}
where the second is \ref{Lem:Gspacec}, the third is \ref{Lem:Gspaced}, and the fourth is \ref{Lem:Gspacee}.
\hfill\checkmark
\end{enumerate}
Since $\chi^-(A_r)B$ is closed in $\SobH{\frac12}(\Sigma;E)$ by Proposition~\ref{Prop:Spaces}~\ref{Spaces4}, we have that $V_- \intersect \SobH{\frac{1}{2}}(\Sigma;E) = V_-^{\frac12}$.

The proof of $\chi^+(A_r^\ast) \sym_0^\ast B^\ad = (V_+^\ast)^{\frac{1}{2}}$ is identical, with $W_+$ in place of $W_-^\ast$, $\chi^+(A_r^\ast)$ in place of $\chi^-(A_r)$ and $\chi^+(A_r)$ in place of $\chi^-(A_r^\ast)$.
\end{proof}

\begin{lem} 
\label{Lem:DualAnn}
Let $B$ and $B^*$ be as in Theorem~\ref{Thm:EllEquiv}~\ref{Thm:EllEquiv:Closed}.
Then the following hold: 
\begin{enumerate}[(i)] 
\item \label{DualAnn0}
$(V_\pm^\ast)^{-\frac{1}{2}} = W_{\pm}^{\perp, \chi^{\pm}(A_r^\ast)\SobH{-\frac{1}{2}}}$,
\item \label{DualAnn1}
$V_\pm^\ast = (V_\pm^\ast)^{-\frac{1}{2}} \intersect \Lp{2}(\Sigma;E)$, 
\item \label{DualAnn3}
$\chi^{\pm}(A_r^\ast)\SobH{-\frac{1}{2}}(\Sigma;E) = (V_\pm^\ast)^{-\frac{1}{2}} \oplus W_\pm^\ast$, and
\item \label{DualAnn4}
the $\Lp{2}$-inner product induces a perfect pairing $\inprod{\cdot,\cdot}: (V_\pm^\ast)^{-\frac{1}{2}}  \times V_\pm^\frac{1}{2} \to \C$.
\end{enumerate} 
\end{lem}
\begin{proof}
It is immediate that $(V_-^\ast)^{-\frac{1}{2}} = W_{-}^{\perp, \chi^{-}(A_r^\ast)\SobH{-\frac{1}{2}}}$ from the definition of $V_-^\ast$.
For the other case, let $(\tilde{V}_+^\ast)^{-\frac12} := W_{+}^{\perp, \chi^{+}(A_r^\ast)\SobH{-\frac{1}{2}}}$.
On mirroring the argument of \ref{Lem:Gspacee} in the proof of Lemma~\ref{Lem:Gspaces} replacing  $\chi^-(A_r)$ by $\chi^+(A_r^\ast)$, $W_-^\ast$ by $W_+$ and $\chi^-(A_r^\ast)$ by $\chi^+(A_r)$, we obtain that  $V_+^\ast = (\tilde{V}_+^\ast)^{-\frac12} \intersect \Lp{2}(\Sigma;E)$.
By density of $V_+^\ast$ in $(V_+^\ast)^{-\frac12}$, we obtain that $(V_+^\ast)^{-\frac12} =  W_{+}^{\perp, \chi^{+}(A_r^\ast)\SobH{-\frac{1}{2}}}$.

By Proposition~\ref{Prop:Spaces}~\ref{Spaces1}, $V_\pm^*$ is a closed subspace of $\Lp{2}(\Sigma;E)$, hence \ref{DualAnn1} holds by definition.

From Proposition~\ref{Prop:Spaces}~\ref{Spaces2} we have $V_\pm^*\oplus W_\pm^* = \chi^{\pm}(A_r^\ast)\Lp{2}(\Sigma;E)$.
Since $W_\pm^*$ is finite-dimensional by Proposition~\ref{Prop:Spaces}~\ref{Spaces3}, completion in $\SobH{-\frac12}(\Sigma;E)$ yields \ref{DualAnn3}.

Recall by Lemma~\ref{Lem:DualSob} that the $\Lp{2}$-inner product induces a perfect pairing $\inprod{\cdot,\cdot}: \chi^\pm(A_r^\ast)\SobH{-\frac{1}{2}}(\Sigma;E) \times \chi^\pm(A_r)\SobH{\frac{1}{2}}(\Sigma;E) \to \C$.
Write $\chi^{\pm}(A_r)\SobH{\frac{1}{2}}(\Sigma;E) = (V_\pm)^{\frac12} \oplus W_\pm$ and $\chi^{\pm}(A_r^\ast) \SobH{-\frac{1}{2}}(\Sigma;E) = (V_\pm^\ast)^{-\frac{1}{2}} \oplus W_\pm^\ast$.
Now \ref{DualAnn0} implies \ref{DualAnn4}.
\end{proof} 

\begin{proof}[Proof of \ref{Thm:EllEquiv:Closed} $\implies$ \ref{Thm:EllEquiv:Admiss} in Theorem~\ref{Thm:EllEquiv}]
We prove that~\ref{Thm:EllEquiv:Closed}  implies \ref{Thm:EllEquiv:Admiss}, and consider the spaces in \eqref{Eq:Spaces}.
By Proposition~\ref{Prop:Spaces}, they satisfy \ref{Def:EllBC:MutualComp} and \ref{Def:EllBC:FiniteDim} in Definition \ref{Def:EllBC}.
Moreover, by Lemma~\ref{Lem:Gspaces}, $V_-^\frac{1}{2} = \chi^{-}(A_r)B$ and $(V_+^\ast)^\frac{1}{2} = \chi^+(A_r^\ast) \sym_0^\ast B^\ad$.
Define the maps 
\begin{align*}
X_- &= \chi^-(A_r) \rest{B\intersect W_+^\perp}: B\intersect W_+^\perp  \to \chi^-(A_r)B,\ \text{and}\\ 
X_+^\ast &= \chi^+(A_r^\ast) \rest{\sym_0^\ast B^\ad \intersect (W_-^\ast)^\perp}: \sym_0^\ast B^\ad \intersect (W_-^\ast)^\perp  \to \chi^+(A_r^\ast)\sym_0^\ast B^\ad.
\end{align*}
We claim that both these maps are bounded and invertible isomorphisms onto their ranges.
For that, let $x \in \chi^-(A_r)B$, that is, $x = \chi^-(A_r)b$ for some $b \in B$.
Then, 
$$
x 
= 
\chi^-(A_r)b
=
\chi^-(A_r)(P_{W_+^\perp, W_+} b + P_{W_+, W_+^\perp} b)
=
\chi^{-}(A_r)  P_{W_+^\perp, W_+} b
$$
since $\chi^{-}(A_r) \circ P_{W_+, W_+^\perp}=0$.
Now, 
$$
P_{W_+^\perp, W_+} b
=
b - P_{W_+, W_+^\perp} b
\in B \cap W_+^\perp
$$
because $W_+\subset B$.
This shows that $X_-$ is surjective.

To show that it is injective, let $u \in B\intersect W_+^\perp$ such that  $0 = X_- u$.
That is, $u \in \nul(\chi^-(A_r)) = \chi^+(A_r) \Lp{2}(\Sigma;E) $ and putting this together, we have that
$$ u \in \chi^+(A_r)\Lp{2}(\Sigma;E) \intersect B \intersect W_+^\perp = W_+ \intersect W_+^\perp = \set{0}.$$
Thus, we have a continuous bijection between two closed subspaces and so by the open mapping theorem, the inverse is also continuous.
The assertion for $X_+^\ast$ follows by a similar argument.

Next, put
\begin{align*} 
&P_+ := P_{V_+,W_-\oplus V_-\oplus W_+}, 	&&P_-^* :=P_{V_-^*,W_-^*\oplus V_+^*\oplus W_+^*}, \\
&Q_+ := P_{W_+, W_- \oplus V_- \oplus V_+}, 	&&Q_-^*	:=P_{W_-^* , V_-^* \oplus V_+^*\oplus W_+^*}.
\end{align*} 
Observe that $P_++Q_+=\chi^+(A_r)$ and $P_-^*+Q_-^*=\chi^-(A_r^*)$.
We define the maps $g_0: \chi^-(A_r)B \to V_+$ and $h_0: \chi^{+}(A_r^\ast) \sym_0^\ast B^\ad \to V_-^\ast$ by
$$ 
g_0 := P_+ \circ (X_-)^{-1}
\quad \text{and}\quad 
h_0 := P_-^\ast \circ (X_+^\ast)^{-1}.$$
By Lemma~\ref{Lem:Vprop}, we have that $P_{+}$ and $P_-^\ast$ restrict to $\SobH{\frac{1}{2}}$-bounded maps on $\SobH{\frac{1}{2}}(\Sigma;E)$ and therefore, it is clear that $g_0(V_-^\frac{1}{2}) \subset V_+^\frac{1}{2}$ and $h_0( (V_+^\ast)^\frac{1}{2}) \subset (V_-^\ast)^\frac{1}{2}$.

Now, we show that $B = W_+ \oplus \set{v + g_0v: v \in V_{-}^\frac{1}{2}}$. 
It is clear that $W_+ \subset B$ by definition, so fix $v \in V_{-}^\frac{1}{2}= \chi^{-}(A_r)B$. 
Then, for $u := (X_-)^{-1} v \in B \intersect W_+^\perp$ we have
\begin{align*}
v + g_0v &= v + P_+ (X_-)^{-1}v \\
&= X_- u + P_+ u \\ 
&= \chi^{-}(A_r) u + P_+ u + Q_+ u - Q_+ u \\
&= \chi^{-}(A_r) u + \chi^+(A_r)u - Q_+ u \\
&= u - Q_+ u.
\end{align*}
Since $Q_+ u \in W_+ \subset B$, we have $v + g_0v \in B$.
That shows the containment ``$\supset$''.

For the reverse, take $b \in B$, and note that $b = b^\perp + b_+$, where $b^\perp \in B \intersect W_+^\perp$ and $b_+ \in W_+$.
Then, $v:=X_-b^\perp \in \chi^{-}(A_r)B$ and therefore,
$$ b^\perp = \chi^{-}(A_r) (X_-)^{-1}v + P_+ (X_-)^{-1}v + Q_+ (X_-)^{-1}v = v + g_0v + Q_+(X_-)^{-1}v.$$
But $Q_+ (X_-)^{-1}v \in W_+$ so $b = w_+ + (v + g_0v)$ where $w_+= (b_+ + Q_+(X_-)^{-1}v) \in W_+$.
This shows $B = W_+ \oplus \set{v + g_0v: v \in V_{-}^\frac{1}{2}}$. 
A similar argument shows that $\sym_0^\ast B^\ad  = W_-^\ast \oplus \set{u + h_0u: u \in  (V_+^\ast)^\frac{1}{2}}$.

Let $g_0^*:(V_+^*)^{-\frac12}\to (V_-^*)^{-\frac12}$ the map adjoint to $g_0$, characterised by $\inprod{v, g_0^\ast u} = \inprod{g_0v,u}$ for $u\in (V_+^*)^{-\frac12}$ and $v\in V_-^{\frac12}$.
We prove $g_0^\ast( (V^\ast_+)^\frac{1}{2} ) \subset (V_-^\ast)^\frac{1}{2}$ by showing $g_0^\ast = -h_0$ on $(V^\ast_+)^\frac{1}{2}$, since we have already shown that  $h_0( (V_+^\ast)^\frac{1}{2}) \subset (V_-^\ast)^\frac{1}{2}$.
Fix $u \in (V_+^\ast)^\frac{1}{2}$ and take $v  \in V_-^\frac{1}{2}$.
From the decompositions of $B$ and $\sym_0^\ast B^\ad$ and $\inprod{B, \sym_0^\ast B^\ad}=\inprod{V_-, V_+^\ast}=\inprod{V_-^\ast, V_+}=0$, we obtain that $0 = \inprod{v + g_0v, u + h_0u} = \inprod{g_0v, u} + \inprod{v, h_0u}$.
Hence,
$$ \inprod{v, g_0^\ast u} = \inprod{v, -h_0 u}$$
and since $V_-^\frac{1}{2}$ is dense in $V_-$, we obtain that $g_0^\ast u = -h_0 u$.
Similarly, we have $g_0v = -h_0^* v$ for all $v\in V_-^\frac{1}{2}$.

It remains to show that $g_0$ can be extended to a continuous map $V_-\to V_+$.
We have already seen that $-h_0^*$ extends $g_0$ to a continuous map $(V_-)^{-\frac12}\to (V_+)^{-\frac12}$.
In order to show $-h_0^*(V_-)\subset V_+$ and $(-h_0^*)|_{V_-}:V_-\to V_+$ is bounded it suffices to prove
\begin{equation}
V_\pm  = [(V_{\pm})^{\frac{1}{2}}, (V_{\pm})^{-\frac{1}{2}}]_{\theta = \frac{1}{2}},
\label{eq:ComplexInterpolation}
\end{equation}
where the right hand side denotes the complex interpolation space.
Since $[\SobH{\frac12}(\Sigma;E),\SobH{-\frac12}(\Sigma;E)]_{\theta = \frac{1}{2}}=\Lp{2}(\Sigma;E)$, $\chi^\pm(A_r)$ is a pseudo-differential projector of order $0$, and $W_\pm$ is a closed subspace of $\SobH{\frac12}(\Sigma;E)$, we have 
\begin{align*}
V_\pm \oplus W_\pm 
&=
\chi^\pm(A_r)\Lp{2}(\Sigma;E) \\
&=
[\chi^\pm(A_r)\SobH{\frac12}(\Sigma;E),\chi^\pm(A_r)\SobH{-\frac12}(\Sigma;E)]_{\theta = \frac{1}{2}} \\
&=
[(V_{\pm})^{\frac{1}{2}}\oplus W_\pm, (V_{\pm})^{-\frac{1}{2}}\oplus W_\pm]_{\theta = \frac{1}{2}} \\
&=
[(V_{\pm})^{\frac{1}{2}}, (V_{\pm})^{-\frac{1}{2}}]_{\theta = \frac{1}{2}} \oplus W_\pm .
\end{align*}
Since $V_\pm\subset(V_{\pm})^{-\frac{1}{2}}$, this implies \eqref{eq:ComplexInterpolation}.
\end{proof}

This establishes the equivalence between \ref{Thm:EllEquiv:Ellbdy}--\ref{Thm:EllEquiv:Admiss2}. 
Since \ref{Thm:EllEquiv:Admiss3}$\implies$\ref{Thm:EllEquiv:Admiss4} is immediate, to demonstrate the remaining equivalences, it suffices to establish \ref{Thm:EllEquiv:Closed}$\implies$\ref{Thm:EllEquiv:Admiss3} and \ref{Thm:EllEquiv:Admiss4}$\implies$\ref{Thm:EllEquiv:Closed}.

\begin{proof}[Proof of \ref{Thm:EllEquiv:Admiss4}$\implies$\ref{Thm:EllEquiv:Closed} in Theorem~\ref{Thm:EllEquiv}]
Fix $r \in \R$ an admissible spectral cut as given by \ref{Thm:EllEquiv:Admiss4}.
We need to show that $B$ is closed in $\checkH(A)$ and that $B^\ad \subset \SobH{\frac12}(\Sigma;F)$.

Observe that $B + \chi^+(A_r)\SobH{\frac12}(\Sigma;E) = \chi^-(A_r)B \oplus \chi^+(A_r)\SobH{\frac12}(\Sigma;E)$ using Lemma~\ref{Lem:SumDirSum} with $Z = \SobH{\frac12}(\Sigma;E)$, $X = \chi^+(A_r)\SobH{\frac12}(\Sigma;E)$ and $Y = \chi^-(A_r)\SobH{\frac12}(\Sigma;E)$ and $W = B$. 
This also yields that $\chi^-(A_r)B$ is closed in $\SobH{\frac12}(\Sigma;E)$ since $B + \chi^+(A_r)\SobH{\frac12}(\Sigma;E)$ is closed by the fact that $(B,\chi^+(A_r)\SobH{\frac12}(\Sigma;E))$ is a Fredholm pair in $\SobH{\frac12}(\Sigma;E)$.

Next, note that $\chi^-(A_r)B \subset \chi^-(A_r)\SobH{\frac12}(\Sigma;E)=\chi^-(A_r)\checkH(A)\subset \checkH(A)$.
Therefore, $\chi^-(A_r)B$ is, in fact, a closed subspace in $\checkH(A)$.

Now, on setting $Z = \checkH(A)$, $X = \chi^+(A_r)\SobH{-\frac12}(\Sigma;E)$ and $Y = \chi^-(A_r)\SobH{\frac12}(\Sigma;E)$ and $W = B$, by Lemma~\ref{Lem:SumDirSum}, we have that $B + \chi^+(A_r)\SobH{-\frac12}(\Sigma;E) =  \chi^-(A_r)B \oplus \chi^+(A_r)\SobH{-\frac12}(\Sigma;E)$.
The latter is a closed subspace in $\checkH(A)$ since $\chi^-(A_r)B$ is closed in $\checkH(A)$.

Let $W_+ := \chi^+(A_r)\SobH{-\frac12}(\Sigma;E)\intersect B$. 
Since $B \subset \SobH{\frac12}(\Sigma;E)$, we have that $W_+ =\chi^+(A_r)\SobH{\frac12}(\Sigma;E)\intersect B$.
This is finite dimensional  since $(B,\chi^+(A_r)\SobH{\frac12}(\Sigma;E))$ is a Fredholm pair in $\SobH{\frac12}(\Sigma;E)$.
Since $W_+ \subset B$, we can write 
\begin{equation}
\label{Eq:BDecomp}
B = W_+ \oplus (B \intersect W_+^{\perp,\SobH{\frac12}}) .
\end{equation} 
Moreover, since $W_+ \subset \chi^+(A_r)\SobH{-\frac12}(\Sigma;E)$ and also $W_+ \subset \SobH{\frac12}(\Sigma;E)$, 
\begin{equation} 
\label{Eq:ChiDecomp}
\chi^+(A_r)\SobH{-\frac12}(\Sigma;E) = W_+ \oplus (\chi^+(A_r)\SobH{-\frac12}(\Sigma;E) \intersect W_+^{\perp,\SobH{-\frac12}}).
\end{equation}

Note also that 
$$B \intersect \chi^+(A_r)\SobH{-\frac12}(\Sigma;E) \intersect W_+^{\perp,\SobH{-\frac12}} = W_+ \intersect W_+^{\perp, \SobH{-\frac12}} = 0.$$
Combining this along with \eqref{Eq:BDecomp} and \eqref{Eq:ChiDecomp}, we obtain
$$ 
B + \chi^+(A_r)\SobH{-\frac12}(\Sigma;E) = B \oplus (\chi^+(A_r)\SobH{-\frac12}(\Sigma;E) \intersect W_+^{\perp,\SobH{-\frac12}}).
$$
From Corollary~4.13 in Chapter~IV, Section~4.2 in \cite{Kato}, we obtain that 
$((B^{\perp,\SobH{\frac12}})^{\perp,-\SobH{-\frac12}}, \chi^+(A_r)\SobH{-\frac12}(\Sigma;E))$ is a Fredholm pair in $\SobH{-\frac12}(\Sigma;E)$.
Let $\check{B} := (B^{\perp,\SobH{\frac12}})^{\perp,-\SobH{-\frac12}} \intersect \checkH(A)=(B^{\perp,\SobH{\frac12}})^{\perp,\checkH(A)}$.
Setting $X = (B^{\perp,\SobH{\frac12}})^{\perp,-\SobH{-\frac12}}$, $Y = \chi^+(A_r)\SobH{-\frac12}(\Sigma;E)$ and $W = \checkH(A)$ and invoking Lemma~\ref{Lem:FPSub}, we obtain  that $(\check{B}, \chi^+(A_r)\SobH{-\frac12}(\Sigma;E))$ is a Fredholm pair in $\checkH(A)$.

Set $Z = \checkH(A)$, $X = B$, $X' = \check{B}$, and $Y = \chi^+(A_r)\SobH{-\frac12}(\Sigma;E)\intersect W_+^{\perp,\SobH{-\frac12}}$.
By construction, $\check{B}$ is an annihilator and hence closed in $\checkH(A)$. 
It is readily seen that $B \subset \check{B}$.
Moreover, by the fact that   $(\check{B}, \chi^+(A_r)\SobH{-\frac12}(\Sigma;E))$ is a Fredholm pair in $\checkH(A)$, we obtain that $\check{B} \intersect \chi^+(A_r)\SobH{-\frac12}(\Sigma;E)$ is finite dimensional. 
Therefore, so is  $\check{B} \intersect \chi^+(A_r)\SobH{-\frac12}(\Sigma;E)\intersect W_+^{\perp,\SobH{-\frac12}}$. 
Thus we can apply Lemma~\ref{Lem:ClosedSub} to $Z$, $X$, $X'$, and $Y$ and conclude that $B \oplus (\check{B} \intersect \chi^+(A_r)\SobH{-\frac12}(\Sigma;E)\intersect W_+^{\perp,\SobH{-\frac12}})$ is closed in $\checkH(A)$.

To prove that $B$ itself is closed we show that $\check{B} \intersect \chi^+(A_r)\SobH{-\frac12}(\Sigma;E)\intersect W_+^{\perp,\SobH{-\frac12}} = 0$. 
For that, we show that $\check{B} \intersect \chi^+(A_r)\SobH{-\frac12}(\Sigma;E) = {B} \intersect \chi^+(A_r)\SobH{-\frac12}(\Sigma;E) = W_+$.
It is immediate that ${B} \intersect \chi^+(A_r)\SobH{-\frac12}(\Sigma;E) \subset \check{B} \intersect \chi^+(A_r)\chi^+(A_r)\SobH{-\frac12}(\Sigma;E)$. 
To obtain the opposite inclusion, it suffices to show that $\dim \big(\check{B} \intersect \chi^+(A_r)\SobH{-\frac12}(\Sigma;E)\big) \leq \dim\big({B} \intersect \chi^+(A_r)\SobH{-\frac12}(\Sigma;E)\big)$.
Letting $\hat{B} = B^{\perp, \hatH(A^\ast)}$, 
\begin{align*} 
\dim \big(\check{B} \intersect \chi^+(A_r)\SobH{-\frac12}(\Sigma;E)  \big)
&= 
\dim \big(\hatH(A^\ast)/(\hat{B} + \chi^-(A_r^\ast)\SobH{-\frac12}(\Sigma;E)) \big) \\
&= 
\dim \big(\chi^+(A_r^\ast)\SobH{\frac12}(\Sigma;E)/\chi^+(A_r^\ast)\hat{B}\big) \\
&\leq 
\dim \big(\chi^+(A_r^\ast)\SobH{\frac12}(\Sigma;E)/\chi^+(A_r^\ast)B^{\perp,\SobH{\frac12}} \big)\\
&= 
\dim \big(\SobH{\frac12}(\Sigma;E)/(B +\chi^-(A_r^\ast)\SobH{\frac12}(\Sigma;E))\big) \\
&=
\dim \big(B\intersect \chi^+(A_r)\SobH{-\frac12}(\Sigma;E)\big),
\end{align*}
where the first and last equalities follow from Theorem~4.8 in Section~4 of Chapter~4 in \cite{Kato}, the inequality from $B^{\perp, \SobH{\frac12}} \subset \hat{B}$, and the remaining equalities from Lemma~\ref{Lem:SumDirSum}.

To show that $B^\ad \subset \SobH{\frac12}(\Sigma;F)$ we show that $B^{\perp,\hatH(A^\dagstar)} = B^{\perp,\SobH{\frac12}}$.
For this, we first demonstrate that $(B, \chi^+(A_r)\SobH{-\frac12}(\Sigma;E))$ is a Fredholm pair in $\checkH(A)$ with the same index as $(B, \chi^+(A_r)\SobH{\frac12}(\Sigma;E))$ in $\SobH{\frac12}(\Sigma;E)$. 

It is clear that both $B$ and $\chi^+(A_r)\SobH{-\frac12}(\Sigma;E)$ are closed in $\checkH(A)$.
As we have noted previously, $B \intersect \chi^+(A_r)\SobH{-\frac12}(\Sigma;E) =  B \intersect \chi^+(A_r)\SobH{\frac12}(\Sigma;E)$ and therefore have the same finite dimension.
As before, setting $Z = \checkH(A)$, $X = \chi^+(A_r)\SobH{-\frac12}(\Sigma;E)$, $Y = \chi^-(A_r)\SobH{\frac12}(\Sigma;)E$, and $W = B$, by Lemma~\ref{Lem:SumDirSum}, we obtain $B + \chi^+(A_r) \SobH{-\frac12}(\Sigma;E)$ is closed in $\checkH(A)$ since $\chi^-(A_r)B$ is closed in $\checkH(A)$ and that 
$$\checkH(A)/(B + \chi^+(A_r)\SobH{-\frac12}(\Sigma;E)) \cong \chi^-(A_r)\SobH{\frac12}(\Sigma;E)/\chi^-(A_r)B.$$
Setting $Z = \SobH{\frac12}(\Sigma;E)$, $X = \chi^+(A_r)\SobH{\frac12}(\Sigma;E)$, $Y = \chi^-(A_r)\SobH{\frac12}(\Sigma;E)$ and $W = B$, again by Lemma~\ref{Lem:SumDirSum}, we obtain 
$$\SobH{\frac12}(\Sigma;E)/(B + \chi^+(A_r)\SobH{\frac12}(\Sigma;E)) \cong \chi^-(A_r)\SobH{\frac12}(\Sigma;E)/\chi^-(A_r)B.$$
On combining these two isomorphisms, we conclude that $(B, \chi^+(A_r)\SobH{-\frac12}(\Sigma;E))$ is a Fredholm pair in $\checkH(A)$ with the same index as $(B, \chi^+(A_r)\SobH{\frac12}(\Sigma;E))$.

Since  $\chi^-(A_r^\ast)\SobH{-\frac12}(\Sigma;E) = \chi^+(A_r)\SobH{-\frac12}(\Sigma;E)^{\perp, \hatH(A^\dagstar)}$, from  Corollary~4.13 in Chapter~IV, Section~4.2 in \cite{Kato}, we conclude that $(B^{\perp,\hatH(A)}, \chi^-(A_r^\ast)\SobH{-\frac12}(\Sigma;E))$ is a Fredholm pair in $\hatH(A^\dagstar)$. 
By repeating the argument that allowed us to obtain $(B, \chi^+(A_r)\SobH{-\frac12}(\Sigma;E))$ as a Fredholm pair in $\checkH(A)$ with the same index as $(B, \chi^+(A_r)\SobH{\frac12}(\Sigma;E))$, but with $B$ replaced with $B^{\perp,\SobH{\frac12}}$,  $\chi^+(A_r)\SobH{-\frac12}(\Sigma;E)$ with 
$\chi^-(A_r^\ast)\SobH{-\frac12}(\Sigma;E)$ and $\checkH(A)$ with $\hatH(A^\dagstar)$, allows us to conclude that $(B^{\perp,\SobH{\frac12}}, \chi^-(A_r^\ast)\SobH{-\frac12}(\Sigma;E))$ is a Fredholm pair in $\hatH(A^\dagstar)$ with the same index as $(B^{\perp,\SobH{\frac12}}, \chi^-(A_r^\ast)\SobH{\frac12}(\Sigma;E))$.
Next, using the fact that $\indx(B^{\perp,\SobH{\frac12}}, \chi^-(A_r^\ast)\SobH{-\frac12}(\Sigma;E)) = -\indx(B, \chi^+(A_r)\SobH{\frac12}(\Sigma;E))$, we obtain that 
$$ \indx(B^{\perp,\SobH{\frac12}}, \chi^-(A_r^\ast)\SobH{-\frac12}(\Sigma;E)) = \indx(B^{\perp,\hatH(A)}, \chi^-(A_r^\ast)\SobH{-\frac12}(\Sigma;E)).$$
This along with the fact that $B^{\perp,\SobH{\frac12}} \subset B^{\perp,\hatH(A)}$ yields $B^{\perp,\SobH{\frac12}} = B^{\perp,\hatH(A)}$ via Lemma~\ref{Lem:FPCon}. 
\end{proof} 

Lastly, we prove the remaining direction.

\begin{proof}[Proof of \ref{Thm:EllEquiv:Closed}$\&$\ref{Thm:EllEquiv:Admiss}$\implies$\ref{Thm:EllEquiv:Admiss3}]
Fix any admissible spectral cut $r \in \R$.
By \ref{Thm:EllEquiv:Closed}, $B$ is a closed subspace of $\SobH{\frac12}(\Sigma;E)$.
Using \ref{Thm:EllEquiv:Admiss},
\begin{align*}
B\cap \chi^+(A_r)\SobH{\frac12}(\Sigma;E)
=
(W_+\oplus \{v+gv : v\in V_-^{\frac12}\})\cap \chi^+(A_r)\SobH{\frac12}(\Sigma;E)
=
W_+
\end{align*}
and
\begin{align*}
B + \chi^+(A_r)\SobH{\frac12}(\Sigma;E)
=
V_-^{\frac12} \oplus \chi^+(A_r)\SobH{\frac12}(\Sigma;E) .
\end{align*}
Hence, $(\chi^+(A_r)\SobH{\frac12}(\Sigma;E),B)$ is a Fredholm pair with index $\dim W_+-\dim W_-$.

By Proposition~\ref{Prop:B0}, $B^{\perp,\hatH(A^\dagstar)}=\sym_0^* B^*$ and by \ref{Thm:EllEquiv:Closed} $B^*\subset \SobH{\frac12}(\Sigma;F)$, thus $B^{\perp,\SobH{\frac12}}=\sym_0^* B^*$.
Using \eqref{eq:B*}, we see that $B^{\perp,\SobH{\frac12}}=W_-^\ast \oplus \set{ u - g^\ast u: u \in (V_+^\ast)^{\frac{1}{2}}}$.
Arguing as above, we see that $(\chi^-(A_r^*)\SobH{\frac12}(\Sigma;E),B^{\perp,\SobH{\frac12}})$ is a Fredholm pair with index $\dim W_-^*-\dim W_+^*$.
Remark~\ref{Rem:*Spaces} shows that $\indx(\chi^+(A_r)\SobH{\frac12}(\Sigma;E),B)=-\indx(\chi^-(A_r^*)\SobH{\frac12}(\Sigma;E),B^{\perp,\SobH{\frac12}})$.
\end{proof}

\subsection{Boundary regularity}
\label{Sec:BdyReg}

Recall the notion of an $(s+\frac{1}{2})$-semiregular boundary condition from Definition~\ref{Def:SemiRegular}, by which we mean that $W_+ \subset \SobH{s}(\Sigma;E)$ and $g(V_-^s) \subset V_+^s$.
It is $(s+\frac{1}{2})$-regular if further $W_-^\ast \subset \SobH{s}(\Sigma;E)$ and $g^\ast( (V_+^\ast)^{s} ) \subset (V_-^\ast)^s$.

\begin{lem}
\label{Lem:RegEquiv} 
Let $B$ be an elliptic boundary condition.
For $s \geq \frac{1}{2}$, the following are equivalent:
\begin{enumerate}[(i)]
\item 
\label{Lem:RegEquiv2.1} 
There exist an admissible spectral cut $r$ such that $B$ is $(s+\frac{1}{2})$-semiregular w.r.t.\ $r$.
\item 
\label{Lem:RegEquiv2} 
For all admissible spectral cuts $r$ we have that $B$ is $(s+\frac{1}{2})$-semiregular w.r.t.\ $r$.
\item 
\label{Lem:RegEquiv3.1} 
There exists an admissible spectral cut $r$, such that whenever $u \in B$ with $\chi^{-}(A_r) u \in \SobH{s}(\Sigma;E)$ we have that $u \in \SobH{s}(\Sigma;E)$.
\item 
\label{Lem:RegEquiv3} 
For any admissible spectral cut $r$, whenever $u \in B$ with $\chi^{-}(A_r) u \in \SobH{s}(\Sigma;E)$ we have that $u \in \SobH{s}(\Sigma;E)$.
\end{enumerate}
\end{lem}
\begin{proof}
The implications \ref{Lem:RegEquiv2}$\Rightarrow$\ref{Lem:RegEquiv2.1} and \ref{Lem:RegEquiv3}$\Rightarrow$\ref{Lem:RegEquiv3.1} are clear.
The implications \ref{Lem:RegEquiv2.1}$\Rightarrow$\ref{Lem:RegEquiv3.1} and \ref{Lem:RegEquiv2}$\Rightarrow$\ref{Lem:RegEquiv3} are easy.
If $r_1$ and $r_2$ are admissible spectral cuts, then the difference of the spectral projectors $\chi^-(A_{r_1})$ and $\chi^-(A_{r_2})$ is smoothing. 
This shows \ref{Lem:RegEquiv3.1}$\Rightarrow$\ref{Lem:RegEquiv3}.

It remains to show \ref{Lem:RegEquiv3.1}$\Rightarrow$\ref{Lem:RegEquiv2.1}.
Since $W_+ \subset B$ and $\chi^{-}(A_r)w_+ = 0 \in \SobH{s}(\Sigma;E)$ for $w_+ \in W_+$, we have by  \ref{Lem:RegEquiv3.1} that $w_+ \in \SobH{s}(\Sigma;E)$, which shows that $W_+ \subset \SobH{s}(\Sigma;E)$.
Now, let $v \in V_-^s$.
Then $v+gv\in B$, hence $\chi^-(A_r)(v+gv)=v\in \SobH{s}(\Sigma;E)$.
By \ref{Lem:RegEquiv3.1}, $v+gv\in \SobH{s}(\Sigma;E)$ and therefore $gv\in\SobH{s}(\Sigma;E)$.
\end{proof}

\begin{proof}[Proof of Theorem~\ref{Thm:HBR}] 
Fix $u \in \dom(D_B)$ such that $D_{\max} u \in \SobH[loc]{k}(M;F)$.
Since $k < m$, $B$ is $(k+\frac{1}{2})$-semiregular.
\begin{enumerate}[a)]
\item W.l.o.g., assume that $\spt u \subset Z_{[0,\rho]} = {[0,\rho]} \times \Sigma$ for some $\rho < 1$ to be chosen sufficiently small so that $\dom((D_{0,r})_{\max};Z_{[0,\rho]}) = \dom(D_{\max};Z_{[0,\rho]})$. 
Through induction, we assume $u \in \SobH[loc]{k}(\Sigma;E)$ as well as $D_{0,r} u \in \SobH{k}(Z_{[0,\rho]}; F)$. 

\item 
Let $\eta:{[0,\rho]} \to [0,1]$ be a smooth cutoff such that $\eta = 1$ on $[0, \frac{\rho}{4}]$ and $\eta = 0$ on $[\frac{3\rho}{4}, \rho]$.
Set 
$$v(t) = \chi^{-}(A_r) \eta(t) u(t) = \eta(t) \chi^{-}(A_r) u(t)$$ 
on $Z_{{[0,\rho]}}$ and note that we have 
$$v(0) = \chi^{-}(A_r) (\eta(0) u\rest{\Sigma}) = \chi^{-}(A_r) u\rest{\Sigma}\quad \text{and} \quad v(\rho) = 0.$$

\item Since $v(\rho) = 0$, Proposition~\ref{Prop:RegInv} implies that
$$ (1 - S_{0,r}\sym_0^{-1} D_{0,r})v  = \exp({-t\modulus{A_r}}) \chi^{+}(A_r) v(0) = 0.$$ 
Hence, we obtain that $v \in \SobH{k+1}(Z_{[0,\rho]};E)$ because $S_{0,r}: \SobH{k}(Z_{[0,\rho]}]; E) \to  \SobH{k+1}(Z_{[0,\rho]}; E)$ by Lemma~\ref{Lem:SReg}.

\item Putting this together, we have that $v \in \dom((D_{0,r})_{\max};Z_{[0,\rho]}) \intersect \SobH{k+1}(Z_{[0,\rho]};E)$.
By the properties of the trace map, we obtain that $v(0) \in \SobH{k+\frac{1}{2}}(\Sigma;E)$. 
That is, $\chi^{-}(A_r) u\rest{\Sigma} = v(0) \in \SobH{k+{\frac{1}{2}}}(\Sigma;E)$.

\item Moreover, $u\rest{\Sigma} = u_0^- + g u_0^-  + w_+ \in V_- \oplus V_+ \oplus W_+$ since $B$ is an elliptic boundary condition.
Therefore, $\chi^{-}(A_r)u\rest{\Sigma} = u_0^-$ and we have that $u_0^- \in \SobH{k+\frac{1}{2}}(\Sigma;E)$.

\item Since $B$ is $(k+\frac{1}{2})$-semiregular, $gu_0^+ \in \SobH{k+\frac{1}{2}}(\Sigma;E)$.
Also, we have that $W_+ \subset \SobH{k+\frac{1}{2}}(\Sigma;E)$ and therefore,
$\chi^{+}(A_r) u\rest\Sigma \in \SobH{k+\frac{1}{2}}(\Sigma;E).$

\item Since we have by hypothesis that $D_{\max} u \in \SobH[loc]{k}(M;F)$ and that $\chi^{+}(A_r) u\rest{\Sigma} \in \SobH{k+\frac{1}{2}}(\Sigma;E)$, by Theorem~\ref{Thm:HighReg}, we obtain that $u \in \SobH[loc]{k+1}(M;E)$.
\qedhere
\end{enumerate} 
\end{proof}  

\subsection{Local and pseudo-local boundary conditions}
\label{Sec:LPL}

Recall from Definition~\ref{Def:LocalBC} that  a boundary condition $B \subset \SobH{\frac{1}{2}}(\Sigma;E)$ is local if  if there exists a subbundle $E' \subset E\rest{\Sigma}$ such that $ B = \SobH{\frac{1}{2}}(\Sigma;E')$. 
More generally, Definition~\ref{Def:PLocalBC} describes the notion of a pseudo-local boundary condition, where we write $B = \close{P \SobH{\frac{1}{2}}(\Sigma;E)}^{\checkH(A)}$ for $P$ a classical pseudo-differential projector of order zero. 

\begin{proof}[Proof of Theorem~\ref{Thm:PL}]
Fix an admissible spectral cut $r \in \R$. 
The equivalence of \ref{Thm:PL2}  and \ref{Thm:PL3} is obtained by  invoking Theorems 19.5.1 and 19.5.2 in \cite{H94}.

We show that \ref{Thm:PL3} and \ref{Thm:PL4} are equivalent.
By Lemma~\ref{Lem:ProjDiff}, we obtain that $\sym_{P - \chi^+(A_r)}(x,\xi) = \sym_{P}(x,\xi) - \sym_{\chi^+(A_r)}(x,\xi)$ is an isomorphism if and only if $\sym_{P}(x,\xi):\ker(\sym_{\chi^+(A_r)}(x,\xi)) \to \sym_{P}(x,\xi)(E_x)$ and $\sym_{P^\ast}(x,\xi):\ker(\sym_{\chi^+(A_r^\ast)}(x,\xi)) \to \sym_{P^\ast}(x,\xi)(E_x)$ are isomorphisms. 
The projectors $\chi^+(A_r)$ and $\chi^+(A_r^\ast)$ are classical pseudo-differential operators of order zero by Theorem~3.2 in \cite{Grubb}.
By Theorem~3.3 in \cite{Grubb}, their symbols $\sym_{\chi^+(A_r)}(x,\xi)$ and $\sym_{\chi^+(A_r^\ast)}(x,\xi)$ are given as contour integrals around the spectrum with positive real part of $\imath\sym_{A_r}(x,\xi)$ and $\imath\sym_{A_r^\ast}(x,\xi)$, respectively. 
Therefore, $\ker(\sym_{\chi^+(A_r)}(x,\xi))$  and $\ker(\sym_{\chi^+(A_r^\ast)}(x,\xi))$ are respectively the sums of the generalised eigenspaces of $\imath\sym_{A_r}(x,\xi)$ and $\imath\sym_{A_r^\ast}(x,\xi)$ with negative real part.
This proves the equivalence between \ref{Thm:PL3} and \ref{Thm:PL4}.

To show that \ref{Thm:PL3} implies \ref{Thm:PL1}, note that since $P$ is a pseudo-differential operator of order zero, we have that 
$$P:\SobH{\frac{1}{2}}(\Sigma;E) \to \SobH{\frac{1}{2}}(\Sigma;E)$$
boundedly. 
Since $P$ is a projection, $P\SobH{\frac{1}{2}}(\Sigma;E)$ is closed in $\SobH{\frac{1}{2}}(\Sigma;E)$.
To show that it defines a boundary condition, that is, that $B$ is a closed subset of $\checkH(A)$, we use Theorems 19.5.1 and 19.5.2 in \cite{H94} to obtain the existence of a pseudo-differential operator $R$ of order zero and a smoothing operator $S$ to write:
$$ R(P - \chi^+(A_r)) = 1 + S.$$
Now, note that for $u \in B$,  
$$ \norm{u}_{\SobH{\frac{1}{2}}} \lesssim \norm{(1 + S)u}_{\SobH{\frac{1}{2}}} + \norm{S u}_{\SobH{\frac{1}{2}}} \lesssim \norm{\chi^{-}(A_r)u}_{\SobH{\frac{1}{2}}} + \norm{\modulus{A_r}^{\frac{1}{2}} Su},$$
since $P - \chi^{+}(A_r) = 1 - \chi^{+}(A_r) = \chi^-(A_r)$ on $B$.
Moreover, since $S$ is smoothing,  
$$
\norm{\modulus{A_r}^{\frac{1}{2}} Su} \lesssim  \norm{u}_{\SobH{-\frac{1}{2}}} \leq \norm{u}_{\checkH(A)} .$$
Combining these estimates, we get that $\norm{u}_{\SobH{\frac{1}{2}}} \simeq \norm{u}_{\checkH(A)}$ and therefore, $B$ is a boundary condition.
In order to see $\sym_0^*B^*\subset \SobH{\frac12}(\Sigma;E)$, observe 
\begin{align*}
v \in \sym_0^\ast B^\ad &\iff v \in \hatH(A^\dagstar)\text{ and } \inprod{v, P u} = 0\ \forall u \in \SobH{\frac{1}{2}}(\Sigma;E) \\
	&\iff  v \in \hatH(A^\dagstar)\text{ and } \inprod{P^\ast v, u} = 0\ \forall u \in \SobH{\frac{1}{2}}(\Sigma;E) \\ 
	&\iff  v \in (1 - P^\ast)\SobH{-\frac{1}{2}}(\Sigma;E)\quad\text{and}\quad \chi^{+}(A_r^\ast)v \in  \SobH{\frac{1}{2}}(\Sigma;E). 
\end{align*}
Since $P^\ast - \chi^{+}(A_r^\ast)$ is also a classical elliptic pseudo-differential operator of order zero, we obtain $R'$ a pseudo-differential operator of order zero and $S'$ a smoothing operator such that $R'(P^\ast - \chi^+(A_r^\ast)) = 1 + S'$.
Then, 
$$ v = (1 + S')v - S'v = R'(P^\ast -\chi^+(A_r^\ast))v - S'v = -R' \chi^+(A_r^\ast)v - S' v $$
since $v \in \nul(P^\ast)$.
Noting that $\chi^+(A_r^\ast) v \in \SobH{\frac{1}{2}}(\Sigma;E)$ and $S'v \in \Ck{\infty}(\Sigma;E)$,
we obtain the conclusion. 
 
Finally, we prove that \ref{Thm:PL1} implies \ref{Thm:PL2}.
For this, write $B = P \SobH{\frac{1}{2}}(\Sigma;E) = W_+ \oplus \set{v + gv: v \in V_{-}^\frac{1}{2}}$.
Since $P$ is a pseudo-differential projector of order zero, we have that $P \Lp{2}(\Sigma;E) = W_+ \oplus \set{v + gv: v \in V_-}$.
Let 
$$
\inprod{\cdot,\cdot}_N := \inprod{\chi^-(A_r)\cdot, \chi^-(A_r)\cdot} + \inprod{\chi^+(A_r)\cdot, \chi^+(A_r)\cdot}
$$ 
and let $g^{\ast,N}:V_+ \to V_-$ the adjoint with respect to this scalar product.
Note that with respect to this scalar product $\chi^+(A_r) \Lp{2}(\Sigma;E) \perp \chi^-(A_r)\Lp{2}(\Sigma;E) $ and that $V_- \perp V_+$.
Then,  we obtain that $V_- \stackrel{\perp}{\oplus} V_+ = \graph(g) \stackrel{\perp}{\oplus} \graph(-g^{\ast,N}).$
Moreover, is it easy to see that  
$$ \begin{pmatrix} 1 & 0 \\ g & 0 \end{pmatrix} \begin{pmatrix} 1 & -g^{\ast,N} \\ g & 1 \end{pmatrix}^{-1}: V_- \oplus V_+ \to \graph (g)$$
is the orthogonal projection (with respect to $\inprod{\cdot,\cdot}_N$).
But then, $P(V_- \oplus V_+) = \graph(g)$ and so 
$$P = \begin{pmatrix} 1 & 0 \\ g & 0 \end{pmatrix} \begin{pmatrix} 1 & -g^{\ast,N} \\ g & 1 \end{pmatrix}^{-1}.$$ 
Now, 
$$
\chi^+(A_r)|_{V_-\oplus V_+}
=
\begin{pmatrix}0 & 0 \\ 0 & 1 \end{pmatrix}
=
\begin{pmatrix}0 & 0 \\ g & 1 \end{pmatrix}
\begin{pmatrix} 1 & -g^{\ast,N} \\ g & 1 \end{pmatrix}^{-1}
$$
and hence,
$$
(P - \chi^+(A_r))\rest{V_- \oplus V_+}
=
\begin{pmatrix}1 & 0 \\ 0 & -1 \end{pmatrix}
\begin{pmatrix} 1 & -g^{\ast,N} \\ g & 1 \end{pmatrix}^{-1}
: V_- \oplus V_+ \to V_- \oplus V_+
$$ 
is an isomorphism.
But $W_+$ and $W_-$ are finite dimensional subspaces and hence $P - \chi^+(A_r)$ is Fredholm.
\end{proof} 

\begin{proof}[Proof of Corollary~\ref{Cor:InfReg}]
Fix $s \geq 1/2$ and $r \in \R$ an admissible spectral cut.
Let $u \in B = P\SobH{\frac{1}{2}}(\Sigma;E)$ with $\chi^-(A_r) u \in \SobH{s}(\Sigma;E)$.
By Theorem~\ref{Thm:PL}, $P - \chi^+(A_r)$ is elliptic of order zero, and
$$ (P - \chi^+(A_r)) u = (1 - \chi^+(A_r))u = \chi^-(A_r)u \in \SobH{s}(\Sigma;E).$$
By ellipticity of $(P - \chi^+(A_r))$, we have that $u \in \SobH{s}(\Sigma;E)$.
\end{proof}

\begin{proof}[Proof of Corollary~\ref{Cor:LS}]
Let $P:E_x \to E'_x$  and $\tilde{P}:F_x \to (F')^\perp_x = (\sym_0^{-1})^*(E'_x)^\perp$  be orthogonal projections.
We verify condition \ref{ProjDiff3} in Lemma~\ref{Lem:ProjDiff} with a choice of $P$ as given here and $Q = \sym_{\chi^+(A_r)}(x,\xi)$.

By the Lopatinsky-Schapiro condition for $B$, for every $x \in M$, $\xi \in T_x^*M\setminus\{0\}$ and each $e \in E_x'$, there exists a unique $u:[0, \infty) \to E_x'$ such that 
\begin{equation} 
 P u(0) = e \text{ with } (\partial_t + \imath \sym_A(x,\xi))u(t) = 0 \text{ and } \lim_{t \to \infty} u(t) = 0.
 \label{eq:LopSch}
\end{equation}
The solution to \eqref{eq:LopSch} is given by $u(t)=\exp(-t\imath \sym_A(x,\xi))u(0)$.
On writing $\imath \sym_A(x,\xi)$ in Jordan normal form, one easily sees that the condition $\lim_{t \to \infty} u(t) = 0$ is equivalent to $u(0)\in  \chi^+( {\imath \sym_A(x,\xi)})(E_x)$, which is the sum of generalised eigenspaces of $\imath \sym_A(x,\xi)$ to eigenvalues with positive real part.
Thus, $P$ restricts to an isomorphism $\sym_{\chi^+(A_r)}(x,\xi)(E_x)\to E'_x$.

Given a $\xi$, applying the condition instead to $-\xi$ and using the $\R$-linearity of $\eta \mapsto \imath \sym_A(x,\eta)$, 
\begin{align*}
\chi^+( {\imath \sym_A(x,-\xi)})(E_x) 
&= 
\chi^+( {-\imath \sym_A(x,\xi)})(E_x) \\
&= \chi^-( {\imath \sym_A(x,\xi)})(E_x) \\
&= \sym_{\chi^-(A_r)}(x,\xi)(E_x).
\end{align*}
That is, $P$ also restricts to an isomorphism $\ker(\sym_{\chi^+(A_r)}(x,\xi) \to E'_x=PE_x$.

Next, we use the  Lopatinsky-Schapiro condition for $B^*$. 
This exactly means that for $x \in \Sigma$ and $\xi \neq 0$ fixed, and $\tilde{f} \in (F'_x)^\perp$,  there is a unique $\tilde{v}: [0,\infty) \to (F'_x)^\perp$ so that $\tilde{P}\tilde{v}(0) = \tilde{f}$, $(\partial_t + \imath \sym_{\tilde{A}}(x,\xi))\tilde{v}(t) = 0$ and $\lim_{t \to \infty} \tilde{v}(t) = 0$.
Recall that $\sym_{\tilde{A}}(x,\xi) = (\sym_0^{-1})^* \sym_{A^*}(x,\xi) \sym_0^*$.
Then, on setting $v^\perp := \sym_0^*\tilde{v}: [0,\infty) \to (E')^\perp_x$, from the invertibility of $\sym_0$, we obtain 
\begin{equation} 
\label{Eq:SL2}
(\partial_t + \imath \sym_{A^*}(x,\xi))v^\perp(t) = 0,\quad\text{and}\quad \lim_{t \to \infty} v^\perp(t) = 0.
\end{equation} 
Since $(F'_x)^\perp = (\sym_0^{-1})^* (E'_x)^\perp$, given any $e^\perp \in (E'_x)^\perp$ and setting $\tilde{f} := (\sym_0^{-1})^* e^\perp$, we obtain a unique $v^\perp:[0,\infty)$ with $(1-P)v^\perp = e^\perp$ satisfying \eqref{Eq:SL2}.
That is precisely that $(1-P):\sym_{\chi^+(A_r)}(x,\xi)(E_x) \to \ker P$ is an isomorphism. 

Therefore, by Lemma~\ref{Lem:ProjDiff}~\ref{ProjDiff3}, we obtain that $P - \sym_{\chi^+(A_r)}(x,\xi)$ is an isomorphism.
By Theorem~\ref{Thm:PL}~\ref{Thm:PL3},  $B$ and $B^*$ are elliptic boundary conditions for $D$ and $D^\dagger$, respectively.

To see that $B^\perp$ and $(B^*)^\perp$ are also elliptic boundary conditions,  fix $x \in M$ and $\xi \neq 0$.
By applying what we have just shown to $-\xi$, we deduce that $P - \sym_{\chi^+(A_r)}(x,-\xi)$ is an isomorphism.
Then, 
\begin{align*}
-(P - \sym_{\chi^+(A_r)}(x,-\xi)) &= (1 - P) - (1 - \sym_{\chi^+(A_r)}(x,-\xi))  \\
	&= (1 - P) -  \sym_{\chi^-(A_r)}(x,-\xi)  \\
	&=  (1 - P) -  \sym_{\chi^+(A_r)}(x,\xi).
\end{align*} 
This shows that $(1 - P) -  \sym_{\chi^+(A_r)}(x,\xi)$ is an isomorphism and hence, by Theorem~\ref{Thm:PL}~\ref{Thm:PL3}, we obtain that $B^\perp$ and $(B^*)^\perp$ are also elliptic boundary conditions for $D$ and $D^\dagger$, respectively.

These boundary conditions are local, hence pseudo-local and thus $\infty$-regular by Corollary~\ref{Cor:InfReg}.  
\end{proof}

\begin{proof}[Proof of \ref{Thm:Fredholm}]
We first note that Theorem 8.5 in \cite{BB12}, which states that $D$ is coercive if and only if $D_B$ has finite-dimensional kernel and closed range, holds in our setting. 
Its proof only uses the fact that   $\dom(D_B)\subset \SobH[loc]{1}(M;E)$ and $\dom(D^\dagger_{B^\ast}) \subset \SobH[loc]{1}(M;F)$ since $B$ is an elliptic boundary condition.  
Since we assume that both $D$ and $D^\dagger$ are coercive at infinity, we obtain that both operators have finite-dimensional kernels and that their ranges are closed.
Coupling this with the simple fact that $\Lp{2}(M;E) = \ker(D_B) \oplus \ran(D^\dagger_{B^\ast})$ and $\Lp{2}(M;F) = \ker(D^\dagger_{B^\ast}) \oplus \ran(D_{B})$, we obtain 
$$ \faktor{\Lp{2}(M;F)}{\ran(D_B)} \cong \ker(D^\dagger_{B^\ast}).$$
From this,  \ref{Thm:Fredholm:1} follows. 

The statement \ref{Thm:Fredholm:2} follows simply on invoking Proposition A.1 in \cite{BB12}.

For the proof of \ref{Thm:Fredholm:3}, fix two complements $C$ and $C'$ of $B$ and $B'$ with $C\subset C'$. 
These complements exist since $\checkH(A)$ is a Hilbert space.
Let $\check{P}$ and $\check{P}'$ be the projectors with kernels $B$ and $B'$ respectively and ranges $C$ and $C'$.
Define $Qu := \check{P}u\rest{\Sigma}$ and $Q'u := \check{P}u'\rest{\Sigma}$.
This yields the following commutative diagram
$$
  \begin{tikzcd}
	&  \Lp{2}(M;E) \oplus C \\  
    \dom(D_{\max}) \arrow{ur}{(D,Q)} \arrow[swap]{dr}{(D,Q')}  &  \\
     & \Lp{2}(M;E) \oplus C'\arrow[hookleftarrow,swap]{uu}{\id \oplus \iota}
  \end{tikzcd}$$
where $\iota: C \embed C'$ is the inclusion map.
By  \ref{Thm:Fredholm:2}, the maps $(D,Q)$ and $(D,Q')$ are Fredholm and therefore, so is $\id \oplus \iota$.
Then, 
$$ \dim \faktor{B}{B'} = \dim \faktor{C'}{C} = - \Ind (\iota) = - \Ind(\id \oplus \iota) = \Ind(D,Q) - \Ind(D,Q'),$$
and the desired index formula follows.
\end{proof}

\appendix
\section{Auxiliary functional analytic facts}
\label{Sec:Appendix}

By $P_{A,B}$, denote the projector with range $A$ and kernel $B$.
\begin{lem}
\label{Lem:SubIsom}
Let $\cH$ be a Hilbert space such that $\cH = \cH_1 \oplus \cH_2$ where $\cH_1,\ \cH_2 \subset \cH$ are closed subspaces.
Let $\cH_1'$ be another closed subspace such that $\cH = \cH_1' \oplus \cH_2$.
Then, $P_{\cH_1, \cH_2}\rest{\cH_1'}: \cH_1' \to \cH_1$ is an isomorphism with bounded inverse.
\end{lem} 
\begin{proof}
Let $u \in \cH_1$ and note that 
$$ u = P_{\cH_1, \cH_2} u = P_{\cH_1,\cH_2} (P_{\cH_1', \cH_2} u + P_{\cH_2, \cH_1'}) u = P_{\cH_1,\cH_2} P_{\cH_1', \cH_2} u, $$
because $P_{\cH_1,\cH_2} \circ P_{\cH_2, \cH_1'} = 0$.
This shows surjectivity. 

Now, let $u \in \cH_1'$ and suppose that $P_{\cH_1, \cH_2} u = 0$. 
But then, $u \in \nul(P_{\cH_1, \cH_2}) = \cH_2$ and so, $u \in \cH_1' \intersect \cH_2 = \set{0}$, which shows that the map is injective.

Since it is a bounded map on all of $\cH_1$, it is bounded, and by the closedness of $\cH_1$ and the open mapping theorem, we can conclude that it has a bounded inverse.
\end{proof}

\begin{lem}
\label{Lem:Abs1} 
Let $(X,\norm{\cdot}_X)$ be a Banach space and $Y \subset X$ a subspace (not necessarily closed in $\norm{\cdot}_X$) for which $(Y,\norm{\cdot}_Y)$ is a Banach space satisfying $\norm{u}_X \lesssim \norm{u}_Y$ for all $u \in Y$.
Then, if $Z \subset X$ is a closed subspace with respect to $\norm{\cdot}_X$, then $Z \intersect Y$ is closed $Y$ with respect to $\norm{\cdot}_Y$.
\end{lem}
\begin{proof}
Let $u_n \in Z \intersect Y$ be Cauchy in $\norm{\cdot}_Y$.
Then, there exists $y \in Y$ such that $u_n \to y$ in $\norm{\cdot}_Y$. 
Moreover, since $u_n \in Y$, we have that $\norm{u_n - u_m}_X \lesssim \norm{u_n - u_m}_Y$ so $u_n$ is Cauchy in $X$, and therefore, by the closedness of $Z$, there exists $z \in Z$ such that $u_n \to z$ in $\norm{\cdot}_X$.
Now, $y,\ u_n \in Y$ for all $n$ and therefore, $\norm{u_n - y}_X \lesssim \norm{u_n - y}_Y$ by hypothesis, which proves that $u_n \to y$ in $X$. 
Therefore, $y = z$ which completes the proof.
\end{proof} 

\begin{lem}
\label{Lem:Abs2}
Let $(X, \norm{\cdot}_1)$ and $(X,\norm{\cdot}_2)$ be Banach spaces for which $\norm{\cdot}_1 \lesssim \norm{\cdot}_2$. 
Then, $\norm{\cdot}_1 \simeq \norm{\cdot}_2$.
\end{lem}
\begin{proof}
The hypothesis implies that $\mathrm{id}: (X, \norm{\cdot}_1) \embed (X,\norm{\cdot}_2)$ is continuous. 
By the open mapping theorem, it is an isomorphism.
\end{proof} 

\begin{lem}
\label{Lem:SumDirSum}
Let $Z$ be a Banach space and $X, Y \subset Z$ be complementary closed subspaces such that $Z = X \oplus Y$. 
For any subspace $W \subset Z$, we have 
$$X + W = X \oplus P_{Y,X} W.$$ 
Moreover $P_{Y,X}W$ is closed if and only if $X + W$ is closed. 
In this case,
$$Z/(X + W) \cong Y/P_{Y,X}W.$$
\end{lem}
\begin{proof}
Let $x + w \in X + W$.
Then, 
$$ x + w = P_{X,Y}(x + w) + P_{Y,X}(x + w) = (x + P_{X,Y}w)  + P_{Y,X}w \in X \oplus P_{Y,X}W.$$ 
Now, for $x + P_{Y,X}w \in X \oplus P_{Y,X}W$, we have that
$$ x + P_{Y,X}w = (x - P_{X,Y}w) + P_{X,Y}w + P_{Y,X}w = (x - P_{X,Y}w)  + w \in X + W.$$

Now assume that $P_{Y,X}W$ is closed.
We need to conclude that $X + W$ is closed which is not as obvious as one might think because the direct sum of closed subspaces need not be closed in general.
So, let $x_n+w_n\in X+W$ such that $x_n+w_n \to z$.
Hence, $P_{Y,X}w_n=P_{Y,X}(x_n+w_n)\to P_{Y,X}z$ and since $P_{Y,X}W$ is closed we have $P_{Y,X}z\in P_{Y,X}W$.
Now, $z=P_{X,Y}z + P_{Y,X}z \in X \oplus P_{Y,X}W = X+ W$.

For the converse statement, let $P_{Y,X}w_n \in P_{Y,X}W$ be a convergent sequence. 
Since $P_{Y,X}W \subset X + W$, there is some $x + w \in X + W$ such that $P_{Y,X}w_n \to x + w$.
However, $P_{Y,X}w_n = P_{Y,X}^2w_n \to P_{Y,X}(x + w) = P_{Y,X}w$.
This shows that  $P_{Y,X}W$ is closed.

The last assertion is obtained simply by noting
\begin{equation*} 
Z/(X + W) = (X \oplus Y)/(X \oplus P_{Y,X}W) \cong Y/P_{Y,X}W.\qedhere
\end{equation*}
\end{proof} 

\begin{lem}
\label{Lem:ClosedSub}
Let $Z$ be a Banach space and $Y$ a closed subspace.
Suppose that $X$ is a complementary subspace with $X \oplus Y$ is closed in $Z$.
Moreover, suppose there exists a subspace $X'$ closed in $Z$ with $X'+Y$ closed such that $X \subset X'$ and $X' \intersect Y$ is finite dimensional. 
Then, $X \oplus (X' \intersect Y) $ is closed.
\end{lem}
\begin{proof}
Let $U := X' \intersect Y$. 
Since this is finite dimensional, there exists a closed complementary subspace $V$ so that $Z = U \oplus V$.
Equivalently, we have a bounded projector $P_{U,V}: Z \to U$ with kernel $V$.
Now, since $U \subset X'$ and $U \subset Y$, it is readily checked that 
$$ 
X' 
= 
U \oplus (X' \intersect V)\ \text{and}\ Y = U \oplus (Y \intersect V).
$$
Consequently,  $X' + Y = X' \oplus (Y \intersect V)$. 
Since $X \intersect U \subset X \intersect Y  = 0$, we have that $X \oplus Y = X \oplus U  \oplus (Y \intersect V)$.

It is readily checked that the natural  map $\Phi: X \oplus Y \to X \oplus Y/(Y \intersect V) = X \oplus U \oplus (Y \intersect V)/(Y \intersect V)$ restricts to an (algebraic) vector space isomorphism from $X\oplus U$ onto $X \oplus Y/(Y \intersect V)$.
Moreover, $ \norm{\Phi(x)} = \inf_{v \in Y \intersect V} \norm{x + v} \leq \norm{x}$
which shows that it is continuous.
Since $X'+Y$ is closed we have the bounded projector $P_{Y \intersect V, X'}$ and we obtain 
$$ 
\norm{x} \leq \norm{x + v}  + \norm{v}  = \norm{ x  + v} + \norm{P_{Y \intersect V,X'}(x + v)} \leq (1 + \norm{P_{Y \intersect V,X'}})  \norm{x + v}$$
so that $\norm{x} \leq (1 + \norm{P_{Y \intersect V,X'}}) \norm{\Phi(x)}$. 
This shows that the inverse $\Phi^{-1}$ is also bounded.
Therefore, $X \oplus U = X \oplus (X' \intersect Y)$ is closed since $X \oplus Y/(Y \intersect V)$ is complete.
\end{proof}

\begin{lem}
\label{Lem:FPSub}
Let $(X,Y)$ be a Fredholm pair in a Banach space $Z$.
Suppose that $W \subset Z$ is a subspace (not necessarily closed) such that $(W,\norm{\cdot}_{W})$ is a Banach space satisfying $\norm{w}_{Z} \lesssim \norm{w}_{W}$ for all $w \in W$.
Then, if $Y \subset W$, we have that $(X\intersect W, Y)$ is a Fredholm pair in $W$.
\end{lem}
\begin{proof}
First, we note that $Y$ and $X \intersect W$ are closed in $W$ by Lemma~\ref{Lem:Abs1}. 
Next, note that $X \intersect W \intersect Y \subset X \intersect Y$, and since $(X,Y)$ is a Fredholm pair, the latter space is finite dimensional and so is $X \intersect W \intersect Y$.

It remains to prove that $X \intersect W + Y$ is closed in $W$ and that $W/(X \intersect W + Y)$ is finite dimensional.
For that, we first prove 
\begin{equation}
\label{Eq:FPSub1}
X \intersect W + Y = (X + Y) \intersect W. 
\end{equation}
To prove the containment ``$\subset$'', fix $x + y \in X \intersect W + Y$ with $x \in X \intersect W$ and $y \in Y$.
Since $Y \subset W$, it follows that $x + y \in (X + Y) \intersect W$.

For the reverse containment, let $x + y \in (X + Y) \intersect W$. 
But $x = x + y - y$ and since $Y \subset W$, we have that $x \in W$.
Therefore, $x + y \in X \intersect W + Y$.

The formula \eqref{Eq:FPSub1} immediately yields that $X \intersect W + Y$ is closed in $W$ by noting that $X + Y$ is closed in $Z$ and on invoking Lemma~\ref{Lem:Abs1}. 

To prove that $W/(X \intersect W + Y)$ is finite dimensional, consider the natural map $\Phi:W/(X \intersect W + Y) \to Z/(X + Y)$.
This map is well-defined because $X \intersect W + Y \subset X + Y$. 
To complete the proof, it suffices to show that $\phi$ is an injection since we know that $Z/(X +Y)$ is finite dimensional by the Fredholm pair assumption on $(X,Y)$.
So, suppose that $\Phi([w]) = 0$.
That is $w \in W$ and $w \in X + Y$.
By \eqref{Eq:FPSub1}, we have that $w \in (X \intersect W) + Y$ and therefore, $[w] = 0$.
\end{proof}

\begin{lem}
\label{Lem:FPCon}
Let $(X_1,Y)$ and $(X_2,Y)$ be two Fredholm pairs in a Banach space $Z$ with $X_1 \subset X_2$ and $\indx(X_1,Y) = \indx(X_2,Y)$.
Then, $X_1 = X_2$.
\end{lem}
\begin{proof}
First note that $X_1  +Y \subset X_2 + Y$ and therefore, 
$$ \dim (Z/(X_1 + Y)) \geq \dim(Z/(X_2 + Y)).$$
Combining this with  $\indx(X_1,Y) = \indx(X_2,Y)$, we obtain that $\dim(X_1 \intersect Y) \geq \dim(X_2 \intersect Y)$. 
But $X_1 \intersect Y \subset X_2 \intersect Y$ and so we have that $X_1 \intersect Y = X_2 \intersect Y$.
Moreover, we get that $\dim (Z/(X_1 + Y)) = \dim(Z/(X_2 + Y))$ and so the natural map $z + X_1 + Y \mapsto z +X_2 +Y$ is an isomorphism.
Therefore, $X_1 + Y = X_2 + Y=:Z_0$.

Set $W := X_1 \intersect Y = X_2 \intersect Y$.
Since $W$ is finite dimensional, we obtain a complementary closed subspace $W^c$ such that $Z_0 = W \oplus W^c$.
Moreover, since $W \subset Y$, $Y = W \oplus (Y \intersect W^c)$.
Write 
$$Z_0 = X_i + (Y \intersect W^c),$$
since $W \subset X_i$.
In fact, 
$$X_i \intersect (Y \intersect W^c) = X_i \intersect Y \intersect W^c = W \intersect W^c = 0.$$
Therefore, $Z_0 = X_1 \oplus (Y \intersect W^c) = X_2 \oplus (Y \intersect W^c)$ and armed with the fact that $X_1 \subset X_2$, we obtain that $X_1 = X_2$.
\end{proof} 

\begin{proposition}\label{Prop:HormPeet}
Let $X$ and $Y$ be Banach spaces and $L:X \to Y$ be a bounded linear map.
Then the following are equivalent:
\begin{enumerate}[(i)]
\item \label{HP1}
The operator $L$ has finite-dimensional kernel and closed image.
\item\label{HP2}
There is a Banach space $Z$, a compact linear map $K:X\to Z$,
and a constant $C$ such that
\[
  \|x\|_X \leq C\cdot \left(\|Kx\|_Z + \|Lx\|_Y \right) ,
\]
for all $x \in X$. 
In particular, $\ker K\cap \ker L=\{0\}$.
\item\label{HP3}
Every bounded sequence $(x_n) $ in $X$ such that $(Lx_n)$ converges in $Y$
has a convergent subsequence in $X$.
\end{enumerate}
Moreover, these equivalent conditions imply
\begin{enumerate}[(iv)]
\item 
For any Banach space $Z$ and compact linear map $K:X \to Z$
such that $\ker K\cap\ker L=\{0\}$, there is a constant $C$ such that
\[  \| x\|_X \le C \big(\|Kx\|_Z + \|Lx\|_Y \big)  \]
for all $x \in X$. 
\end{enumerate}
\end{proposition}

For proof see Proposition~A3 in \cite{BB12}.

\begin{lem}
\label{Lem:ProjDiff}
Let $E$ be a finite dimensional Euclidean vector space and let $P,Q:E\to E$ be projectors, i.e.\ $P^2=P$ and $Q^2=Q$.
Then the following are equivalent:
\begin{enumerate}[(i)]
\item\label{ProjDiff1}
$P-Q:E\to E$ is an isomorphism.
\item\label{ProjDiff2}
$P|_{\ker Q}:\ker Q \to PE$ and $P^*|_{\ker Q^*}:\ker Q^* \to P^*E$ are isomorphisms.
\item\label{ProjDiff3}
$P|_{\ker Q}: \ker Q \to PE$ and $(1 - P): QE \to \ker P$ are  isomorphisms.
\end{enumerate}
\end{lem}

Note that the two conditions in \ref{ProjDiff2} of Lemma~\ref{Lem:ProjDiff} are independent conditions.
For example, let $E=\R^2$ with the usual inner product, let $P=\begin{pmatrix} 1 & 1 \\ 0 & 0 \end{pmatrix}$ and $Q=\begin{pmatrix} 1 & 0 \\ 0 & 0 \end{pmatrix}$.
Then, $P^*=\begin{pmatrix} 1 & 0 \\ 1 & 0 \end{pmatrix}$ and $Q^*=Q$.
The kernel of $Q=Q^*$ is spanned by $e_2=(0,1)$.
Now, $Pe_2\neq0$ while $P^*e_2=0$.
Thus, $P|_{\ker Q}:\ker Q \to PE$ is an isomorphism while $P^*|_{\ker Q^*}:\ker Q^* \to P^*E$ is not.

Indeed, in this example, $P-Q=\begin{pmatrix} 0 & 1 \\ 0 & 0 \end{pmatrix}$ is not an isomorphism.

\begin{proof}[Proof of Lemma~\ref{Lem:ProjDiff}]
Splitting $E$ as $E=QE\oplus \ker Q$ in the domain and as $E=\ker P\oplus PE$ in the target we can write endomorphisms of $E$ as $2\times 2$-matrices:
\begin{align*}
 P =
 \begin{pmatrix}
 0 & \rvline & 0 \\
 \hline
 P|_{QE} & \rvline & P|_{\ker Q}
 \end{pmatrix},
 \quad
 Q = 
 \begin{pmatrix}
 (1-P)|_{QE} & \rvline & 0 \\
 \hline
 P|_{QE}  & \rvline & 0
 \end{pmatrix}
\end{align*}
and hence
\begin{equation*}
P-Q =
\begin{pmatrix}
(P-1)|_{QE} & \rvline & 0 \\
\hline
0  & \rvline & P|_{\ker Q}
\end{pmatrix}.
\end{equation*}
This shows that $P-Q$ is an isomorphism if and only if 
$$
P|_{\ker Q}: \ker Q \to PE
\mbox{ and }
(P-1)|_{QE}: QE \to \ker P
$$
are isomorphisms and hence proves the equivalence between \ref{ProjDiff1} and \ref{ProjDiff3}.
Moreover, this proves the implication \ref{ProjDiff1}$\Rightarrow$\ref{ProjDiff2} since $P^*-Q^*$ is also an isomorphism if $P-Q$ is one.

Conversely, assume \ref{ProjDiff2}.
We have to show that $(P-1)|_{QE}: QE \to \ker P$ is an isomorphism.
Since we know that $\dim \ker Q = \dim PE$, we also know that $\dim QE = \dim\ker P$.
Thus it suffices to show that $(P-1)|_{QE}: QE \to \ker P$ is injective.

So, let $(P-1)x=0$ where $x=Qx$.
Then $Px=Qx=x$.
Since $x\in PE$ we find a $y\in\ker Q$ such that $x=Py$.
Then we have for all $z\in E$:
$$
\langle y,P^*z\rangle
=
\langle Py,z\rangle
=
\langle x,z\rangle
=
\langle Qx,z\rangle
=
\langle x,Q^*z\rangle
=
\langle Py,Q^*z\rangle
=
\langle y,P^*Q^*z\rangle
$$
and thus
$$
\langle y,P^*(1-Q^*)z\rangle
=
0.
$$
This, together with the assumption that $P^*|_{\ker Q^*}:\ker Q^* \to P^*E$ is surjective yields that $y$ is perpendicular to $P^*E$, and hence $y\in\ker P$.
Thus $x=Py=0$.
\end{proof}


\setlength{\parskip}{0pt}

\end{document}